\documentclass[microtype]{gtpart}     
\title{Constrained knots in lens spaces}

\author{Fan Ye}
\givenname{Fan}
\surname{Ye}
\address{Department of Pure Mathematics and Mathematical Statistics, University of Cambridge\\Cambridge, UK}
\email{fy260@cam.ac.uk}


\newtheorem{theorem}{Theorem}[section]    
\newtheorem{lemma}[theorem]{Lemma}          
\newtheorem{corollary}[theorem]{Corollary}
\newtheorem{proposition}[theorem]{Proposition}
\theoremstyle{definition}
\newtheorem{definition}[theorem]{Definition}    
\newtheorem{remark}[theorem]{Remark}
\newtheorem{conjecture}[theorem]{Conjecture}

\usepackage{amsmath}
\usepackage{amssymb}
\usepackage{mathabx}
\usepackage{amsbsy}
\usepackage{graphicx}
\usepackage{appendix}
\usepackage{float}
\usepackage{subfigure}
\usepackage{mathtools}
\usepackage{enumerate}

\newcommand{\al}{\alpha}
\newcommand{\be}{\beta}
\newcommand{\ga}{\gamma}

\newcommand{\p}{\prime}

\newcommand{\aand}{,~{\rm and}~}

\begin{document}

\begin{abstract}    
In this paper, we study a special family of $(1,1)$ knots called constrained knots, which includes 2-bridge knots in the 3-sphere $S^3$ and simple knots in lens spaces. Constrained knots are parameterized by five integers and characterized by the distribution of spin$^c$ structures in the corresponding $(1,1)$ diagrams. The knot Floer homology $\widehat{HFK}$ of a constrained knot is thin. We obtain a complete classification of constrained knots based on the calculation of $\widehat{HFK}$ and presentations of knot groups. We provide many examples of constrained knots constructed from surgeries on links in $S^3$, which are related to 2-bridge knots and 1-bridge braids. We also show many examples of constrained knots whose knot complements are orientable hyperbolic 1-cusped manifolds with simple ideal triangulations.
\end{abstract}
\maketitle

\section{Introduction}
The main object studied in this paper is a special family of knots in lens spaces called constrained knots. Every knot in a closed 3-manifold can be represented by a doubly-pointed Heegaard diagram $(\Sigma,\al,\be,z,w)$ \cite[Section 2]{Ozsvath2004}, where $\Sigma$ is a closed surface, $\alpha=\{\al_1,\dots,\al_g\}$ and $\beta=\{\be_1,\dots,\be_g\}$ are two collections of $g=g(\Sigma)$ simple closed curves on $\Sigma$, and $z$ and $w$ are two basepoints on $\Sigma-(\al\cup \be)$. Conversely, any doubly-pointed Heegaard diagram defines a knot. Explicitly, the knot is the union of an arc $a$ connecting $z$ to $w$ on $\Sigma-\al$, pushed slightly into the $\al$-handlebody, and an arc $b$ connecting $w$ to $z$ on $\Sigma-\be$, pushed slightly into the $\be$-handlebody.

Let $T^2$ be the torus obtained by the quotient map $\mathbb{R}^2\rightarrow T^2$ that identifies $(x,y)$ with $(x+m,y+n)$ for $m,n\in\mathbb{Z}$. Suppose $p,q$ are integers satisfying $p>0$ and $\gcd(p,q)=1$. Let $\alpha_0$ and $\beta_0$ be two simple closed curves on $T^2$ obtained from two straight lines in $\mathbb{R}^2$ of slopes 0 and $p/q$. Then $(T^2,\alpha_0,\beta_0)$ is called the \textbf{standard diagram} of a lens space $L(p,q)$. Let $\alpha_1=\alpha_0$ and let $\beta_1$ be a simple closed curve on $T^2$ such that it is disjoint from $\beta_0$ and $[\beta_1]=[\beta_0]\in H_1(T^2;\mathbb{Z})$. Then $(T^2,\alpha_1,\beta_1)$ is also a Heegaard diagram of $L(p,q)$. Let $z$ and $w$ be two basepoints in $T^2-\alpha_0\cup\beta_0\cup\beta_1$.

The knot defined by the doubly-pointed diagram $(T^2,\alpha_1,\beta_1,z,w)$ is called a \textbf{constrained knot} and the diagram is called the \textbf{standard diagram} of the constrained knot. We will show that constrained knots are parameterized by five integers, which will be denoted by $C(p,q,l,u,v)$. For some technical reason, the knot $C(p,q,l,u,v)$ is in $L(p,q^\p)$, where $qq^\p\equiv 1 \pmod p$. An example is shown in Figure \ref{l1A}, where $(T^2,\alpha_0,\beta_0)$ is the standard diagram of $L(5,2)$ and $(T^2,\alpha_1,\beta_1,z,w)$ defines $C(5,3,2,3,1)$.

\begin{figure}[htbp]
\centering
\begin{minipage}[ht]{0.45\textwidth}
\centering
\includegraphics[width=5.5cm]{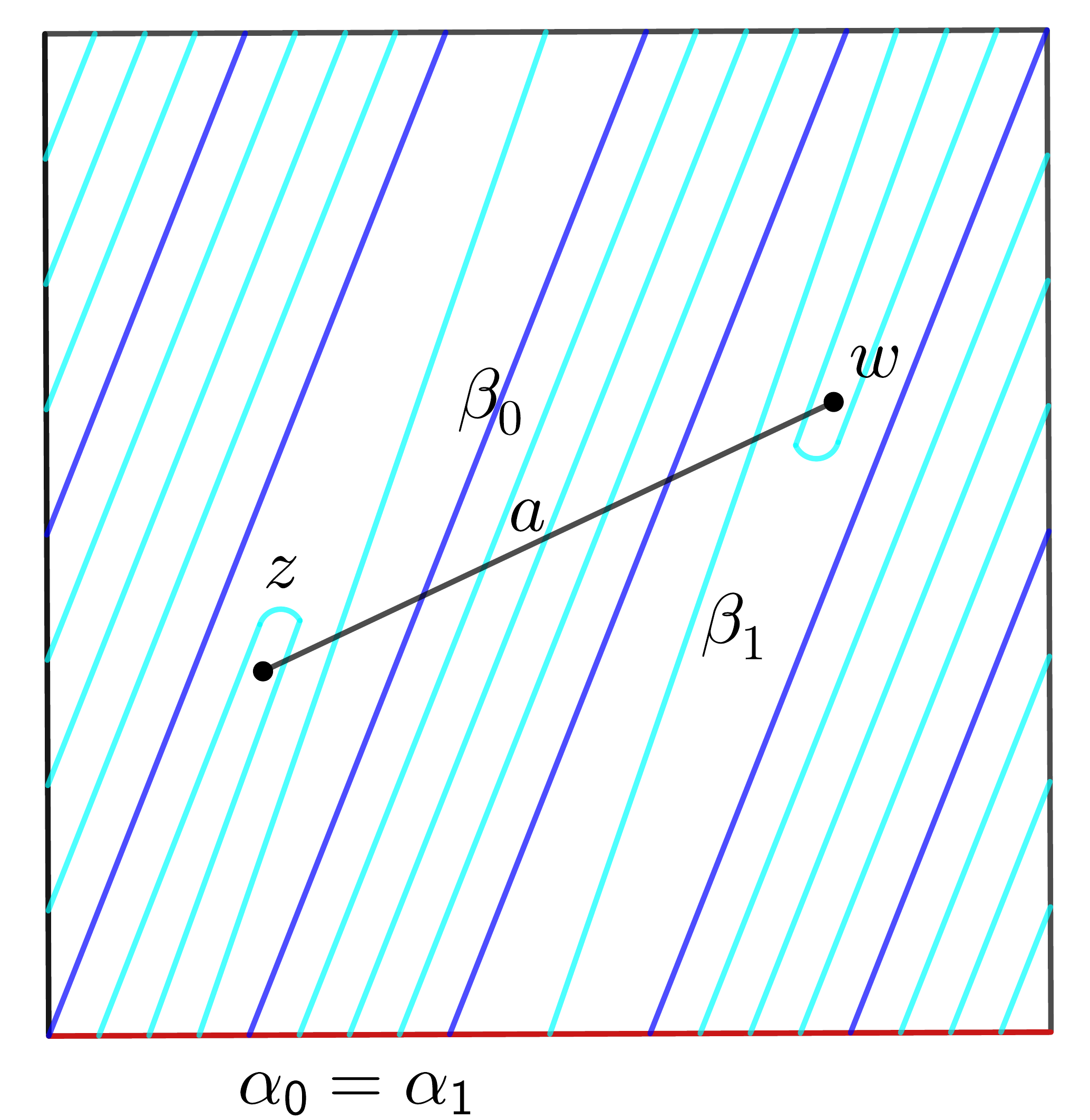}
\caption{A constrained knot in $L(5,2)$\label{l1A}.}
\end{minipage}
\begin{minipage}[ht]{0.45\textwidth}
\centering
\includegraphics[width=6.6cm]{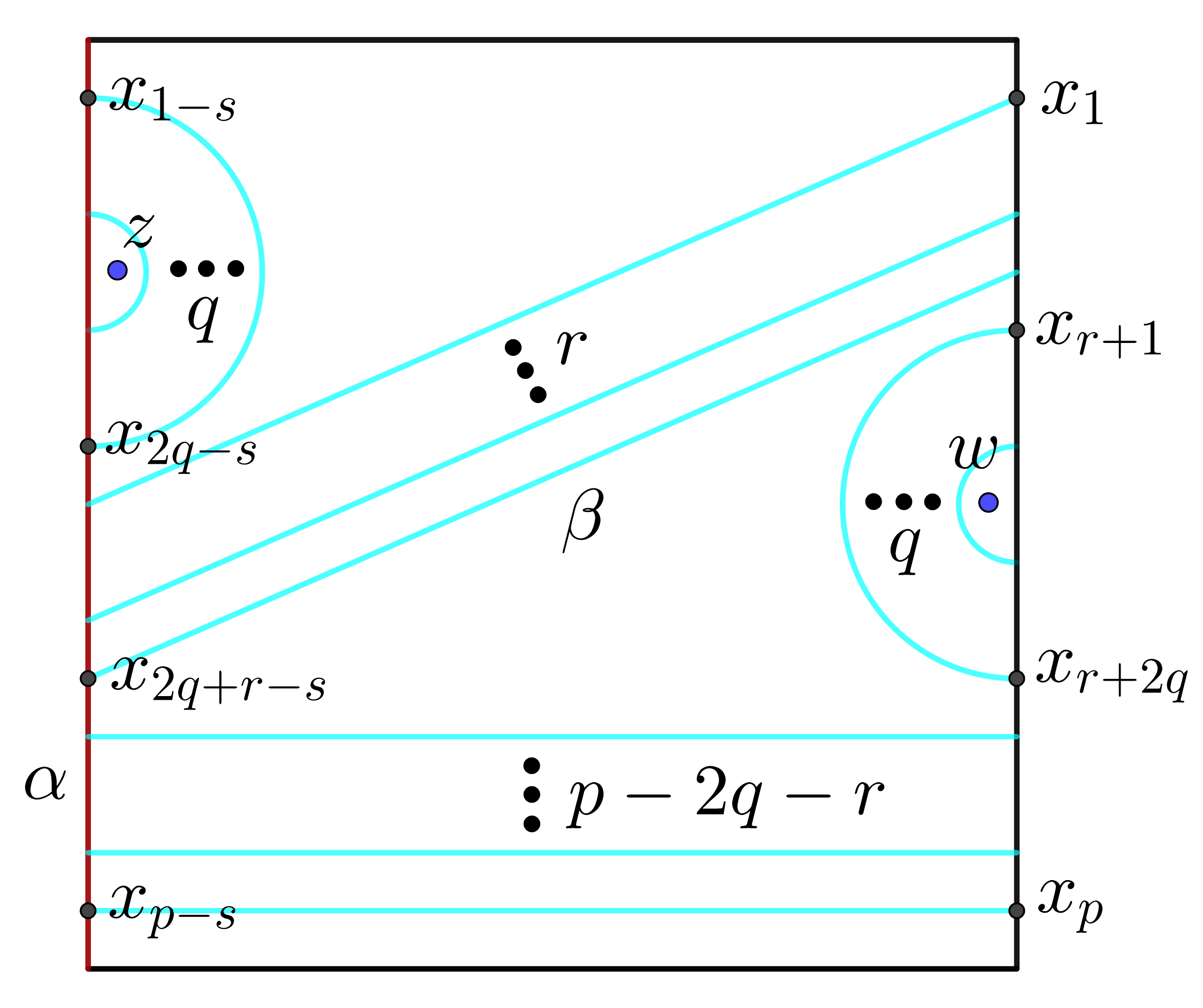}
\vspace{-0.1in}
\caption{A $(1,1)$ diagram.\label{11}}
\end{minipage}

\end{figure}

Roughly speaking, knots defined by doubly-pointed Heegaard diagrams with $g(\Sigma)=1$ are called \textbf{$(1,1)$ knots} and the corresponding diagrams are called \textbf{$(1,1)$ diagrams}; for precise definitions, see Subsection \ref{ss11}. $(1,1)$ knots are parameterized by four integers \cite{Goda2005,Rasmussen2005}, which will be denoted by $W(p,q,r,s)$; see Figure \ref{11}. After rotation, standard diagrams of constrained knots are special cases of $(1,1)$ diagrams. Moreover, the following proposition characterizes constrained knots by the distribution of spin$^c$ structures on the ambient 3-manifold in the corresponding $(1,1)$ diagrams; for the definition of spin$^c$ structures, see \cite{Ozsvath2004,Rasmussen2007}.
\begin{proposition}\label{char}
Let $K=W(p,q,r,s)$ be a $(1,1)$ knot in $Y=L(a,b)$ with $a>1$. Suppose $(T^2,\alpha,\beta,z,w)$ is the corresponding $(1,1)$ diagram of $K$. Let $\{x_i\}$ be intersection points in $\alpha\cap \beta$, ordered by an orientation of $\alpha$. Let $\mathfrak{s}_i=\mathfrak{s}_z(x_i)\in{\rm Spin}^c(Y)$ be the spin$^c$ structures on $Y$ corresponding to $x_i$. The knot $K$ is a constrained knot if and only if the followings hold.
\begin{enumerate}[(i)]
    \item For $k=|{\rm Spin}^c(Y)|(=a)$, there are integers $p_1,\dots,p_k$ so that $$0< p_1<p_2<\cdots<p_k\le p.$$
    \item $\mathfrak{s}_i=\mathfrak{s}_j$ if and only if either $i,j\in (0,p_1]\cup (p_k,p]$, or $i,j\in(p_l,p_{l+1}]$ for some $l\in\{1,\dots,p-1\}$.
\end{enumerate}
\end{proposition}

$(1,1)$ knots $W(p_1,q_1,r_1,s_1)$ and $W(p_2,q_2,r_2,s_2)$ with different parameters can represent the same knot. For example, both $W(5,2,1,3)$ and $W(5,2,1,0)$ represent the figure-8 knot in $S^3$. There is no explicit classification of $(1,1)$ knots by $W(p,q,r,s)$ to the author's knowledge. However, it is possible to classify constrained knots by the parameterization $C(p,q,l,u,v)$. In particular, the case $C(1,0,1,u,v)$ consists of 2-bridge knots in $S^3$ (\textit{c.f.} Proposition \ref{constrained2-bridge}) and the case $C(p,q,l,1,0)$ consists of simple knots in lens spaces (\textit{c.f.} Proposition \ref{sim1}). 2-bridge knots and simple knots are classified by \cite{Schubert1956} and \cite{Rasmussen2007}, respectively. The case $C(p,q,1,u,v)$ consists of connected sums of a core knot in a lens space and a 2-bridge knot (\textit{c.f.} Theorem \ref{compo}). For other constrained knots, the classification is given by the following theorem.
\begin{theorem}\label{thm2}
Suppose that $(p_1,q_1,l_1,u_1,v_1)$ and $(p_2,q_2,l_2,u_2,v_2)$ are two different collections of integers satisfying for $i=1,2$, $$p_i>1,~q_i\in[1,p-1],~l_i\in[2,p_i],~u_i>2v_i>0,~\gcd(p_i,q_i)=\gcd(u_i,v_i)=1\aand u\text{ odd}.$$Then constrained knots $C(p_i,q_i,l_i,u_i,v_i)$ represent the same knot if and only if  \[p_1=p_2=p,~q_1q_2\equiv 1\pmod {p},~l_1,l_2\in\{2,p\}\aand(l_1,u_1,v_1)=(l_2,u_2,v_2).\]
\end{theorem}

The hat version of knot Floer homology $\widehat{HFK}(Y,K)$ \cite{Ozsvath2004,Rasmussen2003} is a powerful invariant for a knot $K$ in a closed 3-manifold $Y$. It decomposes as the direct sum \begin{equation}\label{spinc grading}\widehat{HFK}(Y,K)=\bigoplus_{\mathfrak{s}\in {\rm Spin}^c(Y)}\widehat{HFK}(Y,K,\mathfrak{s}),\end{equation}
with respect to spin$^c$ structures on $Y$. Moreover, the homology $\widehat{HFK}(Y,K,\mathfrak{s})$ inherits two $\mathbb{Z}$-gradings, the Alexander grading and the Maslov grading, from the underlying chain complex $\widehat{CFK}(Y,K,\mathfrak{s})$.
\begin{definition}\label{defn: thin}
A knot $K\subset Y$ is called an \textbf{$\mathfrak{s}$-thin} knot if the difference of the Maslov grading and the Alexander grading on $\widehat{HFK}(Y,K,\mathfrak{s})$ is constant for homogeneous elements. It is called a \textbf{thin} knot if it is an $\mathfrak{s}^\prime$-thin knot for any $\mathfrak{s}^\prime\in\operatorname{Spin}^c(Y)$.
\end{definition}
Thin knots defined as above generalize \textbf{$\delta$-thin knots} in $S^3$ \cite{Rasmussen2005} and \textbf{Floer homological thin knots} \cite{Manolescu2007}. Examples of thin knots include all quasi-alternating knots \cite{Manolescu2007}, in particular all 2-bridge knots.

Suppose $K$ is a thin knot $K$ in $S^3$ and $\mathfrak{s}_0$ is the unique spin$^c$ structure on $S^3 $. Then the minus version of the knot Floer chain complex $CFK^-(S^3,K)=CFK^-(S^3,K,\mathfrak{s}_0)$ is determined by the Alexander polynomial $\Delta_K(t)$ and the signature $\sigma(K)$ up to chain homotopy \cite{Petkova2009}. For a compact 3-manifold $M$ with torus boundary, there exists a set of immersed curves $\widehat{HF}(M)$ on $\partial M-{\rm pt}$, called the \textbf{curve invariant} \cite{Hanselman2016,Hanselman2018} of $M$, which encodes the information of Heegaard Floer theory in a diagrammatic way. Based on results in  \cite[Section 3]{Petkova2009} and \cite[Section 4]{Hanselman2018}, it is easy to draw $\widehat{HF}(E(K))$ of the knot complement $E(K)=S^3- {\rm int}N(K)$ for a thin knot $K\subset S^3$. Roughly speaking, it consists of figure-8 curves and a distinguished curve.

For a $(1,1)$ knot $K\subset Y$, a combinatorial method is established to calculate the chain complex $CFK^-(Y,K)$ \cite{Goda2005}. This method applies well to 2-bridge knots and also constrained knots. From the standard diagram of a constrained knot $K\subset Y$, if we focus on intersection points corresponding to the same spin$^c$ structure $\mathfrak{s}\in {\rm Spin}^c(Y)$, we can obtain an explicit relation between $CFK^-(Y,K,\mathfrak{s})$ and $CFK^-(S^3,K^\p,\mathfrak{s}_0)$, where $K^\p$ is some 2-bridge knot. In particular, for $K=C(p,q,l,u,v)\subset Y$ and $\mathfrak{s}\in \operatorname{Spin}^c(Y)$, the group $\widehat{HFK}(Y,K,\mathfrak{s})$ is determined by Alexander polynomials of 2-bridge knots $K_1=\mathfrak{b}(u,v)$ and $K_2=\mathfrak{b}(u-2v,v)$. Hence we have the following proposition.
\begin{proposition}\label{thin}
Constrained knots are thin.
\end{proposition}
Results about thin complexes in \cite[Section 3]{Petkova2009} apply directly to $CFK^-(Y,K,\mathfrak{s})$ for a constrained knot $K\subset Y$. Then we can draw the part of the curve invariant corresponding to each spin$^c$ structure following the approach in \cite[Section 4]{Hanselman2018}. Similar to the case of a 2-bridge knot, the part of curve invariant consists of figure-8 curves and a distinguished curve; see Figure \ref{curve}. 

To connect the distinguished curves of different parts, we should study the following grading, which relates elements in $\widehat{HFK}(Y,K,\mathfrak{s})$ for different spin$^c$ structures. The total homology $\widehat{HFK}(Y,K)$ inherits a relative $H_1(E(K);\mathbb{Z})$-grading \cite[Section 3.3]{Rasmussen2007},\begin{equation}\label{alexgrading}\widehat{HFK}(Y,K)=\bigoplus_{h\in H_1(E(K);\mathbb{Z})}\widehat{HFK}(Y,K,h).\end{equation}This grading generalizes the Alexander grading on $\widehat{HFK}(S^3,K,\mathfrak{s}_0)$ and corresponds to spin$^c$ structures on $E(K)$ with some boundary conditions \cite[Section 4]{juhasz2006holomorphic}. Under the map $H_1(E(K);\mathbb{Z})\to H_1(Y;\mathbb{Z})$, this grading reduces to the grading in (\ref{spinc grading}).

Similar to the Alexander grading on $\widehat{HFK}(S^3,K,\mathfrak{s}_0)$, summands of $\widehat{HFK}(Y,K)$ in the opposite $H_1(E(K);\mathbb{Z})$-gradings are isomorphic \cite[Section 3]{Ozsvath2004}, up to a global grading shift in $H_1(E(K);\mathbb{Z})$. This symmetry is called the \textbf{global symmetry}. If it is not mentioned, this $H_1(E(K);\mathbb{Z})$-grading is also called the \textbf{Alexander grading}, and denoted by $\operatorname{gr}(x)\in H_1(E(K);\mathbb{Z})$ for a homogeneous element $x\in \widehat{HFK}(Y,K)$. To fix the ambiguity of the global grading shift, a specific grading shift will be used so that under the global symmetry, the absolute value of the Alexander grading is left invariant. The Alexander grading in this specific grading shift is called the \textbf{absolute Alexander grading}. To be clear, when considering the Alexander grading on $\widehat{HFK}(Y,K,\mathfrak{s})$ mentioned before, the spin$^c$ structure $\mathfrak{s}$ will be specified.

Following \cite[Section 3.3]{Rasmussen2007}, for a constrained knot, the Alexander grading can be calculated from the standard diagram. The Alexander grading on $\widehat{HFK}(Y,K)$ indicates an explicit way to connect different parts of the curve invariant. Then it is not hard to draw the whole curve invariant of a constrained knot. As an application, much information about Heegaard Floer theory of a constrained knot can be obtained from the curve invariant of the knot complement.

For a constrained knot $K\subset Y$ and the corresponding 2-bridge knots $K_1$ and $K_2$ mentioned before, the symmetry on $\widehat{HFK}(S^3,K_i,\mathfrak{s}_0)$ for $i=1,2$ induces a symmetry on $\widehat{HFK}(Y,K,\mathfrak{s})$, which is called the \textbf{local symmetry}. For $\mathfrak{s}\in{\rm Spin}^c(Y)$, the average $A(K,\mathfrak{s})$ of any two homogeneous elements $x,y\in \widehat{HFK}(Y,K,\mathfrak{s})$ that are symmetric under the local symmetry is called the \textbf{middle grading} of $\mathfrak{s}$, \textit{i.e.}, $$A(K,\mathfrak{s})=\frac{\operatorname{gr}(x)+\operatorname{gr}(y)}{2}\in H_1(E(K);\mathbb{Z})$$ Then we have following theorems.
\begin{theorem}\label{constrained}
For $i=1,2$, let $K_i=C(p_i,q_i,l_i,u_i,v_i)$ be constrained knots in the same lens space $Y$ with $[K_1]=[K_2]\in H_1(Y;\mathbb{Z})$. Consider the absolute Alexander grading on $\widehat{HFK}(Y,K_i)$. Then there are isomorphisms $H_1(E(K_1);\mathbb{Z})\cong H_1(E(K_2);\mathbb{Z})\cong H_1$, so that $A(K_1,\mathfrak{s})=A(K_2,\mathfrak{s})\in H_1$ for any $\mathfrak{s}\in\operatorname{Spin}^c(Y)$.
\end{theorem}
\begin{theorem}\label{simple}
Suppose $ K$ is a knot in $Y=L(p,q)$. Let $K^\p$ be a simple knot in the same manifold $Y$ with $[K^\prime]=[K]\in H_1(Y;\mathbb{Z})$. Consider the absolute Alexander gradings on $\widehat{HFK}(Y,K)$ and $\widehat{HFK}(Y,K^\p)$. We know $\widehat{HFK}(Y,K^\p)\cong \mathbb{Z}^p$. If $\widehat{HFK}(Y,K)\cong \mathbb{Z}^p$, then there are isomorphisms $H_1(E(K);\mathbb{Z})\cong H_1(E(K^\p);\mathbb{Z})\cong H_1$, so that there is a one-to-one correspondence between generators of $\widehat{HFK}(Y,K)$ and $\widehat{HFK}(Y,K^\p)$ with the same absolute Alexander grading in $H_1$.
\end{theorem}
Theorem \ref{simple} provides a clue for the following conjecture, which is related to Berge's conjecture \cite{Berge2018b}, claiming that any knot in $S^3$ admitting lens space surgeries falls into Berge's list.
\begin{conjecture}[\cite{Baker2007,Hedden2007}]\label{conjb}
Suppose $K$ is a knot in $Y=L(p,q)$. If $\widehat{HFK}(Y,K)\cong \mathbb{Z}^p$, then $K$ is a simple knot, \textit{i.e.} a $(1,1)$ knot $W(p',q',r',s')$ with $q'=0$.
\end{conjecture}

Though constrained knots are defined by doubly-pointed Heegaard diagrams, there are many other ways to construct constrained knots, at least for some special families of parameters. In the following, we introduce two approaches based on Dehn surgeries on links in $S^3$.

The first approach is inspired by the relation between knot Floer homologies of constrained knots and 2-bridge knots. A \textbf{magic link} is a 3-component link as shown in Figure \ref{magic}, where $K_0$ is a 2-bridge knot, and $K_1$ and $K_2$ are unknots. Dehn surgeries on $K_1$ and $K_2$ induce a lens space, in which $K_0$ becomes a knot $K_0^\prime$.
\begin{theorem}\label{magicpara}
Suppose that integers $p$ and $q$ satisfy $p>q>0$ and $\operatorname{gcd}(p,q)=1$. Suppose integers $n_1,n_2$ and $l$ satisfy $$n_1\in [0,\frac{p}{q}), n_2\in [0,\frac{p}{p-q})\aand l\in\{n_1q+1,p-n_1q+1,n_2(p-q)+1,p-n_2(p-q)+1\}.$$Let $L=K_0\cup K_1\cup K_2$ be a magic link with $K_0=\mathfrak{b}(u,v)$. Then the knot $C(p,q,l,u,v)$ is equivalent to the knot $K_0^\prime$ obtained by performing some Dehn surgeries on $K_1$ and $K_2$.
\end{theorem}

The second approach arises from 1-bridge braids. Suppose that the solid torus $H=S^1\times D^2$ is embedded in $\mathbb{R}^3\subset S^3$ in a standard way and suppose $K_1$ is the core of $S^3-H$. Let $K_0\subset H$ be a 1-bridge braid \cite{Gabai1989,Gabai1990}. Then $L=K_0\cup K_1$ is a 2-component link in $S^3$. An example is given in Figure \ref{1bb}. Dehn filling along a simple closed curve on $\partial H$ is equivalent to Dehn surgery on $K_1$. The resulting manifold is a lens space and $K_0$ becomes a knot $K_0^\prime$ in the lens space. A knot $K_0^\prime$ constructed from this approach is called a \textbf{1-bridge braid knot}.
\begin{theorem}\label{cla4}
The knots $C(p,q,l,u,\pm 1)$ are equivalent to 1-bridge braid knots, where $C(p,q,l,u,-1)$ means $C(p,q,l,u,u-1)$.
\end{theorem}
\begin{figure}[htbp]
\centering
\begin{minipage}[ht]{0.4\textwidth}
\centering
\includegraphics[width=3.5cm]{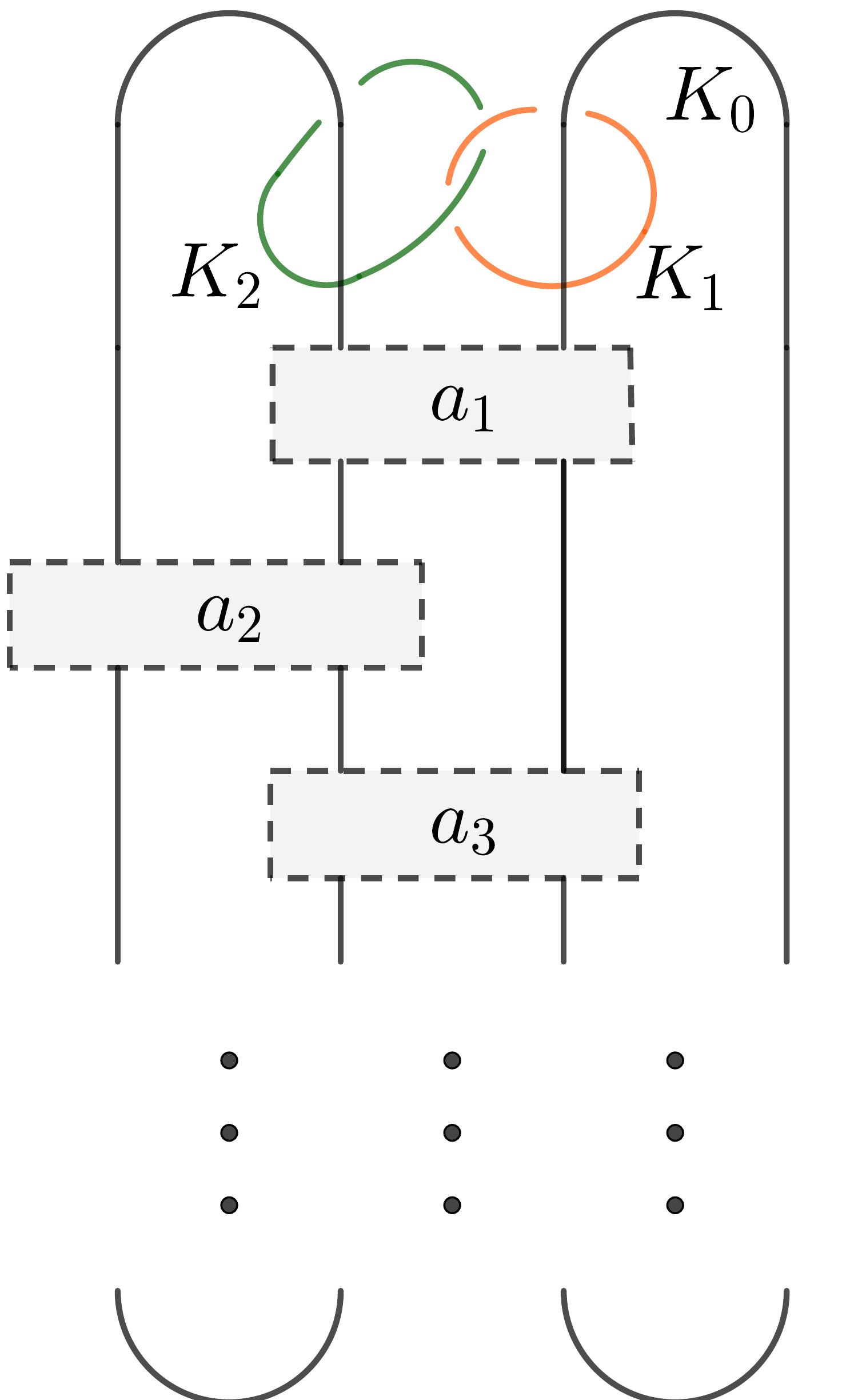}
\caption{Magic link.\label{magic}}
\end{minipage}
\begin{minipage}[ht]{0.4\textwidth}
\centering
\includegraphics[width=3.5cm]{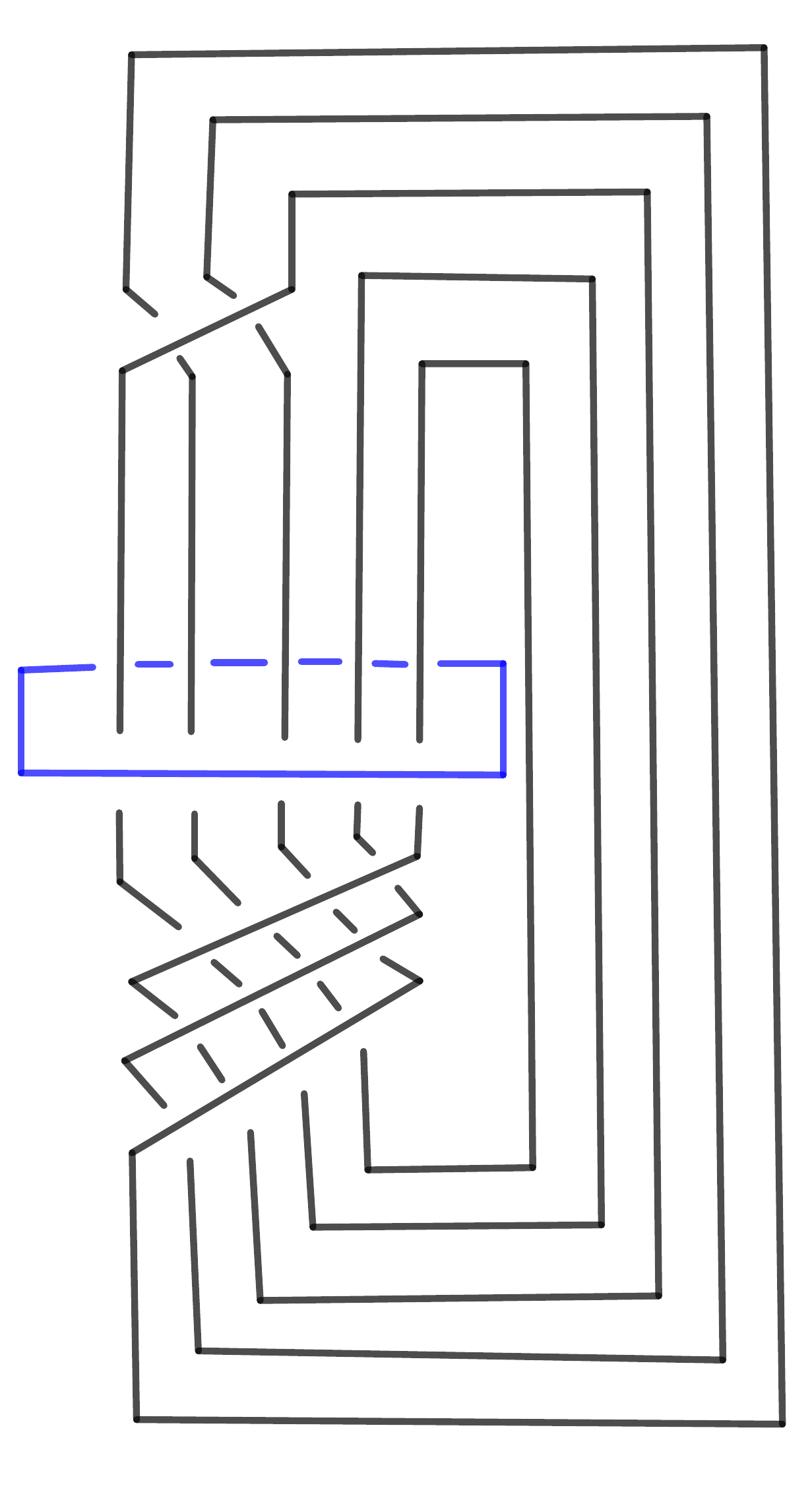}
\vspace{-0.3in}
\caption{1-bridge braid.\label{1bb}}
\end{minipage}

\end{figure}
Other than Dehn surgeries, constrained knots can also be constructed by Dehn filling the boundary of (orientable hyperbolic) 1-cusped manifolds. Many 1-cusped manifolds are knot complements of constrained knots. \textit{SnapPy} \cite{snappy} provides a list of 59068 1-cusped manifolds admitting ideal triangulations with at most 9 tetrahedra. Using the codes in \cite{Ye}, we show 21922 of them are complements of constrained knots. Table \ref{1-cusped} shows examples of 1-cusped manifolds that are complements of constrained knots. The names of manifolds in the table are from \textit{SnapPy}. The slopes in the table are considered in the basis from \textit{SnapPy} and the integers indicate the parameterization of the constrained knot that is equivalent to the core of the filling solid torus. For example, Dehn filling along the curve of slope $1/0$ on the boundary of $m003$ gives $C(10,3,3,1,0)$. If different parameterizations correspond to the same knot (\textit{c.f.} Theorem \ref{thm2}), we only show one collection of parameters. The complete list can be found in \cite{Ye}.
\begin{table}[htbp]
\caption{1-cusped manifolds and constrained knots.\label{1-cusped}}
\begin{tabular}{|p{0.8cm}|p{12cm}||}
\hline  
Name&Slope$+(p,q,l,u,v)$\\
\hline  
$m003$&$(1,0)+(10,3,3,1,0),(-1,1)+(5,4,5,3,1),(0,1)+(5,4,5,3,1)$\\
$m004$&$(1,0)+(1,0,1,5,2)$\\
$m006$&$(0,1)+(15,4,2,1,0),(1,0)+(5,3,4,3,1)$\\
$m007$&$(1,0)+(3,1,2,3,1)$\\
$m009$&$(1,0)+(2,1,2,5,2)$\\
$m010$&$(1,0)+(6,5,6,3,1)$\\
$m011$&$(1,0)+(13,3,3,1,0),(0,1)+(9,4,9,3,1)$\\
$m015$&$(1,0)+(1,0,1,7,2)$\\
$m016$&$(0,1)+(18,5,3,1,0),(-1,1)+(19,7,2,1,0)$\\
$m017$&$(0,1)+(14,3,5,1,0),(-1,1)+(21,8,21,1,0),(1,0)+(7,5,6,3,1)$\\
$m019$&$(0,1)+(17,5,4,1,0),(1,1)+(11,7,11,3,1),(1,0)+(6,5,5,3,1)$\\
$m022$&$(1,0)+(7,6,7,3,1)$\\
$m023$&$(1,0)+(3,1,3,5,2)$\\
$m026$&$(0,1)+(19,4,2,1,0),(1,0)+(8,3,7,3,1)$\\
$m027$&$(1,0)+(16,3,3,1,0),(0,1)+(13,4,13,3,1)$\\
$m029$&$(1,0)+(5,2,3,3,1)$\\
$m030$&$(1,0)+(7,4,5,3,1)$\\
$m032$&$(1,0)+(1,0,1,9,2)$\\
$m033$&$(0,1)+(18,5,5,1,0),(1,0)+(9,7,8,3,1)$\\
$m034$&$(1,0)+(4,1,3,3,1)$\\
$m035$&$(1,0)+(4,1,2,3,1)$\\
$m036$&$(-1,1)+(21,8,2,1,0),(1,0)+(3,2,3,5,1)$\\
$m037$&$(1,1)+(24,7,2,1,0),(1,0)+(8,5,6,3,1)$\\
$m038$&$(1,0)+(3,2,3,5,2)$\\
$m039$&$(1,0)+(4,1,4,5,2)$\\
$m040$&$(1,0)+(8,7,8,3,1)$\\
$m043$&$(0,1)+(25,7,24,1,0),(-1,1)+(25,9,2,1,0)$\\
$m044$&$(0,1)+(24,7,23,1,0),(-1,1)+(17,10,17,3,1),(1,0)+(7,6,5,3,1)$\\
$m045$&$(1,0)+(2,1,2,7,2)$\\
$m046$&$(-1,1)+(30,11,30,1,0),(1,0)+(10,7,8,3,1)$\\
$m047$&$(0,1)+(23,4,2,1,0),(1,0)+(11,3,10,3,1)$\\
$m049$&$(1,0)+(19,3,3,1,0),(0,1)+(17,13,17,3,1)$\\
$m052$&$(0,1)+(26,7,3,1,0),(1,0)+(7,5,4,3,1)$\\
$m053$&$(1,0)+(1,0,1,11,2)$\\
$m054$&$(0,1)+(22,5,7,1,0),(1,0)+(11,8,9,3,1)$\\
$m055$&$(1,0)+(23,7,5,1,0),(0,1)+(14,11,13,3,1)$\\

\hline 
\end{tabular}

\end{table}

\begin{proposition}\label{mfds}
Curve invariants $\widehat{HF}(M)$ of knot complements $M$ of all 1-cusped manifolds that have ideal triangulations with at most 5 tetrahedra can be drawn explicitly, except the manifolds in Table \ref{exception}.
\end{proposition}
\begin{proof}
There are 286 orientable 1-cusped manifolds that have ideal triangulations with at most 5 ideal tetrahedra. 232 manifolds are complements of constrained knots, whose curve invariants can be calculated by the method in Section \ref{s4}. Other than examples from constrained knots, 37 manifolds are Floer simple (by the list in \cite{Dunfield2019}), whose curve invariants can be calculated by the approach in \cite[Section 1]{Hanselman2018}. Other manifolds are listed in Table \ref{exception} ($(1,1)$ parameters are from Dunfield's codes \cite{Ye}). The chain complex $CFK^-(Y,K)$ of a $(1,1)$ knot can be calculated by the method in \cite{Goda2005}. Then the curve invariant can be calculated by \cite[Section 4]{Hanselman2018}. Note that chain complexes of $8_{20},9_{42},9_{46}$ in the table were calculated in \cite{Ozsvath2003}.
\end{proof}
\begin{table}[h]
\centering
\caption{Exceptions of 1-cusped manifolds.\label{exception}}
\begin{tabular}[!htbp]{|p{0.8cm}|p{4.7cm}|p{0.8cm}|p{4.7cm}|}
\hline  
Name&Comments&Name&Comments\\
\hline  
$m136$&No lens space filling&$m305$&No lens space filling\\
$m137$&$W(8, 2, 3, 1)\subset S^1\times S^2$&m$306$&No lens space filling\\
$m199$&$9_{42}=W(9, 2, 2, 3)\subset S^3$ &$m345$&$W(10, 3, 1, 5)\subset L(2,1)$\\
$m201$&$10_{132}=W(11, 2, 1, 3)\subset S^3$&$m370$&$(1,0)$ filling gives $L(8,3)$\\
$m206$&$(1,0)$ filling gives $L(5,2)$&$m372$&$9_{46}=$Pretzel$(-3,3,3)\subset S^3$\\
$m222$&$8_{20}=W(9, 3, 0, 2)\subset S^3$&$m389$&$10_{139}=W(11, 3, 1, 4)\subset S^3$\\
$m224$&$11_{190}=W(13, 2, 1, 8)\subset S^3$&$m390$&$(1,0)$ filling gives $L(7,2)$\\
$m235$&No lens space filling&$m410$&No lens space filling\\
$m304$&$W(12, 3, 0, 5)\subset L(2,1)$&&\\
\hline 
\end{tabular}

\end{table}

It is known that a 2-bridge knot $\mathfrak{b}(u,v)$ is a torus knot if $v=1$ or $v=u-1$. The latter case is written as $v=-1$. If $v\neq \pm 1$, the 2-bridge knot $\mathfrak{b}(u,v)$ is hyperbolic, \textit{i.e.}, the interior of the knot complement admits a hyperbolic metric of finite volume. We may generalize the results about 2-bridge knots to constrained knots. Note that the knot complement of a torus knot is a Seifert fibered space. We have the following theorem.
\begin{theorem}\label{torus}
If $C(p,q,l,u,v)$ has Seifert fibered complement, then $v=\pm 1$.
\end{theorem}
Since $C(p,q,1,u,v)$ is a connected sum of two knots, there is an essential torus in the knot complement, and hence $C(p,q,1,u,v)$ is not hyperbolic. Using the codes in \cite{Ye} and the \textit{verify\_hyperbolicity()} function in \textit{SnapPy}, we verified that $C(p,q,l,u,v)$ are hyperbolic for $p\le 10,l>1,u< 20,v\neq \pm 1$. Then we have the following conjecture.

\begin{conjecture}
$C(p,q,l,u,v)$ with $l>1$ and $v\neq \pm 1$ are hyperbolic.
\end{conjecture}

\textit{Organizations}.\quad The remainder of this paper is organized as follows. In Section \ref{s2}, we collect some conventions and definitions in 3-dimensional topology and facts about $(1,1)$ knots, simple knots and 2-bridge knots. In Section \ref{s3}, we describe the parameterization of constrained knots and proves Proposition \ref{char}. Many propositions about constrained knots are also given in Section \ref{s3}. In Section \ref{s4}, an algorithm of the knot Floer homology of a constrained knot is obtained, which induces Proposition \ref{thin} and the necessary part of Theorem \ref{thm2}. In Section \ref{s5}, we study knots in the same homology class and proves Theorem \ref{constrained} and Theorem \ref{simple} by Turaev torsions of 3-manifolds. In Section \ref{s6}, we finish the proof of Theorem \ref{thm2} by constructing isomorphisms between fundamental groups of knot complements and applying the fact that knots are determined by their fundamental groups. The last three sections discuss magic links, 1-bridge braid knots and SnapPy manifolds, respectively.

\textit{Acknowledgements}. The author would like to thank his supervisor Jacob Rasmussen for introducing him to this project and guiding him on his research. The author is grateful to Nathan M. Dunfield for sharing codes about $(1,1)$ knots, Sirui Lu and Muge Chen for helping him calculate many examples by computer. The author is grateful to Zekun Chen, Zhenkun Li, Donghao Wang, Zipei Nie and Wenzhao Chen for helpful conversations, and anonymous referees for the helpful comments. The author would also like to thank his parents and relatives for their support. Calculations are based on \textit{Mathematica} \cite{WolframResearchInc.2020}, \textit{SageMath} \cite{sage} and \textit{SnapPy} \cite{snappy}.
\section{Preliminaries}\label{s2}


In this section, we begin by basic conventions. For $r\in\mathbb{R}$, let $\lceil r\rceil$ and $\lfloor r \rfloor$ denote the minimal integer and the maximal integer satisfying $\lceil r\rceil\ge r$ and $\lfloor r \rfloor\le r$, respectively. For a group $H$, let $\operatorname{Tors}H$ denote the set of torsion elements in $H$.

If it is not mentioned, all manifolds are smooth, connected and oriented, and orientations of knots are omitted. The fundamental group of a manifold $M$ is denoted by $\pi_1(M)$, where the basepoint is omitted. For a submanifold $A$ in a manifold $Y$, let $N(A)$ denote the regular neighborhood of $A$ in $Y$ and let ${\rm int} N(A)$ denote its interior. Suppose $Y$ is a closed 3-manifold and $K$ is a knot in $Y$. Let $E(K)=Y-{\rm int} N(K)$ denote the knot complement of $K$. 

For a simple closed curve $\al$ on a surface $\Sigma$, let $[\al]$ denote its homology class in $H_1(\Sigma;\mathbb{Z})$. If it is clear, we do not distinguish $\al$ and $[\al]$. The algebraic intersection number of two curves $\alpha$ and $\beta$ on a surface $\Sigma$ is denoted by $[\al]\cdot[\be]$ or $\al\cdot \be$, while the number of intersection points of $\al$ and $\be$ is denoted by $|\al\cap \be|$.

A basis $(m,l)$ of $H_1(T^2;\mathbb{Z})$ always satisfies $m\cdot l=-1$. Suppose $K$ is a knot in a closed 3-manifold $Y$. A basis of $\partial E(K)$ means a basis of $H_1(\partial E(K);\mathbb{Z})\cong H_1(T^2;\mathbb{Z})$. In practice, there are two standard choices of the basis of $\partial E(K)$.

\begin{enumerate}
    \item Let $m$ and $l$ be simple closed curves on $\partial E(K)$ such that Dehn filling along $m$ gives $Y$,  $m\cdot l=-1$, and the orientation of $m$ is determined from the orientation of $K$ by the `right-hand rule'. The curves $m$ and $l$ are called the \textbf{meridian} and the \textbf{longitude} of the knot $K$, respectively. The basis $(m,l)$ is called the \textbf{regular basis} of $\partial E(K)$.
    \item Let $m^*$ and $l^*$ be simple closed curves on $\partial E(K)$ such that  $l^*$ represents the generator of $\operatorname{Ker}(H_1(E(K);\mathbb{Q})\rightarrow H_1(Y;\mathbb{Q}))$ and $m^*\cdot l^*=-1$. They are called the \textbf{homological meridian} and the \textbf{homological longitude} of the knot $K$, respectively. The basis $(m^*,l^*)$ is called the \textbf{homological basis} of $\partial E(K)$.
\end{enumerate}

The choices of $l$ and $m^*$ are not unique. The longitude $l$ is isotopic to $K$, while $m^*$ does not have any geometric meaning. Sometimes (\textit{e.g.} for knots in $S^3$) these two choices of the basis are equivalent. If it is not mentioned, we choose the regular basis $(m,l)$ as the basis of $\partial E(K) $. The slope $p/q$ of a Dehn surgery indicates that the meridian of the filling solid torus is glued to the curve corresponding to $pm+ql$.

Suppose $M$ is an oriented manifold. Let $-M$ denote the same manifold with the reverse orientation, called the \textbf{mirror manifold} of $M$. Suppose $K$ is an (oriented) knot in a 3-manifold $M$. Then it is specified by the knot complement $E(K)$ and the (oriented) meridian $m$ of the knot. The \textbf{mirror image} of $K$ is the knot in $-M$ specified by $(-M,-m)$.

When mentioning that $Y=L(p,q)$ is a lens space, we always suppose that $p$ and $q$ are integers satisfying $\gcd(p,q)=1$ and $(p,q)\neq (0,1)$. In particular, the manifold $S^1\times S^2$ is not considered as a lens space. The lens space is oriented as follows. Let $(T^2,\al_0,\be_0)$ be the standard diagram of a lens space. Then the orientation on the $\alpha_0$-handlebody is induced from the standard embedding of a solid torus in $\mathbb{R}^3$. With this convention, the lens space $L(p,q)$ is obtained from the $p/q$ Dehn surgery on the unknot in $S^3$.

Then we recall some definitions about knots in closed 3-manifolds. Suppose $K$ is a knot in a lens space $Y$.

The knot $K$ is called a \textbf{trivial knot} or a \textbf{unknot} if it bounds a disk embedded in $Y$. It is called a \textbf{core knot} if $E(K)$ is homeomorphic to a solid torus. It is called a \textbf{split knot} if $Y$ contains a sphere which decomposes $Y$ into a punctured lens space and a ball containing $K$ in its interior. It is called a \textbf{composite knot} if $Y$ contains a 2-sphere $S$ which intersects $K$ transeversely in two points and $S\cap E(K)$ is $\partial$-incompressible in $E(K)$. It is called a \textbf{prime knot} if it is not a composite knot.

The torus $T^2\subset Y$ in the standard diagram $(T^2,\al_0,\be_0)$ is called the \textbf{Heegaard torus} of $Y$. The knot $K$ is called a \textbf{$(p,q)$ torus knot in $Y$} if $K$ can be isotoped to lie on the Heegaard torus as an essential curve with slope $p/q$ in the standard diagram of $Y$. The unknot is considered as a torus knot. Complements of torus knots in lens spaces are Seifert fibered spaces.

The knot $K$ is called a \textbf{satellite knot} if $E(K)$ has an essential torus. For $q>1$, the space $C_{p,q}$ is obtained by removing a regular fiber from a solid torus with a $(p, q)$ fibering, which is called a \textbf{cable space of type $(p,q)$}. The knot $K$ is called a $(p,q)$ \textbf{cable knot} on $K_0$ if $K_0$ is knot in $Y$ such that $E(K)=E(K_0)\cup C_{p,q}$. In this case, the knot $K$ lies as an essential curve on $\partial N(K_0)$, and $K$ is neither a longitude nor a meridian of $K_0$. It is well-known that composite knots are satellite knots. A cable knot on $K_0$ with $E(K_0)$ having an incompressible boundary is also a satellite knot.

\subsection{$(1,1)$ knots}\label{ss11}
In this subsection, we review some facts about $(1,1)$ knots. Proofs are omitted.

A knot $K$ in a closed 3-manifold $Y$ has \textbf{tunnel number one} if there is a properly embedded arc $\gamma$ in $E(K)$ so that $E(K)-N(\gamma)$ is a genus two handlebody. Equivalently, the knot complement $E(K)$ admits a genus two Heegaard splitting. The arc $\gamma$ is called an \textbf{unknotting tunnel} for $K$. A proper embedded arc $\gamma$ in a handlebody $H$ is called a \textbf{trivial arc} if there is an embedded disk $D\subset H$ such that $\partial D=\gamma \cup (D\cap \partial H)$. The disk $D$ is called the \textbf{cancelling disk} of $\gamma$. A knot $K$ in a 3-manifold $Y$ admits a \textbf{$(1,1)$ decomposition} if there is a genus one Heegaard splitting $Y= H_1 \cup_{T^2}H_2$ so that $K \cap  H_i$ is a properly embedded trivial arc $k_i$ in $H_i$ for $i = 1, 2$. In this case, $Y$ is either a lens space (including $S^3$), or $S^1\times S^2$. A knot $K$ that admits a $(1,1)$ decomposition is called a \textbf{$(1,1)$ knot}. In this paper, we do not consider $(1,1)$ knots in $S^1\times S^2$. Note that any $(1,1)$ knot has tunnel number one.
\begin{proposition}[{\cite[Proposition 3.2]{Williams2009}}]\label{haken}
If a nontrivial knot in a lens space has tunnel number one, then the complement is irreducible. Consequently, the complement is a Haken manifold.
\end{proposition}
$(1,1)$ knots are parameterized by their doubly-pointed Heegaard diagrams. The orientation of the knot is unimportant in this paper so it is free to swap two basepoints.
\begin{proposition}[\cite{Goda2005,Rasmussen2005}]\label{oneonepara}
For $p,q,r,s\in \mathbb{N}$ satisfying $2q+r\le p$ and $s<p$, a $(1,1)$ decomposition of a knot determines and is determined by a doubly-pointed Heegaard diagram. After isotopy, such a diagram looks like $(T^2,\alpha,\beta,z,w)$ in Figure \ref{11}, where $p$ is the total number of intersection points, $q$ is the number of strands around each basepoint, $r$ is the number of strands in the middle band, and the $i$-th point on the right-hand side is identified with the $(i + s)$-th point on the left-hand side.
\end{proposition}
Let $W(p,q,r,s)=W(p,q,r,s)_+$ denote the $(1,1)$ knot defined by Figure \ref{11}, and let $W(p,q,r,s)_-$ denote the knot defined by the diagram that is vertically symmetric to Figure \ref{11}. These doubly-pointed Heegaard diagrams are called \textbf{$(1,1)$ diagrams}. In the diagrams, strands around basepoints are called \textbf{rainbows} and strands in the bands are called \textbf{stripes}. The roles of $\alpha$ and $\beta$ curves here are different from those in \cite{Rasmussen2005}. For the same parameters, the knot $W(p,q,r,s)$ is the mirror image of $K(p,q,r,s)$ in \cite{Rasmussen2005}.
\begin{proposition}\label{check0}
There are relations among $(1,1)$ knots:
\begin{enumerate}[(i)]
\item $W(p,q,r,s)_+$ is the mirror image of $W(p,q,r,p-s)_-$;
\item $W(p,q,r,s)_+$ is equivalent to $W(p,q,p-2q-r,s-2q)_-$.

\end{enumerate}
Thus, we know that $W(p,q,r,s)_+$ is the mirror image of $W(p,q,p-2q-r,p-s+2q)_+$.

\end{proposition}
\begin{proof}
The first relation is from the vertical symmetry. The second relation is from redrawing the diagram in a way that the lower band becomes the middle band and the middle band becomes the lower band.
\end{proof}
\begin{definition}
For a closed 3-manifold $Y$, consider the hat version of Heegaard Floer homology $\widehat{HF}(Y)$ defined in \cite{Ozsvath2004b}. A closed 3-manifold $Y$ is called an \textbf{L-space} if $\widehat{HF}(Y,\mathfrak{s})\cong \mathbb{Z}$ for any $\mathfrak{s}\in \operatorname{Spin}^c(Y)$. A knot $K$ in an L-space $Y$ is called an \textbf{L-space knot} if a nontrivial Dehn surgery on $K$ gives an L-space.
\end{definition}
\begin{theorem}[{\cite[Theorem 1.2]{Greene2018}}]\label{lspace}
A $(1,1)$ knot is an L-space knot if and only if in the corresponding $(1,1)$ diagram with any orientation of $\beta$, all of rainbows around a fixed basepoint are oriented in the same way.
\end{theorem}
\begin{definition}[{\cite[Section 2.1]{Rasmussen2007}}]\label{simpp}
Let $(T^2,\alpha_0,\beta_0)$ be the standard Heegaard diagram of $L(p,q)$ and let $P_i$ for $i\in\mathbb{Z}/p\mathbb{Z}$ be components of $T^2-\alpha_0\cup\beta_0$, ordered from left to right. Let $z\in P_1$ and $w\in P_{k+1}$ be two points. The knot defined by $(T^2,\alpha_0,\beta_0,z,w)$ is called a \textbf{simple knot}, which is denoted by  $S(p,q,k)$ ($K(p,q,k)$ in \cite{Rasmussen2007}). The orientation of the knot is induced by the orientation of the arc connecting $z$ to $w$.
\end{definition}
\begin{proposition}[{\cite[Lemma 2.5]{Rasmussen2007}}]\label{sim}
There are relations among $S(p,q,k)$:
\begin{enumerate}[(i)]
\item $S(p, q, -k)$ is the orientation-reverse of $S(p, q, k)$;
\item $S(p, -q, -k)$ is the mirror image of $S(p, q, k)$;
\item $S(p, q, k) \cong S(p, q^\prime, kq^\prime)$, where $qq^\prime \equiv 1 \pmod p$.
\end{enumerate}
\end{proposition}
Note that a simple knot is homotopic to an immersed curve on $T^2$. The homology class $[S(p,q,k)]$ in $H_1(L(p,q);\mathbb{Z})$ is $k[b]$, where $b$ is the core curve of $\beta_0$-handlebody. The simple knots $S(p,q,k_1)$ and $S(p,q,k_2)$ represent the same homology class if and only if $k_1 \equiv k_2 \pmod p$. Thus, there is no relation other than thoses in Proposition \ref{sim}.
\subsection{2-bridge knots}\label{s2b}

In this subsection, we review some facts about 2-bridge links from \cite{Rasmussen2002,Burde2003,Murasugi2008}.
\begin{definition}\label{2bknot}
Suppose $h$ is the height function given by the $z$-coordinate in $\mathbb{R}^3\subset S^3$. A knot or a link in $S^3$ is called a \textbf{2-bridge knot} or a \textbf{2-bridge link} if it can be isotoped in a presentation so that $h$ has two maxima and two minima on it. Such a presentation is called the \textbf{standard presentation} of the knot.
\end{definition}
A 2-bridge link has two components. Each component is equivalent to the unknot. Suppose integers $a$ and $b$ satisfying $\operatorname{gcd}(a,b)=1$ and $a>1$. For every oriented lens space $L(a,b)$, there is a unique 2-bridge knot or link whose branched double cover space is diffeomorphic to $L(a,b)$. Let $\mathfrak{b}(a,b)$ denote the knot or link related to $L(a,b)$. It is a knot if $a$ is odd, and a link if $a$ is even. Thus, the classification of 2-bridge knots or links depends on the classification of lens spaces \cite{Brody1960}. For $i=1,2$, two 2-bridge knots or links $\mathfrak{b}(a_i,b_i)$ are equivalent if and only if $a_1=a_2=a$ and $b_1\equiv b_2^{\pm 1} \pmod a$.
\begin{figure}[htbp]
\centering
\begin{minipage}[ht]{0.2\textwidth}
\centering
\includegraphics[width=3cm]{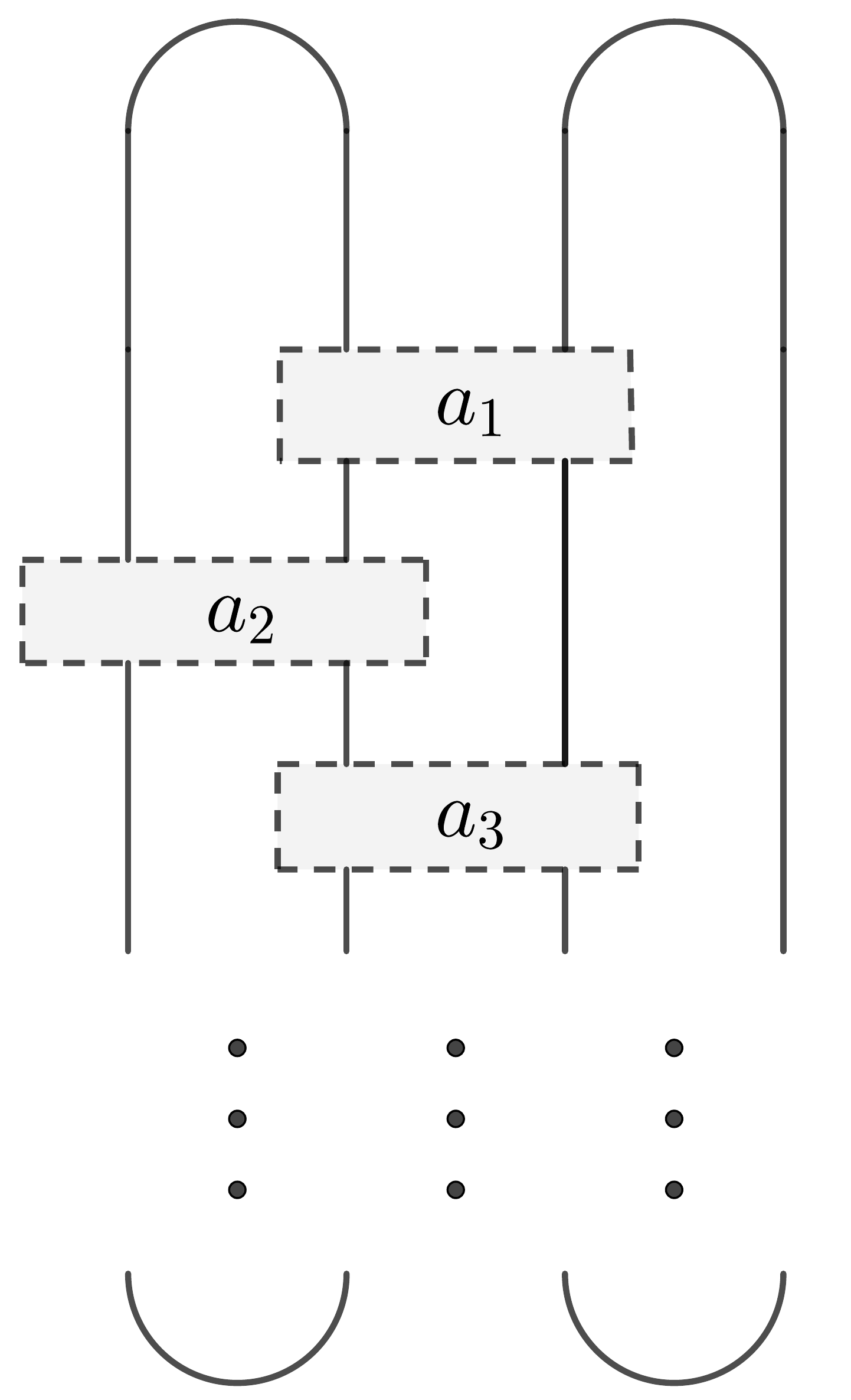}
\vspace{-0.3in}
\caption{2-bridge.\label{st}}
\end{minipage}
\begin{minipage}[ht]{0.38\textwidth}
\centering
\includegraphics[width=4.5cm]{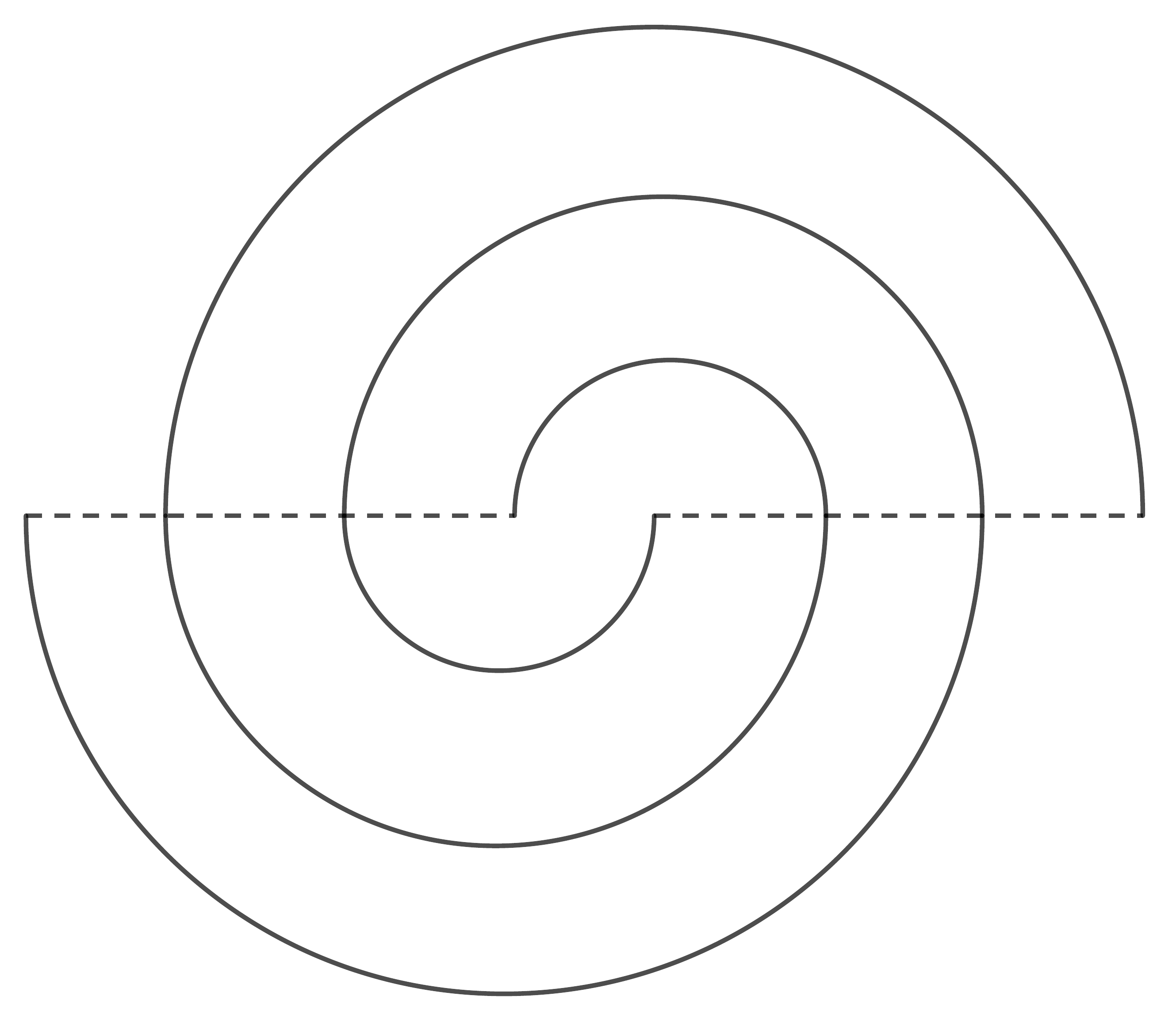}
\vspace{+0.3in}
\caption{$\mathfrak{b}(3,1)$.\label{f5}}
\end{minipage}
\begin{minipage}[ht]{0.38\textwidth}
\centering
\includegraphics[width=6cm]{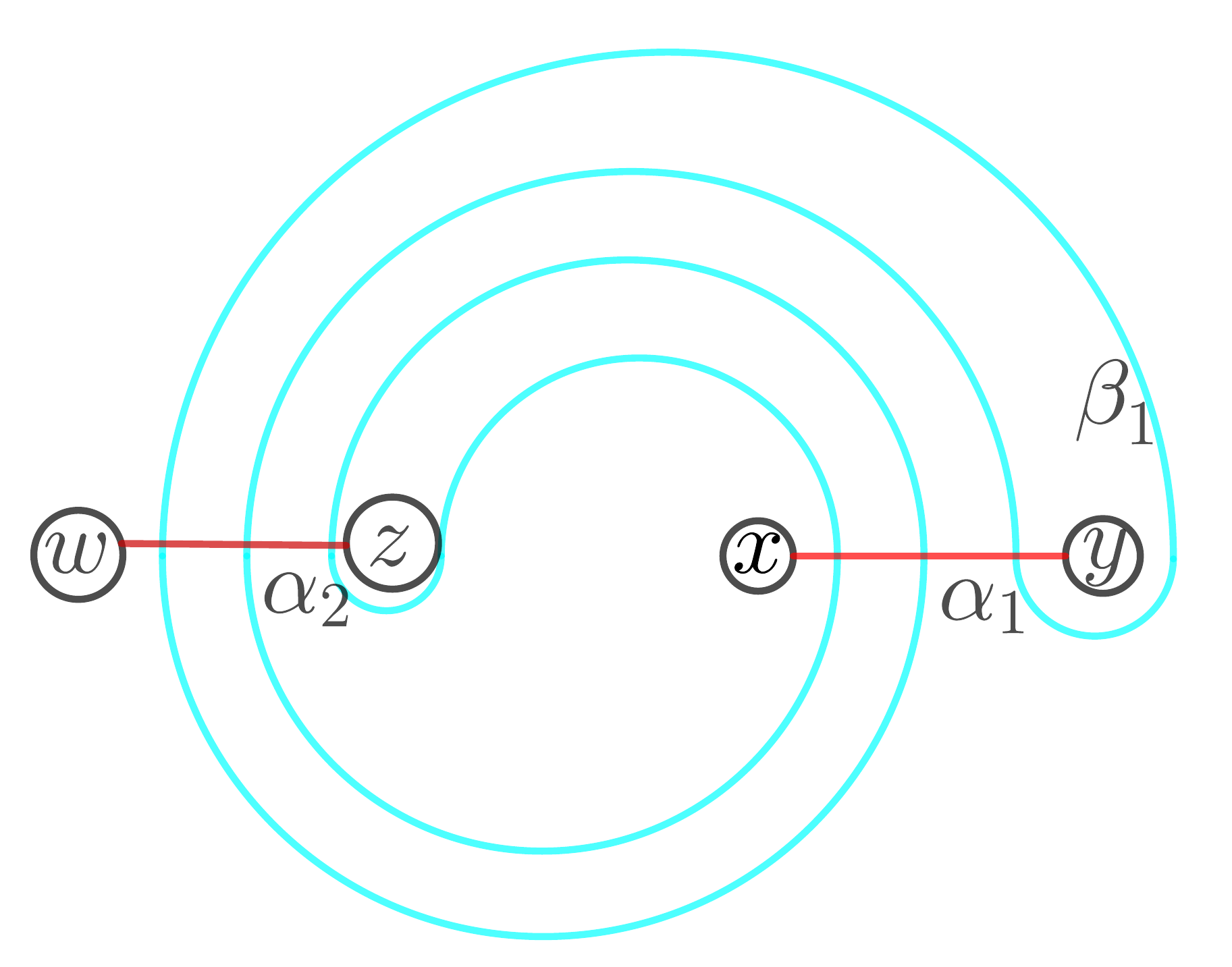}
\vspace{-0.25in}
\caption{Diagram of $E(\mathfrak{b}(3,1))$.\label{f6}}
\end{minipage}
\end{figure}

Suppose $a/b$ is represented as the continued fraction \[[0;a_1,-a_2,\dots,(-1)^{m+1}a_m]=0+\frac{1}{a_1-\frac{1}{a_2-\frac{1}{a_3-\cdots}}}.\]Moreover, suppose $m$ is odd. The standard presentation of a 2-bridge knot or link $\mathfrak{b}(a,b)$ looks like Figure \ref{st}, where $|a_i|$ for $i\in[1,m]$ represent numbers of half-twists in the boxes and signs of $a_i$ represent signs of half-twists. Different choices of continued fractions give the same knot or link. For any 2-bridge knot or link, the numbers $(-1)^{i+1}a_i$ can be all positive, which implies any 2-bridge knot or link is alternating.

The knot or link $\mathfrak{b}(a,b)$ admits another canonical presentation known as the \textbf{Schubert normal form}. It induces a Heegaard diagram of $E(\mathfrak{b}(a,b))$ and a doubly-pointed Heegaard diagram of $\mathfrak{b}(a,b)$. Figure \ref{f5} gives an example of the Schubert normal form of $\mathfrak{b}(3,1)$ and Figure \ref{f6} is the corresponding Heegaard diagram of the knot complement. The corresponding doubly-pointed Heegaard diagram is obtained by replacing $\alpha_2$ by two basepoints $z$ and $w$. Two horizontal strands in the Schubert normal form are arcs near two maxima in the standard presentation. Thus, two 1-handles attached to points $w,z$ and $x,y$ in Figure \ref{f6} are neighborhoods of these arcs, respectively.
\begin{proposition}[\cite{Rasmussen2002}]\label{alexander}
Suppose $K=\mathfrak{b}(a,b)$ with $b$ odd and $|b|<a$. The symmetrized Alexander polynomial $\Delta_K(t)$ and the signature $\sigma(K)$ satisfy
\[\Delta_K(t)=t^{-\frac{\sigma(K)}{2}}\sum_{i=0}^{a-1} (-1)^it^{\sum_{j=0}^i (-1)^{\lfloor\frac{ib}{a}\rfloor}},\sigma(K)=\sum_{i=1}^{a-1}(-1)^{\lfloor\frac{ib}{a}\rfloor}.\]
\end{proposition}
\begin{proposition}[\cite{Doll1992,Hayashi1999}]Let $K$ be a $(1,1)$ knot in a lens space. Then $K$ is a split knot if and only if $K$ is the unknot. The knot $K$ is a composite knot if and only if it is a connected sum of a 2-bridge knot and a core knot of a lens space.
\end{proposition}
\section{Parameterization and characterization}\label{s3}

For a constrained knot $K$, there is a standard diagram $(T^2,\alpha_1,\beta_1,z,w)$ of $K$ defined in the introduction. Based on standard diagrams, we describe the parameterization of constrained knots. For integers $p,q,q^\p$ satisfying $$\gcd(p,q)=\gcd(p,q^\p)=1\aand qq^\p\equiv 1\pmod p,$$we know that $L(p,q)$ is diffeomorphic to $L(p,q^\p)$ \cite{Brody1960}. Suppose $(T^2,\al_0,\be_0)$ is the standard diagram of $L(p,q^\p)$, \textit{i.e.}, the curve $\be_0$ is obtained from a straight line of slope $p/q^\p$ in $\mathbb{R}^2$, and suppose that the diagram $(T^2,\alpha_1,\beta_1,z,w)$ is induced by $(T^2,\al_0,\be_0)$ as in the introduction. The curves $\alpha_0$ and $\beta_0$ divide $T^2$ into $p$ regions, which are parallelograms in Figure \ref{l1A}; see also the left subfigure of Figure  \ref{l2}. A new diagram $C$ is obtained by gluing top edges and bottom edges of parallelograms. We can shape $C$ into a square. An example is shown in Figure \ref{l2}, where $p=5,q=3,q^\p=2$.
\begin{figure}[ht]
\centering 
\includegraphics[width=0.9\textwidth]{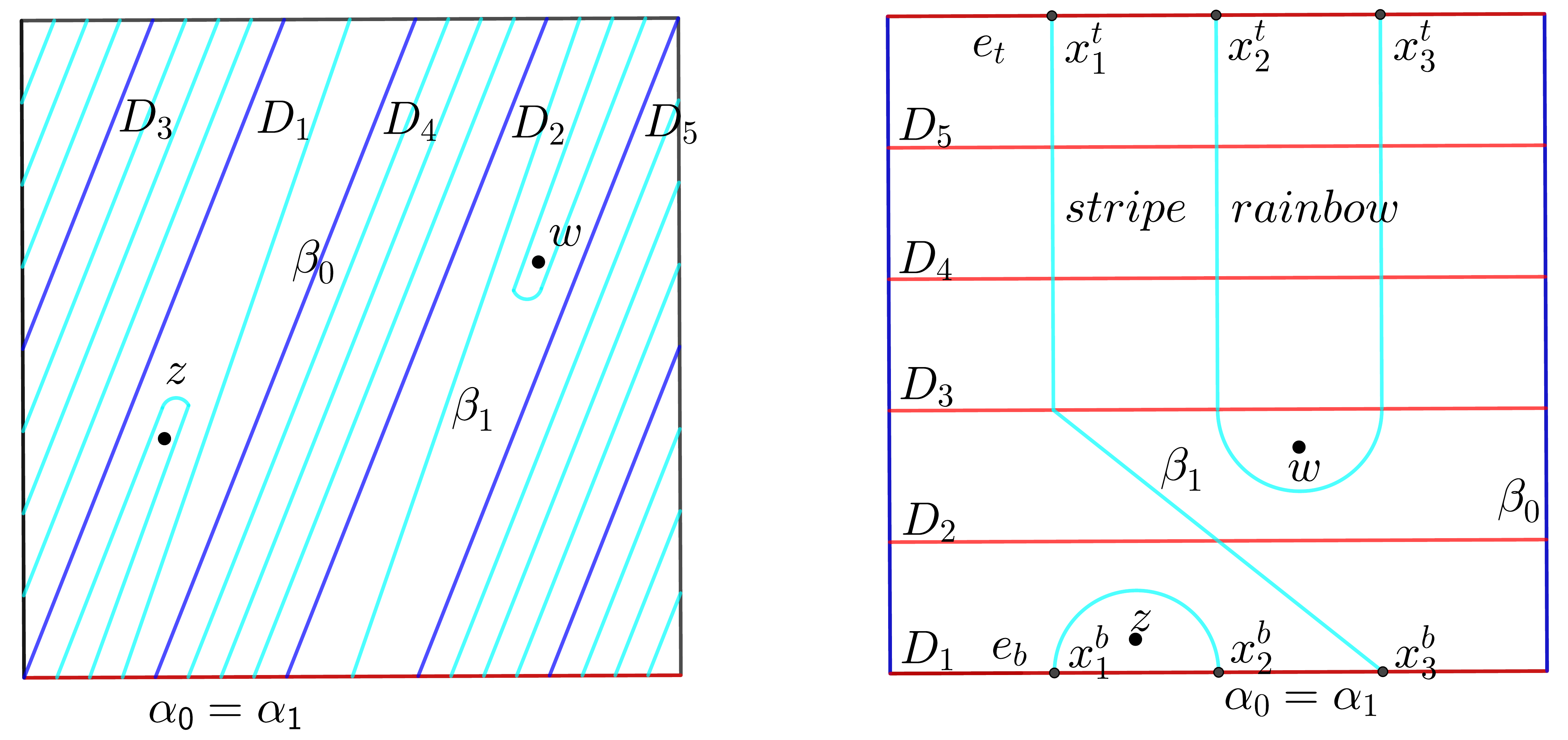} 
\caption{Heegaard diagrams of $C(5,3,2,3,1)$.}
\label{l2}
\end{figure}

For $i\in\mathbb{Z}/p\mathbb{Z}$, let $D_i$ denote rectangles in $C$, ordered from the bottom edge to the top edge. Since $qq^\p\equiv 1\pmod 1$ and we start with the standard diagram of $L(p,q^\p)$, we know that the right edge of $D_j$ is glued to the left edge of $D_{j+q}$. The bottom edge $e_b$ of $D_1$ is glued to the top edge $e_t$ of $D_p$. By definition of a constrained knot, the curve $\al_1$ is the same as $\al_0$ and the curve $\be_1$ is disjoint from $\be_0$. Thus, in this new diagram $C$, the curve $\alpha_1$ is the union of $p$ horizontal lines and $\beta_1$ is the union of strands which are disjoint from vertical edges of $D_i$ for $i\in\mathbb{Z}/p\mathbb{Z}$.

Similar to the definitions for $(1,1)$ knots, strands in the standard diagram of a constrained knot are called \textbf{rainbows} and \textbf{stripes}. Boundary points of a rainbow and a stripe are called \textbf{rainbow points} and \textbf{stripe points}, respectively. A rainbow must bound a basepoint, otherwise it can be removed by isotopy. Numbers of rainbows on $e_b$ and $e_t$ are the same since the numbers of rainbow points are the same. Without loss of generality, suppose $z$ is in all rainbows on $e_b$ and $w$ is in all rainbows on $e_t$. Let $x^b_i$ and $x^t_i$ for $i\in[1,u]$ be boundary points on the bottom edge and the top edge, respectively, ordered from left to right in the right subfigure of Figure \ref{l2}.
\begin{lemma}\label{paralem}
The number $u$ of boundary points on $e_b$ or $e_t$ is odd. When $u=1$, there is no rainbow and only one stripe. When $u>1$, there exists an integer $v\in (0,u/2)$ so that one of the following cases happens:
\begin{enumerate}[(i)]
\item the set $\{x_i^b|i\le 2v\}\cup\{x_i^t|i> u-2v\}$ contains all rainbow points;
\item the set $\{x_i^t|i\le 2v\}\cup\{x_i^b|i> u-2v\}$ contains all rainbow points.
\end{enumerate}
\end{lemma}
\begin{proof} The algebraic intersection number of $\beta_1$ and $e_b$ is odd. Hence $u$ is also odd. If $u=1$, then the argument is clear.

Suppose $u>1$, we show the last argument in three steps. Firstly, if both $x^b_i$ and $x_j^b$ are boundary points of the same rainbow $R$, then $x^b_k$ for $i<k<j$ are all rainbow points, otherwise the stripe corresponding to the stripe point $x^b_k$ would intersect $R$. Thus, rainbow points on $e_b$ are consecutive. The same assertion holds for $x^t_i$.

Secondly, one of $x^b_1$ and $x^t_1$ must be a rainbow point. Indeed, if this were not true, then both $x^b_1$ and $x^t_1$ would be stripe points. They cannot be boundary points of the same stripe, otherwise $\beta_1$ would not be connected. They cannot be boundary points of different stripes, otherwise two corresponding stripes would intersect each other. Thus, the assumption is false. Similarly, one of $x^b_u$ and $x^t_u$ must be a rainbow point.

Finally, if $x^b_1$ is a rainbow point, then $x^b_u$ cannot be a rainbow point, otherwise all points were rainbow points. As discussed above, the point $x^t_u$ is a rainbow point. Since the number of rainbow points on $e_t$ is even, there exists an integer $v$ satisfying Case (i). If $x^t_1$ is a rainbow point, similar argument implies there exists $v$ satisfying Case (ii).
\end{proof}
When $u=1$, after isotoping $\beta_1$, suppose the unique stripe is a vertical line in $C-\{z,w\}$. By moving $z$ through the left edge or the right edge if necessary, suppose basepoints $z$ and $w$ are in different sides of the stripe. If $z$ is on the left of the stripe, set $v=0$. If $z$ is on the right of the stripe, set $v=1$.

Then suppose $u>1$. When in Case (i) of Lemma \ref{paralem}, rainbows on $e_b$ connect $x^b_i$ to $x^b_{2v+1-i}$ for $i\in[1,v]$, rainbows on $e_t$ connect $x^t_{u+1-i}$ to $x^t_{u-2v+i}$ for $i\in[1,v]$, and stripes connect $x_j^b$ to $x^t_{u+1-j}$ for $j\in[2v+1,u]$. When in Case (ii) of Lemma \ref{paralem}, the setting is obtained by replacing $i$ and $j$ by $u+1-i$ and $u+1-j$, respectively. Without loss of generality, suppose $z$ is in $D_1$, and $w$ is in $D_l$. Note that now basepoints cannot be moved through vertical edges of $C$. Otherwise the rainbows would intersect the vertical edges, which contradicts the definition of the constrained knot. Then we parameterize constrained knots in $L(p,q^\prime)$ by the tuple $(l,u,v)$ for Case $(i)$ and $(l,u,u-v)$ for Case $(ii)$. Since $\beta_1$ is connected, we have $\operatorname{gcd}(u,v)=1$. In summary, the following theorem holds.
\begin{theorem}\label{parameter}
Constrained knots are parameterized by five integers $(p,q,l,u,v)$, where $p>0,q\in [1,p-1],l\in [1,p],u>0,v\in[0,u-1]$, $u$ is odd, and $\operatorname{gcd}(p,q)=\operatorname{gcd}(u,v)=1$. Moreover, $v\in[1,u-1]$ when $u>1$ and $v\in\{0,1\}$ when $u=1$.
\end{theorem}
Note that the parameter $v$ in Theorem \ref{parameter} is different from the integer $v$ in Case (ii) of Lemma \ref{paralem}. Intuitively, for $v\in[1,u-1]$ in the parameterization $(p,q,l,u,v)$ with $u>1$, the number $\min\{v,u-v\}$ is the number of rainbows around a basepoint.

For paramters $(p,q,l,u,v)$, let $C(p,q,l,u,v)$ denote the corresponding constrained knot. When considering the orientation, let $C(p,q,l,u,v)^+$ denote the knot induced by $(T,\alpha_1,\beta_1,z,w)$ and let $C(p,q,l,u,v)^-$ denote the knot induced by $(T,\alpha_1,\beta_1,w,z)$. For $q\neq [1,p-1]$ and $l\neq [1,p]$, consider the integers $q$ and $l$ modulo $p$. If $u>1$ and $v\neq [1,u-1]$, consider the integer $v$ modulo $u$. For $p<0$, let $C(p,q,l,u,v)$ denote $C(-p,-q,l,u,v)$.
\begin{remark}\label{rem: diff}
The knot $C(p,q,l,u,v)$ is in $L(p,q^\prime)$, where $qq^\prime\equiv 1\pmod p$. Though $L(p,q)$ is diffeomorphic to $L(p,q^\prime)$, constrained knots $C(p,q,l,u,v)$ and $C(p,q^\prime,l,u,v)$ is not necessarily equivalent. For example, Theorem \ref{thm2} implies that constrained knots $C(5,2,3,3,1)$ and $C(5,3,3,3,1)$ are not equivalent.
\end{remark}
Then we provide some basic propositions of constrained knots. Also, we indicate the relationship of constrained knots with other families of knots mentioned in Section \ref{s2}.

\begin{proposition}\label{mirr2}
$C(p,-q,l,u,-v)$ is the mirror image of $C(p,q,l,u,v)$ for $u>1$. $C(p,-q,l,1,1)$ is the mirror image of $C(p,q,l,1,0)$.
\end{proposition}
\begin{proof}
    It follows from the vertical reflection of the standard diagram.
\end{proof}
Hence we only consider $C(p,q,l,u,v)$ with $0\le 2v<u$ in the rest of the paper.
\begin{proposition}\label{constrained2-bridge}
$C(1,0,1,u,v)\cong \mathfrak{b}(u,v)$.
\end{proposition}
\begin{proof}
By cutting along $\alpha_1$ and a small circle around $x$ in Figure \ref{f6}, the doubly-pointed diagram of a 2-bridge knot can be shaped into a square. This proposition is clear by comparing this diagram with the new diagram $C$ related to $ C(1,0,1,u,v)$.
\end{proof}
\begin{proposition}\label{alternating}
For any fixed orientations of $\alpha_1$ and $\beta_1$ in the standard diagram of a constrained knot, intersection points $x^b_i$ have alternating signs and adjacent strands of $\beta_1$ in the new diagram $C$ have opposite orientations.
\end{proposition}
\begin{proof}
From a similar observation in the proof of Proposition \ref{constrained2-bridge}, for $C(p,q,l,u,v)$, the curve $\beta_1$ in the new diagram $C$ is same as the curve $\beta$ in the doubly-pointed Heegaard diagram of $\mathfrak{b}(u,v)$. Thus, it suffices to consider the 2-bridge knot $\mathfrak{b}(u,v)$. The Schubert normal form of $\mathfrak{b}(u,v)$  is the union of two dotted horizontal arcs behind the plane and two winding arcs on the plane. Suppose $\gamma$ is one of the winding arc. Then $\beta_1=\partial N(\gamma)$ cuts the plane into two regions, the inside region ${\rm int} N(\ga)$ and the outside region $\mathbb{R}^2-N(\ga)$. Points $x$ and $y$ in Figure \ref{f6} are in different regions and points $x^b_i$ are on the arc connecting $x$ to $y$. Since regions on different sides of $\be_1$ must be different, the arc connecting $x$ to $y$ is cut by $x^b_i$ into pieces that lie in the inside region and the outside region alternately. For each piece of the arc, the endpoints are boundary points of a connected arc in $\be_1$. Thus, signs of $x^b_i$ are alternating. The orientations on strands of $\beta_1$ are induced by signs of $x^b_i$. Hence adjacent strands of $\be_1$ have opposite orientations.
\end{proof}

\begin{proposition}\label{sim1}
For $p,q,q^\p\in\mathbb{Z}$ satisfying $qq^\prime\equiv 1\pmod p$, there are relations:
\begin{enumerate}[(i)]
\item  $S(p, q^\p, k)\cong C(p,q,l,1,0)^+$, where $k-1\equiv (l-1)q^\p \pmod p$;
\item $S(p, q^\p, k)\cong C(p,q,l,1,1)^+$, where $k+1\equiv (l-1)q^\p \pmod p$.
\end{enumerate}
\end{proposition}
\begin{figure}[ht]
\centering 
\includegraphics[width=0.95\textwidth]{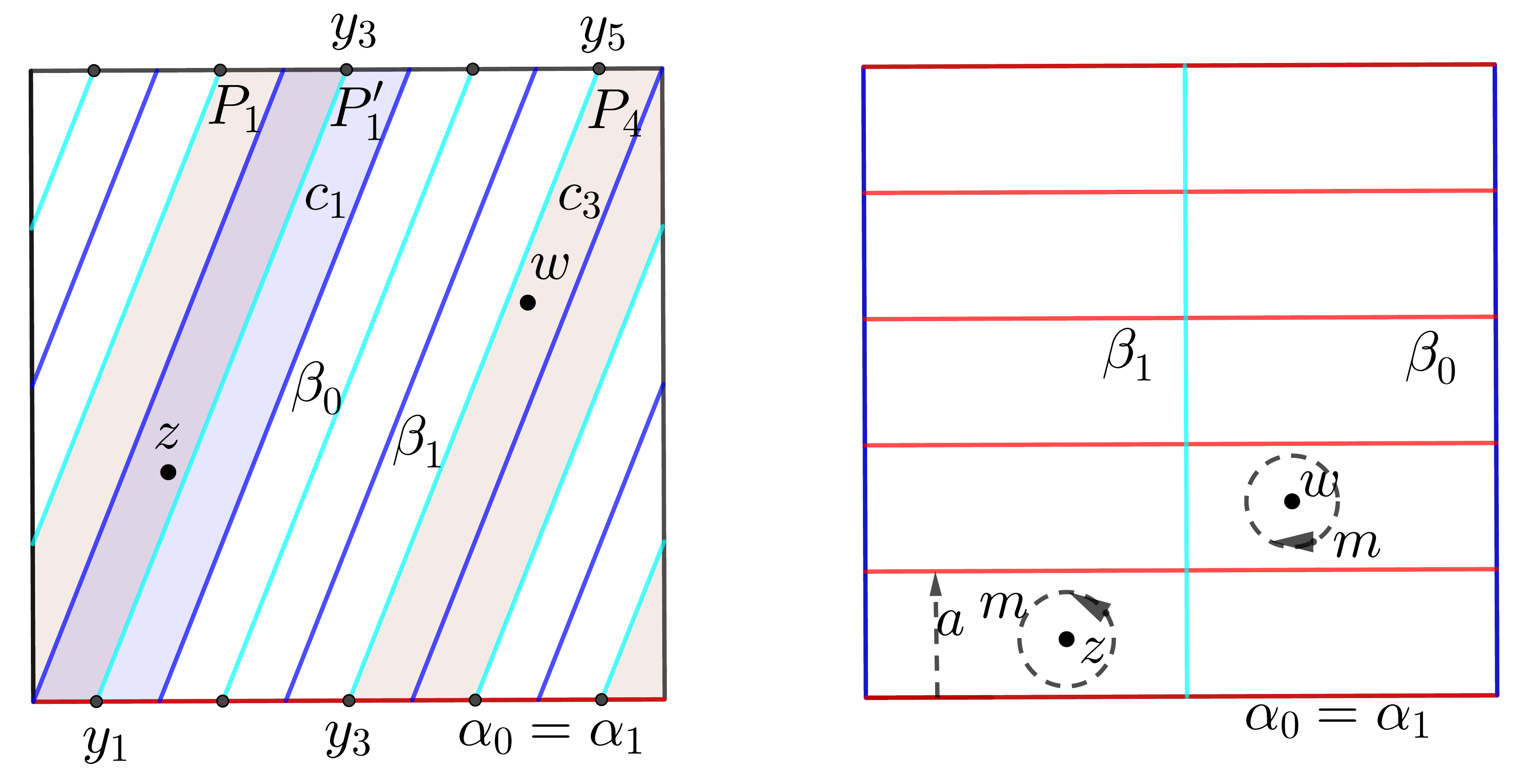} 
\caption{$S(5,2,3)\cong C(5,3,2,1,0)^+$, where regions $P_1,P_4$ and $P_1^\prime$ are indicated by shadow.}
\label{l3}
\end{figure}
\begin{proof}
Consider curves $\alpha=\alpha_0=\alpha_1,\beta_0$ and $\beta_1$ in the definition of a constrained knot. When $u=1$, the curve $\beta_1$ is parallel to $\beta_0$. Consider the new diagram $C$ and regions $D_i$ for $i\in \mathbb{Z}/p\mathbb{Z}$ as in the right subfigure of Figure \ref{l2}. Suppose components of $ T^2-\alpha\cup\beta_1$ are $P_i$ and components of $ T^2-\alpha\cup\beta_0$ are $P_i^\prime$, ordered from left to right as in Figure \ref{l3} so that $z\in P_1\cap P_1^\prime$. Suppose $y_i$ are intersection points of $\alpha$ and $\beta_1$ on the bottom edge of $P_i^\prime$. The strand $c_i=\beta_1\cap P^\prime_i$ connects $y_i$ to $y_{i+q^\p}$, so the strand $\beta_1\cap D_l$ in the new diagram $C$ is $c_{1+(l-1)q^\p}$. When $v=0$, the other basepoint $w$ is in $P_{(l-1)q^\p+2}$, so $k\equiv (l-1)q^\p+1 \pmod p$. When $v=1$, the other basepoint $w$ is in $P_{(l-1)q^\p}$, so $k\equiv (l-1)q^\p-1 \pmod p$.
\end{proof}
\begin{corollary}\label{rela}
For $p,q,q^\p\in\mathbb{Z}$ satisfying $qq^\prime\equiv 1\pmod p$, there are relations:
\begin{enumerate}[(i)]
\item  $C(p,q,l,1,0)\cong C(p,q,l+2q,1,1)$;
\item $C(p,q,l,1,0)^+\cong C(p,q,-2q+2-l,1,0)^-$;
\item $C(p,q,l,1,0)$ is the mirror image of $C(p,-q,l+2q,1,0)\cong C(p,-q,l,1,1)$;
\item  $C(p,q,l,1,0)\cong C(p,q^\prime,q^\prime l-2q^\prime+2,1,0)$;
\item $S(p,q,k)\cong S(p,q^\prime,kq^\prime )\cong C(p,q,k-q+1,1,0)^+$.
\end{enumerate}
\end{corollary}
\begin{proof}
These relations follow from Proposition \ref{sim} and Proposition \ref{sim1}.
\end{proof}
\begin{corollary}\label{core}
The knot $C(p,q,-q+1,1,0)$ is an unknot in a lens space. The knot $C(p,q,l,1,0)$ for $l=1,-2q+1,-q+2,-q$ is a core knot of a lens space.
\end{corollary}
\begin{proof}
The unknot case is obtained by substituting $k=0$ in Case (v) of Corollary \ref{rela}. Note that $S(p,q,0)$ is the unknot: the knot is isotopic to a circle bounding a disk on $T^2$. The core knot cases are obtained by substituting $k=\pm 1,\pm q$ in Case (v) of Corollary \ref{rela}. Note that $S(p,q,q)$ is isotopic to a simple closed curve on $T^2$ that intersects $\al$ once, which also is isotopic to the core curve of the $\al$-handlebody. By Proposition \ref{sim1}, simple knots $S(p,q,\pm q)$ and $S(p,q,\pm 1)$ are also core knots.
\end{proof}
\begin{proposition}\label{homo}
For $K=C(p,q,l,1,0)$, we have a presentation of the homology
\[H_1(E(K);\mathbb{Z})\cong \langle [a],[m]\rangle/(p[a]+k[m])\cong\mathbb{Z}\oplus\mathbb{Z}/\operatorname{gcd}(p,k)\mathbb{Z},\]
where $m$ is the circle in Figure \ref{l3}, $a$ is the core curve of $\alpha_0$-handle and $k\in(0,p]$ satisfies $k-1\equiv (l-1)q^{-1} \pmod p$.
\end{proposition}
\begin{proof}
This follows from Proposition \ref{sim1} and results in \cite[Section 3.3]{Rasmussen2007}.
\end{proof}

\begin{proposition}\label{check}
Suppose $C(p,q,l,u,v)$ is a constrained knot in $L(p,q^\p)$ with $0\le 2v<u$. Let $q_i\in [0,p)$ be integers satisfying $q_i\equiv iq^\prime\pmod p$ and let $k\in[1,p]$ be the integer satisfying $k-1\equiv (l-1)q^\prime \pmod p$. Moreover, let
\[n_1=\#\{i\in[0,l-1]|q_i\in[0,k-1]\}\aand n_2=\#\{i\in[0,l-1]|q_i\in [1,q^\prime-1]\}.\]Then $C(p,q,l,u,v)\cong W(pu-2v(l-1),v,uk-2vn_1,uq^\prime-2vn_2)$.
\end{proposition}
\begin{proof}
The parameters $(p-l+1)u+(l-1)(u-2v)=pu-2v(l-1)$ and $v$ are from counting the numbers of intersection points and rainbows in the standard diagram of a constrained knot, respectively. Suppose $P_i^\p$ are components of $T^2-\al_0\cup\be_0$ in the standard diagram of $L(p,q^\p)$, ordered from left to right so that $z\in P_1^\p$. Similar to the proof of Proposition \ref{sim1}, we know $w\in P_k^\p$. Then the parameter $(k-n_1)u+n_1(u-2v)=uk-2vn_1$ counts the number of stripes between rainbows and the parameter $(q^\p-n_2)u+n_2(u-2v)=uq^\p-2vn_2$ counts the twisting number.
\end{proof}
\begin{proof}[Proof of Theorem \ref{torus}]
For a knot $K$ in a lens space with Seifert fibered complement, any Dehn surgery other than the one along homological longitude gives a Seifert fibered space. By discussion in \cite[Section 5]{Rasmussen2017}, all orineted Seifert fibered spaces over $\mathbb{RP}^2$ are L-spaces and the classification of L-spaces over $S^2$ is given by \cite[Theorem 5.1]{Rasmussen2017}. Moreover, no higher genus Seifert fibered spaces are L-spaces. The classification in \cite[Theorem 5.1]{Rasmussen2017} indicates there are at least two Dehn fillings on the knot complement are L-spaces, \textit{i.e.} $K$ is always an L-space knot. By Proposition \ref{check}, we can transform standard diagrams of constrained knots into $(1,1)$ diagrams. By Proposition \ref{alternating} and Theorem \ref{lspace}, a constrained knot is an L-space knot if and only if $(u,v)=(1,0),(1,1)$ or $u>1,v=\pm 1$.
\end{proof}

\begin{proof}[Proof of Proposition \ref{char}]
The necessary part of the proposition follows directly from the definition of constrained knots: the intersection points of $\al_1$ and $\be_1$ between two consecutive intersection points of $\al_0$ and $\be_0$ correspond the same spin$^c$ structure on the lens space, where $\al_1,\be_1,\al_0,\be_0$ are curves in the standard diagrams of the constrained knot and the lens space.

Then we prove the sufficient part of this proposition. For simplicity, intervals are considered in $\mathbb{Z}/p\mathbb{Z}$. In particular, let $(p_k,p_1]$ denote $(0,p_1]\cup (p_k,p]$. Consider intersection points $x_i$ for $i\in[1,p]$ as shown in Figure \ref{11}.

Firstly, spin$^c$ structures $\mathfrak{s}_i$ are equal for all $i\in[r+1,r+2q]$. Indeed, for $i\in [1,q]$, the points $x_{r+i}$ and $x_{r+2q+1-i}$ are boundary points of a rainbow, \textit{i.e.} there is a holomorphic disk connecting $x_{r+i}$ to $x_{r+2q+1-i}$. Thus $\mathfrak{s}_{r+i}=\mathfrak{s}_{r+2q+1-i}$. If $q=1$, this assertion is trivial. If $q>1$ and the assertion did not hold, then there must be an integer $q_0$ and two spin$^c$ structures $\mathfrak{s}_A,\mathfrak{s}_B$ so that $\mathfrak{s}_i=\mathfrak{s}_A$ for all $i\in[r+q_0,r+2q+1-q_0]$ and $\mathfrak{s}_j=\mathfrak{s}_B$ for all $j\notin[r+q_0,r+2q+1-q_0]$, which implies $a=2$. Since spin$^c$ structures of two boundary points of a stripe are different, for all $i\in [2q+1-s,p-s]$, spin$^c$ structures $\mathfrak{s}_i$ are different from $\mathfrak{s}_B$. Thus $\mathfrak{s}_i=\mathfrak{s}_A$ for all $i\in [2q+1-s,p-s]$. For $i\in [1,q]$, points $x_{i-s}$ and $x_{2q+1-i-s}$ are boundary points of a rainbow, so $\mathfrak{s}_{i-s}=\mathfrak{s}_{2q+1-i-s}$. Since there are $2q_0$ points corresponding to $\mathfrak{s}_A$, integers $q_0$ should satisfy the inequality $2q_0>p-2q$. For $i\in [q+q_0-p/2,q]$, points $x_{i-s}$ and $x_{2q+1-i-s}$ correspond to $\mathfrak{s}_B$. In particular, points $x_{r+1}$ and $x_{r+2q}$ are identified with $x_{2q+1-i_0-s}$ and $x_{i_0-s}$ for $i_0=q+q_0-p/2$, respectively. Let $R_1$ be the rainbow with boundary points $x_{r+1}$ and $x_{r+2q}$, and let $R_2$ be the rainbow with boundary points $x_{2q+1-i_0-s}$ and $x_{i_0-s}$. The union of $R_1$ and $R_2$ becomes a component of $\beta$, which contradicts the assumption that $\beta$ only has one component.

By a similar proof, we can show that the spin$^c$ structures $\mathfrak{s}_i$ are equal for all $i\in[1-s,2q-s]$. From this discussion, for any $i\in[1,k]$, we have \[p_i\neq r+1,r+2,\dots,r+2q-1,1-s,2-s,\dots,2q-1-s.\]

Suppose $y_i$ for $i\in[1,k]$ are points on $\alpha$ between $x_{p_i}$ and $x_{p_i+1}$.  If $p_i\neq r,r+2q,p$, then $p_i$ and $p_i+1$ must be boundary points of two successive stripes. Suppose $x_{j}$ and $x_{j+1}$ are the other boundary points of these stripes, respectively. There must be a point $y_j$ between $x_j$ and $x_{j+1}$ because $\mathfrak{s}_j-\mathfrak{s}_{j+1}=\mathfrak{s}_{p_i}-\mathfrak{s}_{p_i+1}\neq 0$. Let $b_i$ be a strand connecting $y_i$ to $y_j$ which is disjoint from $\beta$.

Suppose $p_i=p$. If $r\neq 0$ and $p-2q-r\neq 0$ there are stripes connecting $\mathfrak{s}_p$ to $\mathfrak{s}_{p-s}$ and connecting $\mathfrak{s}_1$ to $\mathfrak{s}_{2q+1-s}$, respectively. Thus $\mathfrak{s}_{p-s}-\mathfrak{s}_{2q+1-s}=\mathfrak{s}_p-\mathfrak{s}_{1}\neq 0$. There is a point $y_j$ either between $x_{p-s}$ and $x_{1-s}$, or between $x_{2q-s}$ and $x_{2q+1-s}$ for some $j$. Only one case will happen because the number of intersection points corresponding to any fixed spin$^c$ structure is odd. Let $b_i$ be a strand connecting $y_i$ to $y_j$ which is disjoint from $\beta$. If either $r=0$ or $p-2q-r=0$, by choosing different stripes, the inequality $\mathfrak{s}_{p-s}-\mathfrak{s}_{2q+1-s}\neq 0$ still holds. The point $y_j$ and the strand $b_i$ can also be found. By a similar argument, this is also true for $p_i=r,r+2q$.

Let $\beta_0$ be the union of $b_i$. Without considering basepoints, $\beta_0$ is isotopic to $\beta$. Thus, it has only one component. Finally, the curves $\beta_0,\alpha,\beta$ can be identified with $\beta_0,\alpha_1,\beta_1$ in the definition of a constrained knot. Thus, we conclude that the (1,1) knot is a constrained knot.
\end{proof}
\section{Knot Floer homology}\label{s4}
Heegaard Floer homology is an invariant for closed 3-manifolds discovered by Oszv\'{a}th and Szab\'{o} \cite{Ozsvath2004b,Ozsvath2004c}. It is generalized to knot Flor homology \cite{Ozsvath2004,Rasmussen2003}, sutured Floer homology \cite{juhasz2006holomorphic}, bordered Floer homology \cite{Lipshitz2008} and immersed curves for manifolds with torus boundary \cite{Hanselman2016,Hanselman2018}. See \cite[Section 3]{Rasmussen2007} for a brief review of knot Floer homology for rationally null-homologous knots. See also \cite{Ozsvath2011}.

Throughout this section, suppose $K=C(p,q,l,u,v)$ is a constrained knot in $Y=L(p,q^\prime)$, where $qq^\p\equiv 1\pmod p$. Write $H_1=H_1(E(K);\mathbb{Z})$ and $\widehat{HFK}(K)=\widehat{HFK}(Y,K)$ for short. For any homogeneous element $x\in \widehat{HFK}(K)$, let $\operatorname{gr}(x)\in H_1$ be the Alexander grading of $x$ mentioned in the introduction. Note that the Alexander grading is well-defined up to a global grading shift  \cite{Friedl2009}, \textit{i.e.} up to multiplication by an element in $H_1$. However, the difference $\operatorname{gr}(x)-\operatorname{gr}(y)$ for two homogeneous elements $x$ and $y$ is always well-defined. This difference can be calculated explicitly by the doubly-pointed Heegaard diagram of the knot by the approach in \cite[Section 3.3]{Rasmussen2007}.

Consider the group ring $\mathbb{Z}[H_1]$. Two elements $f_1$ and $f_2$ in $\mathbb{Z}[H_1]$ are \textbf{equivalent}, denoted by $f_1\sim f_2$, if there exists an element $g\in \pm H_1$ so that $f_1=gf_2$. For any element $h\in H_1$, there is a grading summand $\widehat{HFK}(K,h)$ of $\widehat{HFK}(K)$ as in (\ref{alexgrading}). There is also a relative $\mathbb{Z}/2$ grading on $\widehat{HFK}(K)$ induced by signs of the intersection numbers in the Heegaard diagram (\textit{c.f.} \cite[Section 2.4]{Friedl2009}) and related to the modulo 2 Maslov grading on $\widehat{HFK}(K,\mathfrak{s})$. This grading respects the Alexander grading and induces a $\mathbb{Z}/2$ grading on $\widehat{HFK}(K,h)$. Then the Euler characteristic $\chi(\widehat{HFK}(K,h))$ is well-defined up to sign. We can consider the (graded) Euler characteristic of $\widehat{HFK}(K)$: $$\begin{aligned}\chi(\widehat{HFK}(K))&= \sum_{h\in H_1}\chi(\widehat{HFK}(K,h))\cdot h\\&=\sum_{h\in H_1}({\rm rk}\widehat{HFK}_{even}(K,h)-{\rm rk}\widehat{HFK}_{odd}(K,h))\cdot h.\end{aligned}$$From the above discussion, we know $\chi(\widehat{HFK}(K))$ is an element in $\mathbb{Z}[H_1]$ up to equivalence. For $\mathfrak{s}\in {\rm Spin}^c(Y)$, we consider $\widehat{HFK}(K,\mathfrak{s})$ as a subgroup of $\widehat{HFK}(K)$ so that it also has an $H_1$-grading and $\chi(\widehat{HFK}(K,\mathfrak{s}))$ is also an element in $\mathbb{Z}[H_1]$ up to equivalence.

For a constrained knot $K$, we will show $\widehat{HFK}(K)$ totally depends on $\chi(\widehat{HFK}(K))$. Explicitly this means that, for any $h\in H_1$, the dimension of $\widehat{HFK}(K,h)$ is the same as the absolute value $|\chi(\widehat{HFK}(K,h)|$.

As shown in Figure \ref{l2} and Figure \ref{gr}, suppose $e^j$ is the top edge of $D_j$ and $x_i^j$ is the intersection point of $e^j$ and $\beta_1$ for $j\in\mathbb{Z}/p\mathbb{Z},i\in[1,u(j)]$. Let $x_{\text{middle}}^j=x^j_{(u(j)+1)/2}$ be middle points. It is clear that $\mathfrak{s}_z(x^{j_1}_{i_1})=\mathfrak{s}_z(x^{j_2}_{i_2})$ if and only if $j_1=j_2$. For any integer $j\in[1,p]$, define $\mathfrak{s}_j=\mathfrak{s}_z(x^j_{middle})\in {\rm Spin}^c(Y).$
\begin{figure}[ht]
\centering 
\includegraphics[width=0.45\textwidth]{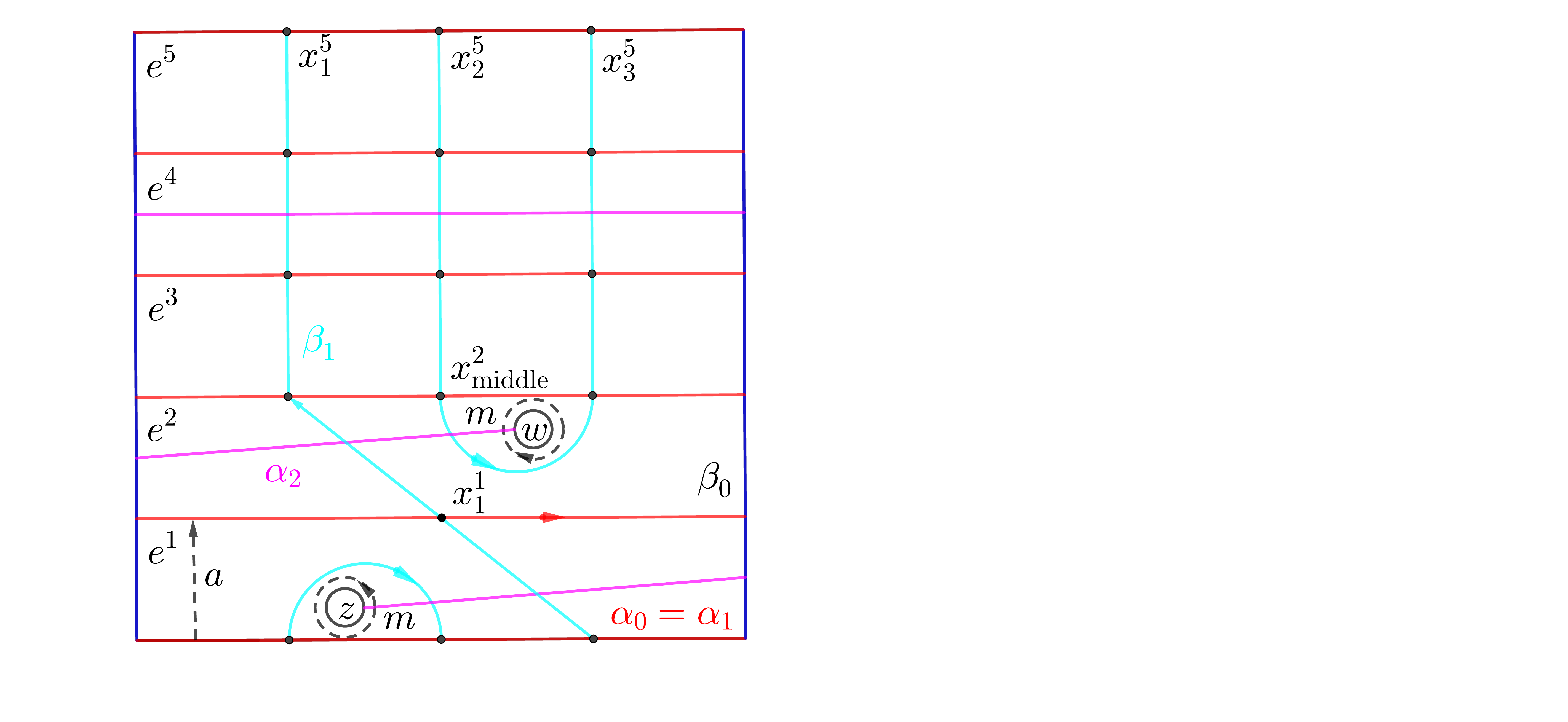} 
\caption{Heegaard diagram of $E(C(5,3,2,3,1))$.}
\label{gr}
\end{figure}
\begin{lemma}\label{lam0}
For $K=C(p,q,l,u,v)$ with $u>2v>0$, suppose $k\in(0,p]$ is the integer satisfying $k-1\equiv (l-1)q^{-1} \pmod p$. Define $$k^\p=\begin{cases}k-2 & v\text{ odd,}\\ k & v\text{ even.}
\end{cases}$$Suppose $d=\operatorname{gcd}(p,k^\prime)$. Then there is a presentation of the homology $H_1$: $$H_1=H_1(E(K);\mathbb{Z})\cong\langle [a],[m]\rangle/(p[a]+k^\prime [m])\cong\mathbb{Z}\oplus\mathbb{Z}/d\mathbb{Z},$$where $m$ is the circle in Figure \ref{gr} and $a$ is the core curve of $\alpha_0$-handle.
\end{lemma}
\begin{proof}
Suppose $\beta_1$ is oriented so that the orientation of the middle stripe is from bottom to top. Let $[\beta_1(p,q,l,u,v)]$ denote the homology class of $\beta_1$ corresponding to $C(p,q,l,u,v)$. By Proposition \ref{alternating}, orientations of rainbows around a basepoint are alternating. Note that moving all rainbows of $\be_1$ across basepoints gives the diagram of the simple knot $C(p,q,l,1,0)$. Then $$\begin{cases}
[\beta_1(p,q,l,u,v)]+2[m]=[\beta_1(p,q,l,1,0)] & v \text{ odd,}\\
[\beta_1(p,q,l,u,v)]=[\beta_1(p,q,l,1,0)] & v \text{ even.}
\end{cases}$$
Then this proposition follows from Proposition \ref{homo}. Note that $[a]$ and $[m]$ correspond to core curves of $\al_1$ and $\al_2$ and the relation in the presentation of $H_1$ corresponds to algebraic intersection numbers $\al_1\cdot \be$ and $\al_2\cdot \be$; see Section \ref{s6} for the approach to obtain a presentation of $\pi_1(E(K))$ and note that $H_1$ is the abelianization of $\pi_1(E(K))$.
\end{proof}
\begin{lemma}[{Proposition \ref{thin}}]\label{lam1}
For $K=C(p,q,l,u,v)$ with $u>2v\ge 0$, suppose $H_1$ is presented as in Lemma \ref{lam0}. For any integer $j\in[1,p]$, let $\mathfrak{s}_j=\mathfrak{s}_z(x^j_{middle})$ for intersection points $x^j_{middle}$ in Figure \ref{gr}. Then for any $j$, the group $\widehat{HFK}(K,\mathfrak{s}_j)$ is determined by its Euler characteristic.

Moreover, suppose integers $u^\p$ and $v^\p$ satisfy $u^\prime=u-2v$ and $v^\prime \equiv  v \pmod {u^\prime}$. Let $\Delta_1(t)$ and $\Delta_2(t)$ be Alexander polynomials of $\mathfrak{b}(u,v)$ and $\mathfrak{b}(u^\p,v^\p)$, respectively. Then
$$\chi(\widehat{HFK}(K,\mathfrak{s}_j))\sim\begin{cases}\Delta_{1}([m]) & j\in[l,p],\\
\Delta_{2}([m]) & j\in[1,l-1].
\end{cases}$$
\end{lemma}
\begin{proof}
For $j\in [1,p]$, consider the edge $e^j$ and the intersection numbers $x^j_i$ of $e^j$ and $\be_1$ in the diagram $C$. Suppose $(e^j)^\prime$ is the curve obtained by identifying two endpoints of $e^j$. For $j\in[l,p]$, the diagram $(T^2,(e^j)^\prime,\beta_1,z,w)$ is the same as the diagram of $K_1=\mathfrak{b}(u,v)$. For $j\in[1,l-1]$, by Case (iii) of Lemma \ref{lamuv}, the diagram $(T^2,(e^j)^\prime,\beta_1,z,w)$ is isotopic to the diagram of $K_2=\mathfrak{b}(u^\p,v^\p)$. For readers' convenience, we sketch the proof.

The fact that $u^\p=u-2v$ follows directly from the number of intersection points of $(e^j)^\p$ and $\be_1$, which is the same as the number of stripes. Then we consdier $v^\p$. Let $D=N(x_{middle}^p)$ be the neighborhood of $x_{middle}^p$ so that $D$ contains all rainbows. Consider the isotopy obtained by rotating $D$ counterclockwise. If $v>u^\p$, after rotation, the resulting diagram has $v-u^\p$ rainbows. The formula for $v^\p$ follows by induction.

2-bridge knots are alternating, hence are thin \cite{Ozsvath2003alternating} in the sense of Definition \ref{defn: thin}. By comparing the number of generators of $\widehat{CFK}(K_i)$ for $i=1,2$ from $(T^2,(e^j)^\prime,\beta_1,z,w)$ and the dimension of $\widehat{HFK}(K_i)$ from the Alexander polynomial (\textit{c.f.} Proposition \ref{alexander}), we know there is no differential on $\widehat{CFK}(K_i)$. This fact can also be shown by a direct calculation following the method in \cite{Goda2005}. Thus, the constrained knot $K$ is also thin and there is no differential on $\widehat{CFK}(K,\mathfrak{s}_j)$. In particular, the group $\widehat{CFK}(K,\mathfrak{s}_j)$ is determined by its Euler characteristic.

As discussed at the start of this section, the characteristic $\chi(\widehat{HFK}(K,\mathfrak{s}_j))$ is an element in $H_1$ up to equivalence. Similar to the proof of \cite[Lemma 3.4]{Rasmussen2002}, for $j\in [l,p]$, we have \[\operatorname{gr}(x_{i+1}^j)-\operatorname{gr}(x_i^j)=[m]^{(-1)^{\lfloor\frac{iv}{u}\rfloor}}.\]For $j\in[1,l-1]$, just replace $u$ and $v$ by $u^\p$ and $v^\p$ in the above formula, respectively. Comparing the formula of the Alexander polynomial in Proposition \ref{alexander}, we conclude the formula of $\chi(\widehat{HFK}(K,\mathfrak{s}_j))$.
\end{proof}
\begin{lemma}\label{lam2}
Consider integers $k,k^\prime$ and the presentation of $H_1$ as in Lemma \ref{lam0}.

For $j\neq 0,l-1$,
$\operatorname{gr}(x_{\text{middle}}^{j+1})-\operatorname{gr}(x_{\text{middle}}^j)=
\begin {cases}
 [a]+[m]&\text{if }jq^{-1}\equiv 1,\dots, k-2 \pmod p\\
[a]&\text{otherwise.}
\end {cases}$

For $l\neq 1$ and $j=0,l-1$, $\operatorname{gr}(x_{\text{middle}}^{j+1})-\operatorname{gr}(x_{\text{middle}}^j)=
\begin {cases}
 [a]+[m]&v \text{ even}\\
[a]& v\text{ odd.}
\end {cases}$

For $l=1$, $\operatorname{gr}(x_{\text{middle}}^{j+1})-\operatorname{gr}(x_{\text{middle}}^j)=
\begin {cases}
 [a]+[m]&v \text{ even}\\
[a]-[m]& v\text{ odd.}
\end {cases}$
\end{lemma}
\begin{proof}
For simple knots, the proof is based on Fox calculus (\textit{c.f.} \cite[Proposition 6.1]{Rasmussen2007}). For a general constrained knot and $j\neq 0,l-1$, the proof in \cite{Rasmussen2007} still works because orientations of strands are alternating. The differences of gradings for $j=0$ and $j=l-1$ are the same because $z$ and $w$ are symmetric by rotation. The proof follows from the following equations \[\sum_{j=0}^{p-1}\operatorname{gr}(x_{\text{middle}}^{j+1})-\operatorname{gr}(x_{\text{middle}}^j)=0\in H_1\aand p [a]+k^\prime [m]=0\in H_1.\]
\end{proof}
\begin{corollary}\label{coro1}Suppose $K=C(p,q,l,u,v)$ is a constrained knot in $Y=L(p,q^\p)$, where $qq^\p\equiv 1\pmod p$. For any integer $j\in[1,p]$, let $\mathfrak{s}_j=\mathfrak{s}_z(x^j_{middle})\in {\rm Spin}^c(Y)$ for intersection points $x^j_{middle}$ in Figure \ref{gr}. Then $\mathfrak{s}_{j+1}-\mathfrak{s}_{j}$ only depends on $p$ and $q$.
\end{corollary}
\begin{proof}
By the map $H_1(E(K);\mathbb{Z})/([m])\to H_1(Y;\mathbb{Z})$, the grading difference $\operatorname{gr}(x_{\text{middle}}^{j+1})-\operatorname{gr}(x_{\text{middle}}^j)$ is mapped to $\mathfrak{s}_{j+1}-\mathfrak{s}_j$, which only depends on the image of $[a]$.
\end{proof}
\begin{lemma}\label{lam3}
Consider $\mathfrak{b}(u,v)$ and $\mathfrak{b}(u^\prime,v^\prime)$ as in Lemma \ref{lam1}. Then\[\sigma(\mathfrak{b}(u^\prime,v^\prime))=\begin {cases}
\sigma(\mathfrak{b}(u,v))&v \text{ even}\\
\sigma(\mathfrak{b}(u,v))+2& v\text{ odd.}
\end {cases}.\]
\end{lemma}
\begin{proof}
Consider standard presentations of 2-bridge knots in Subsection \ref{s2b}. It is easy to see $\mathfrak{b}(u,v)$ and $\mathfrak{b}(u^\p,v^\p)$ form two knots in the skein relation. By the skein relation formula of signatures of knots, we can conclude this lemma. Moreover, we provide another proof based on the Alexander grading as follows.

By the algorithm of the Alexander grading, we have $$\operatorname{gr}(x_{u^\prime}^{1})-\operatorname{gr}(x_{u}^0)=
[a]+[m].$$From the rotation symmetry and the formula of the signature in Proposition \ref{alexander}, \[\operatorname{gr}(x_{u}^{0})-\operatorname{gr}(x_{\text{middle}}^0)=
\operatorname{gr}(x_\text{middle}^{0})-\operatorname{gr}(x_{1}^0)=\frac{\sigma(\mathfrak{b}(u,v))}{2}[m],\] \[\operatorname{gr}(x_{u^\prime}^{1})-\operatorname{gr}(x_{\text{middle}}^1)=
\operatorname{gr}(x_\text{middle}^{1})-\operatorname{gr}(x_{1}^1)=\frac{\sigma(\mathfrak{b}(u^\prime,v^\prime))}{2}[m].\]Then this lemma follows from these equations and Lemma \ref{lam2}.
\end{proof}
\begin{theorem}\label{thm1}
For a constrained knot $K=C(p,q,l,u,v)$, consider the Alexander polynomials $\Delta_1(t)$ and $\Delta_2(t)$ in Lemma \ref{lam1}. Then $\widehat{HFK}(K)$ is determined by its Euler characteristic, which is calculated by the following formula: \begin{equation}\label{eq1}\chi(\widehat{HFK}(K))=\Delta_1([m])\sum_{j=l}^p\operatorname{gr}(x^j_{\text{middle}})+\Delta_2([m])\sum_{j=1}^{l-1}\operatorname{gr}(x^j_{\text{middle}})\end{equation}
\end{theorem}
\begin{proof}
By the result of Lemma \ref{lam1}, we only need to consider the (relative) signs of intersection points corresponding to different spin$^c$ structures. By Proposition \ref{alternating}, signs of intersection points $x^j_i$ for fixed $j$ are alternating. Since $u$ and $u^\p=u-2v$ are odd, signs of $x^j_1$ and $x^j_{u(j)}$ are the same, where $u(j)$ is either $u$ or $u^\p$ by Lemma \ref{lam1}. From the diagram, signs of $x^j_{u(j)}$ for $j\in[0,l]$ are the same and signs of $x^k_{1}$ for $k\in[l,p]$ are the same. Thus, we obtain Formula (\ref{eq1}).
\end{proof}
All terms in Formula \ref{eq1} can be calculated by Lemma \ref{lam2} and Lemma \ref{lam3}. Thus, we obtain an algorithm of $\widehat{HFK}(K)$ for a constrained knot $K$.

Let signs of $x^j_1$ be positive. The Alexander grading can be fixed by the global symmetry, \textit{i.e.}, we consdier the absolute Alexander grading. Note that the global symmetry corresponds to switch the roles of $z$ and $w$, which is equivalent to a rotation of the standard diagram of a constrained knot. Then we have $$\operatorname{gr}(x^j_{\text{middle}})=-\operatorname{gr}(x^{2l-j}_{\text{middle}}) \text{ for any }j.$$In this assumption, we may use square roots of elements in $H_1$ to achieve the symmetry, and the Euler characteristic $\chi(\widehat{HFK}(K))$ is a well-defined element in $(\frac{1}{2}\mathbb{Z})[H_1]$ for this case. The group $\widehat{HFK}(K)$ with the Alexander grading fixed as above is called the \textbf{canonical representative}.
\begin{proof}[{Proof of the necessary part of Theorem \ref{thm2}}]
For $i=1,2$, if $K_i=C(p_i,q_i,l_i,u_i,v_i)$ are equivalent, then $p_1=p_2=p$ and $q_1\equiv q_2^{\pm 1} \pmod{p}$ by the classification of lens spaces \cite{Brody1960}. Suppose $Y$ is the lens space containing $K_1$ and $K_2$. For $i=1,2$, consider $(u^\p_i,v^\p_i)$ as in Lemma \ref{lam1}. By comparing knot Floer homologies, we know $l_1=l_2$,
\[\begin{aligned}
u_1&=|\Delta_\mathfrak{b}(u_1,v_1)(-1)|=|\Delta_\mathfrak{b}(u_2,v_2)(-1)|=u_2\text{, and}\\
u_1-2v_1&=|\Delta_{\mathfrak{b}(u_1^\prime,v_1^\prime)}(-1)|=|\Delta_{\mathfrak{b}(u^\prime_2,v^\prime_2)}(-1)|=u_2-2v_2.
\end{aligned}\]
Thus, we have $(l_1,u_1,v_1)=(l_2,u_2,v_2)=(l,u,v)$. Moreover, the sets of spin$^c$ structures corresponding to $\mathfrak{b}(u,v)$ for two constrained knots should be the same. By Corollary \ref{coro1}, it suffices to consider simple knots. Let $\mathfrak{s}_j^i$ be spin$^c$ structures related to diagrams of $K_i$ for $i=1,2$. As traveling along $\alpha_1$ of $K_1$, middle points are in the order $$x^0_\text{middle},x^{q_1}_\text{middle},\dots,x^{(p-1)q_1}_\text{middle}.$$
Thus, we have $$\mathfrak{s}_{q_1+j}^1-\mathfrak{s}_{j}^1=\mathfrak{s}^2_{j+1}-\mathfrak{s}^2_{j}\in H^2(Y;\mathbb{Z}).$$

Then the following sets are the same:$$\{\mathfrak{s}_j^1-\mathfrak{s}_0^1+\mathfrak{s}_j^1-\mathfrak{s}_1^1|j\in[l,p]\}\text{ and }\{\mathfrak{s}_j^2-\mathfrak{s}_0^2+\mathfrak{s}_j^2-\mathfrak{s}_1^2|j\in[l,p]\}.$$Equvalently, numbers in  $\{0,q_1,\dots,(p-l)q_1\}$ should be consecutive congruence classes modulo $p$. By the following proposition, this can only happen when $l\in\{2,p\}$.
\end{proof}
\begin{proposition}
Suppose that integers $p,q,k$ satisfy $1<q<p-1,\operatorname{gcd}(p,q)=1$ and $0\le k<p-1$. Then there exists an integer $x$ so that the sets $\{x,x+1,\dots,x+k\}$ and $\{0,q,\dots,kq\}$ can be identified modulo $p$ if and only if $k=0,p-2$.
\end{proposition}
\begin{proof}
If $k=0,p-2$, this proposition is trivial. Suppose $k\neq 0,p-2$. Consider elements in sets are in $\mathbb{Z}/p\mathbb{Z}$ in this proof. Define $$T=\{0,1,\dots,p-1\},S^q=\{0,q,\dots,kq\}\aand S_{x}=\{x,x+1,\dots,x+k\}.$$

Suppose $S^q=S_x$ for some $x$ and $n=\lfloor p/q\rfloor\ge 2$. If $k\le n$, then the set $S^q$ cannot be identified with $S_x$. Thus $k\ge n+1$ and $\{0,q,\dots,nq\}\subset S^q=S_x$. Suppose $T-S_x=\{y,y+1,\dots, y+p-k-2\}$, where $y=x+k+1$. Since $(T-S_x)\cap S^q$ is empty by assumption, the set $T-S_x$ must be either a subset of $\{iq+1,iq+2,\dots,(i+1)q-1\}$ for some integer $i\in [0,n-1]$ or a subset of  $\{nq+1,nq+2,\dots,p-1\}$. If $q=2$, then $k=0$, which contradicts the assumption. Suppose $q>2$. Since $k\neq 0,p-2$, we know $y,y+1\in T-S_x$.

If the first case happens with $i=0$, then we know $\{q+1,q+2,\dots,2q-1\}\subset S_x=S_q$ because $n\ge 2$. Since $y+q,y+1+q\in \{q+1,q+2,\dots,2q-1\}$, there exist different integers $k_0,k_1\in [1,k]$, so that $$y+q\equiv k_0q~{\rm and}~y+1+q\equiv k_1q\pmod p.$$If $k_0>k_1$, then $k_0-1 \in[1,k-1]$ and $y= (k_0-1)q\in S^q$. If $k_0<k_1$, then $k_1-1\in [1,k-1]$ and  $y+1= (k_1-1)q\in S^q$. Both contradict the assumption.

If the first case happens with $i>0$ or the second case happens, then there exist different integers $k_0,k_1\in [1,k]$, so that $$y-q\equiv k_0q~{\rm and}~y+1-q\equiv k_1q\pmod p.$$ If $k_0>k_1$, then $k_1+1\in [2,k]$ and $y+1= (k_1+1)q\in S^q$. If $k_0<k_1$, then $k_0+1\in[2,k]$ and $y= (k_0+1)q\in S^q$. Both contradict the assumption.

In summary, for $p>2q$, there is a contradiction if $k\neq 0,p-2$. If $p<2q$ and $S^q=S_x$, then we consider \[S^{p-q}=\{-x,-x-1,\dots,-x-k\}=S_{-x-k}.\]Note that $p>2(p-q)$. From a similar discussion, there is also a contradiction.
\end{proof}
In the rest of this section, we indicate how to draw the curve invariant \cite{Hanselman2016,Hanselman2018} of the knot complement of a constrained knot. Readers who are not familiar with the curve invariant can safely skip the following discussion since there is no further result in this paper relying on it.

Suppose $K=C(p,q,l,u,v)$ is a constrained knot in $Y=L(p,q^\prime)$, where $qq^\prime\equiv 1\pmod p$. Let $M=E(K)$. From the standard diagram of the constrained knot, we know $[K]=k^\prime[b]\in H_1(Y;\mathbb{Z})$, where $b$ is the core curve of $\beta_0$-handle and $k^\p$ is the integer in Lemma \ref{lam0}. Since $K$ is thin, the curve invariant $\widehat{HF}(M)$ can be drawn as follows.

The curve invariant can be decomposed with respect to $\operatorname{Spin}^c(M)$, which is affine over $H^2(M;\mathbb{Z})$. By Poinc\'{a}re duality and the long exact sequence from $(M,\partial M)$, we know\[|H^2(M;\mathbb{Z})|=|H_1(M,\partial M;\mathbb{Z})|=|H_1(M;\mathbb{Z})/\operatorname{Im}(H_1(\partial M;\mathbb{Z}))|=|\operatorname{Tors}H_1(M;\mathbb{Z})|.\]For simplicity, suppose $H_1(M;\mathbb{Z})\cong \mathbb{Z}$. Then $|\operatorname{Spin}^c(M)|=1$ and $\operatorname{gcd}(p,k^\prime)=1$.

The curve invariant can be lifted to the universal cover $\mathbb{R}^2$ of $\partial M$. Suppose the basis is $([l^*],-[m^*])$, where the homological meridian $m^*$ (\textit{c.f.} Section \ref{s2}) is chosen so that $[m]=p[m^*]-k_0 [l^*]$ for some $k_0\in [0,p)$. Consider parallel lines with the slope $p/k_0$ away from the basepoint on $M$. They cut $\mathbb{R}^2$ into bands. Suppose that lifts of the basepoint are integer points and lie on a line with the slope $p/k_0$ in each band. Since $\widehat{HF}(Y,\mathfrak{s})\cong \mathbb{Z}$ for any $\mathfrak{s}\in {\rm Spin}^c(Y)$, the curve invariant intersects each line once.

Based on the proof of Lemma \ref{lam1}, the chain complex $\widehat{CFK}(K,\mathfrak{s})$ for any $\mathfrak{s}\in {\rm Spin}^c(Y)$ is similar to the chain complex related to a 2-bridge knot. Moreover, from the relation of the standard diagram of $K$ and the Heegaard diagram of a 2-bridge knot, the minus version of the knot Floer chain complex $CFK^-(K,\mathfrak{s})$ is also related to $CFK^-$ of a 2-bridge knot. From the results in \cite[Section 3]{Petkova2009} about thin complexes and the results in \cite[Section 4]{Hanselman2018} about how to draw the curve invariant from $CFK^-$, the part of the curve invariant of $K$ in a band is the union of some purple figure-8 curves and a distinguished red arc as shown in Figure \ref{curve}, which totally depends on the Alexander polynomial and the signature of the related 2-bridge knot.
\begin{figure}[ht]
\centering 
\includegraphics[width=0.3\textwidth]{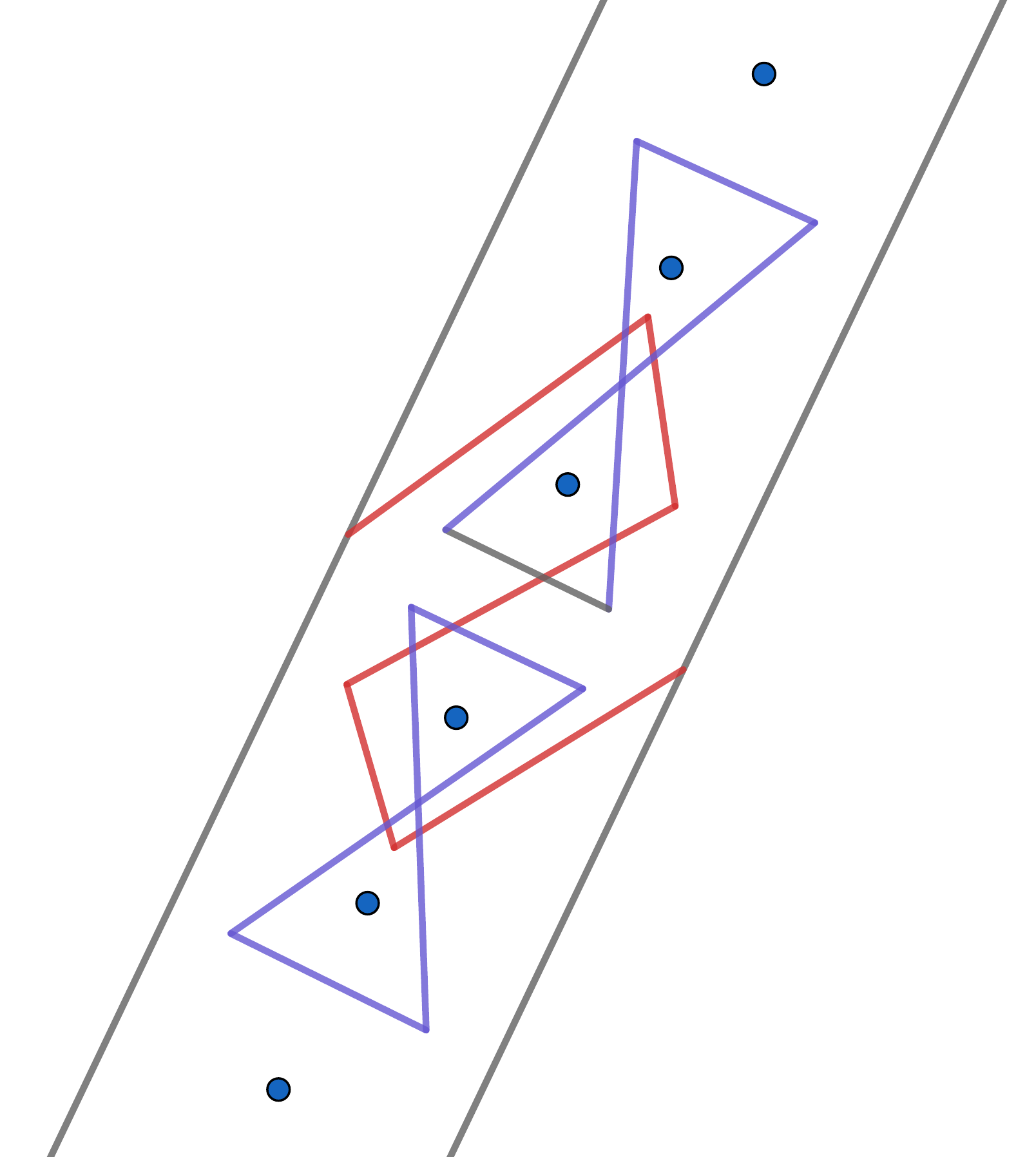} 
\caption{Part of the curve invariant of $C(p,q,l,11,3)$.}
\label{curve}
\end{figure}
\begin{lemma}
Suppose $H_1(M;\mathbb{Z})\cong \mathbb{Z}$ and consider $k_0,k^\prime$ as above. Suppose $a,b$ are core curves of $\alpha_0,\beta_0$ handles corresponding to the standard diagram of $Y=L(p,q^\prime)$. Then $k_0q(k^\prime)^2\equiv -1\pmod p$. Hence $k_0$ is determined by $k^\prime$.
\end{lemma}
\begin{proof}
The homology $H_1(M;\mathbb{Z})$ is generated by $[m^*]$. Let $\tilde{m}^*$ denote the image of $[m^*]$ in $H_1(Y;\mathbb{Z})$. By Lemma \ref{lam0}, $[a]=-k^\prime\tilde{m}^*$. The relation $[b]=q[a]$ implies $[K]=-q(k^\prime)^2\tilde{m}^*$. Then a lift of $[K]$ in $H_1(T^2;\mathbb{Z})$ equals to $-q(k^\prime)^2 [m^*]+k_1[l^*]$ for some $k_1$. Since $l$ is isotopic to $K$, we have $[K]=[l]\in H_1(T^2;\mathbb{Z})$. Since $[m]=p[m^*]-k_0[l^*]$ and $[m]\cdot[l]=[m^*]\cdot[l^*]=-1$, we have $$[m]\cdot [l]=(p[m^*]-k_0[l^*])\cdot (-q(k^\prime)^2 [m^*]+k_1[l^*])=(pk_1-k_0q(k^\prime)^2)[m^*]\cdot[l^*].$$Hence we conclude the congruence result for $k_0$.
\end{proof}
For $i\in \mathbb{Z}/p\mathbb{Z}$, suppose $B_i$ are bands in $\mathbb{R}^2$ mentioned above, ordered from left to right. Suppose $\mathfrak{s}_i\in \operatorname{Spin}^c(Y)$ are spin$^c$ structures corresponding to $B_i$. Since the slope of parallel lines is $p/k_0$, the difference $\mathfrak{s}_{i+1}-\mathfrak{s}_{i}$ is $k_0^\prime\tilde{m}^*$ for the integer $k^\p_0$ satisfying $k_0k_0^\prime\equiv -1 \pmod p$. By the above lemma, we have $k_0^\prime\equiv q(k^\prime)^2\pmod p$. By definition of $k^\prime$ in Lemma \ref{lam0}, we have\[-qk^\prime\equiv \begin{cases} -q-l+1& v \text{ even}\\ q-l+1 & v\text{ odd}\end{cases}\pmod p.\] Since $[a]=-k^\prime\tilde{m}^*$, bands $B_{-iqk^\prime}$ for $i\in [1,l-1]$ correspond to $\mathfrak{b}(u^\prime,v^\prime)$ and $B_{-iqk^\prime}$ for $i\in [l,p]$ correspond to $\mathfrak{b}(u,v)$ in $\widehat{HF}(M)$. Finally, the Alexander grading indicates the relative height of the curves in bands and there is a unique way to connect curves in different bands.

\section{Knots in the same homology class}\label{s5}
For fixed parameters $(p,q,u,v)$ and each $h\in H_1(L(p,q^\prime);\mathbb{Z})\cong \mathbb{Z}/p\mathbb{Z}$, where $qq^\p\equiv 1\pmod p$, there is a parameter $l\in[1,p]$ so that $C(p,q,l,u,v)$ is a representative of $h$, \textit{i.e.} $[C(p,q,l,u,v)]=h$. In other words, for any knot $K$ in $L(p,q^\p)$, there are infinitely many constrained knots $K^\prime$ satisfying $[K^\prime]=[K]\in H_1(L(p,q^\p);\mathbb{Z})$.

In this section, we focus on knots representing the same homology class in a lens space. The main results are Theorem \ref{constrained} and Theorem \ref{simple}. Since we will not use the parameters of a constrained knot, we denote a lens space by $L(p,q)$ rather than $L(p,q^\p)$ as in other sections. Many results in this section are related to the Turaev torsion $\tau(M)$ of a 3-manifold $M$ with torus boundary \cite{Turaev2002}, which can be calculated by any presentation of $\pi_1(M)$. For simplicity, write $\tau(K)=\tau(E(K))$. The following proposition enables us to compare elements in homology groups of different knot complements.
\begin{proposition}[{\cite{Brody1960}}]\label{homologyclass}
Let $K$ be a knot in a 3-manifold $Y$. The isomorphism class of the homology $H_1(E(K);\mathbb{Z})$ only depends on the homology class $[K]\in H_1(Y;\mathbb{Z})$.
\end{proposition}
Suppose $Y=L(p,q)$ and $K$ is a knot in $Y$. By Proposition \ref{homo}, Lemma \ref{lam0} and Proposition \ref{homologyclass}, there exists a positive integer $d$ satisfying $H_1(E(K);\mathbb{Z})\cong\mathbb{Z}\oplus\mathbb{Z}/d\mathbb{Z}$. Let $m$ be the meridian of $K$ in the sense of Section \ref{s2}. Suppose $t,r$ are generators of $\mathbb{Z}\oplus\mathbb{Z}/d\mathbb{Z}$ such that $$H_1(E(K);\mathbb{Z})\cong\langle t,r\rangle/(dr).$$Then there exist $p_0,a\in\mathbb{Z}$ so that the above isomorphism sends $[m]$ to $p_0t+ar$.
\begin{lemma}\label{lam6}
The integer $p$ is divisible by $d$, and $p_0=\pm p/d$. Moreover, the greatest common divisor of $p_0,d$ and $a$ is 1.
\end{lemma}
\begin{proof}
By the isomorphism $H_1(E(K);\mathbb{Z})/([m])\cong H_1(Y;\mathbb{Z})$, the order $p$ of $H_1(Y;\mathbb{Z})$ is the same as $$|\operatorname{det}(\begin{bmatrix}
p_0&a\\0&d
\end{bmatrix})|=|dp_0|.$$If the greatest common divisor of $p_0,d$ and $a$ is not 1, then the Smith normal form of this matrix cannot be $$\begin{bmatrix}
1&0\\0&p
\end{bmatrix}$$because elementary transformations in the algorithm of the Smith normal form does not decrease the common divisor of all entries.
\end{proof}
\begin{lemma}\label{lam7}
Let $K_1$ and $K_2$ be two knots in $Y=L(p,q)$ representing the same homology class $h\in H_1(Y;\mathbb{Z})$. Let $m_1$ and $m_2$ be meridians of $K_1$ and $K_2$ in the sense of Section \ref{s2}, respectively. For $i=1,2$, there are isomorphisms $j_i:H_1(E(K_i);\mathbb{Z})\to\mathbb{Z}\oplus\mathbb{Z}/d\mathbb{Z}$ so that $j_1([m_1])=j_2([m_2])$.
\end{lemma}
\begin{proof}
For $i=1,2$, by discussion after Proposition \ref{homologyclass}, there exists an isomorphism $j_i^\prime:H_1(E(K_i);\mathbb{Z})\to\mathbb{Z}\oplus\mathbb{Z}/d\mathbb{Z}$ so that \[j_1^\prime([m_1])=p_0t+ar\aand j_2^\prime([m_2])=p_0^\prime t+br.\]Then it suffices to find an automorphism $f$ of $\mathbb{Z}\oplus\mathbb{Z}/d\mathbb{Z}$ so that $f(p_0t+ar)=p^\prime_0t+br$. By Lemma \ref{lam6}, the integers $p_0,p_0^\prime$ are in $\{p/d,-p/d\}$. Let $f_0$ be the automorphism of $\mathbb{Z}\oplus\mathbb{Z}/d\mathbb{Z}$ sending $(t,r)$ to $(-t,r)$. If $p_0=-p/d$, the map $j^\prime_1$ can be replaced by $f_0\circ j^\prime_1$. The same assertion holds for $p_0^\prime$. Without loss of generality, suppose $p_0=p_0^\prime=p/d$. Let $g=\operatorname{gcd}(p_0,d)$ and $p_0=gp_1,d=gd_0$. Then $\operatorname{gcd}(p_1,d_0)=1$, and there exist integers $x_0,k_0$ satisfying $x_0p_1+k_0d_0=1$. By Lemma \ref{lam6}, $\operatorname{gcd}(g,a)=\operatorname{gcd}(g,b)=1$. There exist integers $a_0,k_1$ satisfying $a_0a+k_1g=b$ and $\operatorname{gcd}(a_0,g)=1$. Suppose $x=(k_1-k_2a)x_0$ and $y=k_2g+a_0$ for some integer $k_2$. Then
\[\begin{aligned}xp_0+ya\equiv &(k_1-k_2a)x_0gp_1+(k_2g+a_0)a \\\equiv &(k_1-k_2a)(1-k_0d_0)g+(k_2g+a_0)a\\\equiv &k_1g+a_0a\equiv b\pmod{ gd_0}.\end{aligned}\]
The map
\[\begin{aligned}
f:\mathbb{Z}\oplus\mathbb{Z}/d\mathbb{Z}&\to \mathbb{Z}\oplus\mathbb{Z}/d\mathbb{Z}\\
t&\mapsto t+xr\\r&\mapsto yr
\end{aligned}\]
is an isomorphism if and only if $\operatorname{gcd}(y,d)=1$. Since $f(t+ar)=t+(xp_0+ya)r$, this lemma follows from the next proposition.
\end{proof}
\begin{proposition}
Suppose integers $a_0$ and $g$ satisfying $\operatorname{gcd}(a_0,g)=1$. For any integer $d$, there exists integer $k_2$ satisfying $\operatorname{gcd}(y,d)=1$, where $y=k_2g+a_0$.
\end{proposition}
\begin{proof}
If $q$ is a prime number satisfying $p|\operatorname{gcd}(g,d)$, then $a_0$ is not divisible by $q$ and neither is $y$ because $\operatorname{gcd}(a_0,g)=1$. Then $\operatorname{gcd}(y,d)=\operatorname{gcd}(y,d/q$). Without loss of generality, suppose $\operatorname{gcd}(g,d)=1$. By the Chinese remainder theorem, the following congruence equations have a solution $y$:\[y\equiv a_0 \pmod g,y\equiv 1 \pmod d.\]
Then $\operatorname{gcd}(y,d)=1$. We know that $k_2=(y-a_0)/g$ satisfies the proposition.
\end{proof}
From now on, let us fix isomorphisms $j_1$ and $j_2$ as in Lemma \ref{lam7}. Then the homology classes of meridians and their images under $j_i$ for $i=1,2$ can be identified, \textit{i.e.} $[m_1]$ and $[m_2]$ are regarded as the same element $[m]$ in $\mathbb{Z}\oplus\mathbb{Z}/d\mathbb{Z}$. The following lemma is the key lemma in this section, which is based on results in \cite{Turaev2002}.
\begin{lemma}\label{lam8}
Let $K_1$ and $K_2$ be two knots in $Y=L(p,q)$ representing the same homology class. Let $j_i$ be the isomorphisms $H_1(E(K_i);\mathbb{Z})\cong\mathbb{Z}\oplus\mathbb{Z}/d\mathbb{Z}=H_1$ as in Lemma \ref{lam7}. Then $\tau(K_1)-\tau(K_2)$ can be regarded as an element in $\mathbb{Z}[H_1]/\pm H_1$. Moreover, we have $$\tau(K_1)-\tau(K_2)=(1-[m])g\text{ for some }g\in \mathbb{Z}[H_1]/\pm H_1.$$
\end{lemma}
\begin{proof}
Note that $\tau(K_i)$ is not \textit{a priori} an element in $\mathbb{Z}[H_1(E(K_i);\mathbb{Z})]/\pm H_1(E(K_i);\mathbb{Z})$ (\textit{c.f.} \cite[Corollary II.4.3]{Turaev2002}). However, the difference $\tau(K_1)-\tau(K_2)$ is a well-defined element in $\mathbb{Z}[H_1]/\pm H_1$ under the isomorphisms of group rings induced by $j_1$ and $j_2$. To resolve the ambiguity of $\pm H_1$, we can choose an Euler structure and a homology orientation on $E(K_i)$ (\textit{c.f.} \cite[Section I.1]{Turaev2002}). For any compact 3-manifold with torus boundary, Euler structures are one-to-one corresponding to spin$^c$ structures related to the Alexander grading. For any closed 3-manifold, Euler structures are one-to-one corresponding to spin$^c$ structures on the manifold. We omit the choice of the homology orientation that determines the sign of $\tau(K_i)$, and only consider the choice of the Euler structure for simplity. For an Euler strucure $e$ on $M$, the Turaev torsion $\tau(M)$ has a representative $\tau(M,e)$.

For $i=1,2$, let $e_i$ be Euler structures on $E(K_i)$ inducing the same Euler structure $e^Y$ on $Y$. Adapting notations from \cite[Section II.4.5]{Turaev2002}, suppose the integer $K(e_i)$ satisfying\[c(e_i)=e_i/e_i^{-1}\in t^{K(e_i)}\operatorname{Tors}H_1.\]We can also consider $c(e_i)$ as the Chern class of the spin$^c$ structure on $E(K_i)$ corresponding to $e_i$. Note that $t$ is the generator of the free part of $H_1$.

From the correspondence between Euler structures and spin$^c$ structures, it is possible to choose $e_i$ so that $K(e_1)=K(e_2)$. In the proof of \cite[Lemma II.4.5.1(i)]{Turaev2002}, we have \[\tau(E(K_i),e_i)\in \frac{-\Sigma_{H_1}}{t-1}+\mathbb{Z}[H_1],\] where $\Sigma_H=\Sigma_{h\in \operatorname{Tors}H_1}h$. Thus $\tau(K_1,e_1)-\tau(K_2,e_2)\in\mathbb{Z}[H_1]$. Moreover, in \cite[Section II.4]{Turaev2002}, for a 3-manifold $M$ with $b_1(M)=1$, the polynomial part $[\tau](M,e)\in(\frac{1}{2}\mathbb{Z})[H_1]$ of $\tau(M,e)$ is defined by
\begin{equation}
[\tau](M,e)=(\tau(M,e)+\frac{\Sigma_{H_1(M)}}{t-1})\times\begin{cases}
t^{\frac{K(e)+1}{2}}&K(e)\text{ odd,} \\ t^{\frac{K(e)}{2}}(\frac{t+1}{2})&K(e)\text{ even.}
\end{cases}
\end{equation}By \cite[Remark II.4.5.2]{Turaev2002}, for any Euler structure $e$ on $M$, the polynomial part $[\tau](M,e)$ is in the kernel of the map $\operatorname{aug}:\mathbb{Z}[H_1]\to \mathbb{Z}$ that sends elements in $H_1$ to $1\in\mathbb{Z}$.
Thus,\[\operatorname{aug}(\tau(K_1,e_1)-\tau(K_2,e_2))=\operatorname{aug}([\tau](K_1,e_1)-[\tau](K_2,e_2))=0.\]
By $m=1$ case in \cite[Theorem X.4.1]{Turaev2002}, since the map $\kappa:\mathbb{Q}[H_1]\to\mathbb{Q}[H_1]$ that sends $x$ to $x-\operatorname{aug}(x)\Sigma_{H_1}/|H_1|$ is trivial, we have \[\operatorname{pr}(\tau(K_1,e_1)-\tau(K_2,e_2))=-([K_1]-1)\tau(Y,e^Y)+([K_2]-1)\tau(Y,e^Y)=0,\]where $\operatorname{pr}$ is the map in the following proposition. Also from the following proposition, there is an element $g\in \mathbb{Z}[H_1]$ so that $$\tau(K_1,e_1)-\tau(K_2,e_2)=(1-[m])g.$$Since $\tau(K_1,e_1)-\tau(K_2,e_2)$ reduces to $\tau(K_1)-\tau(K_2)$ in $\mathbb{Z}[H_1]/\pm H_1$, we obtain the equation for elements in $\mathbb{Z}[H_1]/\pm H_1$.
\end{proof}
\begin{proposition}
Let $\operatorname{pr}:\mathbb{Z}[\mathbb{Z}\oplus\mathbb{Z}/d\mathbb{Z}]\to \mathbb{Z}[\mathbb{Z}/p\mathbb{Z}]$ be the map between group rings induced by the composition of maps:\[\begin{aligned}
\mathbb{Z}\oplus\mathbb{Z}/d\mathbb{Z}&\stackrel{\cong}{\to}H_1(E(K_i);\mathbb{Z})\to H_1(E(K_i);\mathbb{Z})/([m_i])\\&\stackrel{\cong}{\to} H_1(Y;\mathbb{Z})\stackrel{\cong}{\to}\mathbb{Z}/p\mathbb{Z}.
\end{aligned}\]Then the kernel of $\operatorname{pr}$ is the ideal generated by $1-[m]$.
\end{proposition}
\begin{proof}
Suppose $\mathbb{Z}/p\mathbb{Z}=\{s_1,\dots,s_p\}$ and suppose $H=\sum_{i=1}^ka_ih_i$ is an element in the kernel of $\operatorname{pr}$, where $a_i\in\mathbb{Z}$ and $h_i\in \mathbb{Z}\oplus\mathbb{Z}/d\mathbb{Z}$. Let $\Theta_H(s_j)$ be the set consisting of all elements $h_i$ in the summation of $H$ satisfying $\operatorname{pr}(h_i)=s_j$ in the summation of $H$. Then $\sum_{h_i\in \Theta_H(s_j)} a_ih_i$ is also in the kernel of $\operatorname{pr}$ for any $j$. Without loss of generality, suppose $\operatorname{pr}(h_i)=s_1$ for any $h_i$ in the summation of $H$. By definition of the map $\operatorname{pr}$, for any $i$, we have $h_i=[m]^{\alpha(i)}h_1$ for some integer $\al(i)$. Then  $$H=\sum_{j=0}^{k^\prime}b_j[m]^jh_1$$for some integer $k^\p$. Since $H$ is in the kernel of $\operatorname{pr}$,  we have $$\sum_{j=0}^{k^\prime}b_j=0$$Thus, the polynomial $$\sum_{j=0}^{k^\prime}b_jx^j$$has a root $x=1$. In other words, $\sum_{j=0}^{k^\prime}b_jx^j=(1-x)g(x)$ for some polynomial $g(x)$. Then we have $H=(1-[m])g([m])h_1$ and conclude the proposition.

There is another quick proof from the referee. The functor that takes a group to its group ring is left-adjoint to the
functor that takes a commutative ring to its group of units. The quotient $\mathbb{Z}/p\mathbb{Z}$ is the colimit of the diagram $\mathbb{Z}\rightrightarrows \mathbb{Z}\oplus \mathbb{Z}/d\mathbb{Z}$, where one map is $1\mapsto [m]$ and the other is the zero map. Then the proposition follows from the fact that left-adjoints preserve colimits.
\end{proof}
\begin{lemma}[{\cite[Proposition 2.1]{Rasmussen2017}}]\label{lam9}
Suppose $K$ is a knot in $Y=L(p,q)$. Let $H_1=H_1(E(K);\mathbb{Z})$. Then  $$\chi(\widehat{HFK}(Y,K))= (1-[m])\tau(K)\in \mathbb{Z}[H_1]/\pm H_1.$$
\end{lemma}
\begin{theorem}\label{thm3}
Let $K_1$ and $K_2$ be two knots representing the same homology class in $Y=L(p,q)$. Suppose $H_1(E(K_i);\mathbb{Z})\cong\mathbb{Z}\oplus\mathbb{Z}/d\mathbb{Z}=H_1$ as in Lemma \ref{lam7}. After shifting Alexander gradings on $\widehat{HFK}(Y,K_i)$ for $i=1,2$, the difference of their Euler characteristics satisfies the following condition: for any $\mathfrak{s}\in\operatorname{Spin}^c(Y)$, there exists a Laurent polynomial $f(x)\in\mathbb{Z}[x,x^{-1}]$ and an element $\tilde{s}\in H_1$ so that\[\chi(\widehat{HFK}(Y,K_1,\mathfrak{s}))-\chi(\widehat{HFK}(Y,K_2,\mathfrak{s}))=([m]-1)^2f([m])\tilde{s}.\]
\end{theorem}
\begin{proof}
Note that $\chi(\widehat{HFK}(Y,K_i))$ is an element in $\mathbb{Z}[H_1]$ up to equivalence. Fixing the Alexander grading on $\widehat{HFK}(Y,K_i)$ is equivalent to choose a representative of $\chi(\widehat{HFK}(Y,K_i))$ in $\mathbb{Z}[H_1]$. By Lemma \ref{lam8} and Lemma \ref{lam9}, after shifting Alexander gradings, there exists some $g\in \mathbb{Z}[H_1]/\pm H_1$ so that $$\chi(\widehat{HFK}(Y,K_1))-\chi(\widehat{HFK}(Y,K_2))=(1-[m])(\tau(K_1)-\tau(K_2))=([m]-1)^2g.$$Choose a lift $\tilde{g}$ of $g$ in $\mathbb{Z}[H_1]$. It can be written as the sum $\tilde{g}=\sum_{j=1}^p g_j$, where $g_j$ contains terms that are in the preimage of $s_j\in H_1(Y;\mathbb{Z})$ under the map $H_1\to H_1(Y;\mathbb{Z})=\{s_1,\dots,s_p\}$. For any $j$, there exists a Laurent polynomial $f_j(x)$ and an element $\tilde{s}_j\in H_1$ so that $g_j=f_j([m])\tilde{s}_j$. Thus, the above equation can be decomposed into spin$^c$ structures, which induces the theorem.
\end{proof}
\begin{remark}
For constrained knots $K_1$ and $K_2$, the group $\widehat{HFK}(Y,K_i)$ can be chosen as the canonical representative in Section \ref{s4}, \textit{i.e.} we consider the absolute Alexander grading mentioned in the introduction.
\end{remark}
\begin{proof}[Proof of Theorem \ref{constrained}]
    We choose the isomorphisms $H_1(E(K_i);\mathbb{Z})\cong H_1$ considered in Lemma \ref{lam7}. By Lemma \ref{lam1}, for a constrained knot $K_i\subset Y$ and a spin$^c$ structure $\mathfrak{s}$ on $Y$, there is a symmetrized Alexander polynomial $\Delta_i(t)$ of a 2-bridge knot, so that $$\chi(\widehat{HFK}(Y,K_i,\mathfrak{s}))\sim \Delta_i([m]).$$Since the Alexander grading reduces to the grading induced by spin$^c$ structures under the map $H_1(E(K_i);\mathbb{Z})\to H_1(Y;\mathbb{Z})$, we know Alexander gradings of nontrivial summands of $ \widehat{HFK}(Y,K_i,\mathfrak{s})$ correspond to the spin$^c$ structure $\mathfrak{s}$. By definition of the equvalence on $\mathbb{Z}[H_1]$, there exists an element $\tilde{s}\in H_1$ in the preimage of $\mathfrak{s}$ so that  \begin{equation}\label{eq3}\chi(\widehat{HFK}(Y,K_i,\mathfrak{s}))=\pm \Delta_i([m])[m]^{\gamma_i} \tilde{s},\end{equation}where $\ga_i$ is an integer. Write $f_i(x)=\pm \Delta_i(x)x^{\gamma_i}$ for simplicity. Since $\Delta_i(t)$ is symmetrized, the middle grading is the grading of $[m]^{\ga_i}\tilde{s}$. Note that the multiplication in $\mathbb{Z}[H_1]$ corresponds to the addition in $H_1$. Then we have $$A(K_1,\mathfrak{s})-A(K_2,\mathfrak{s})=(\ga_1[m]+\tilde{s})-(\ga_2[m]+\tilde{s})=(\gamma_1-\gamma_2)[m]\in H_1.$$By Theorem \ref{thm3}, there is a Laurent polynomial $f(x)\in\mathbb{Z}[x,x^{-1}]$ so that $$f_1(x)-f_2(x)=(x-1)^2f(x).$$Hence for a large integer $N$, there is a polynomial $f_0(x)$ so that $$x^N(f_1(x)-f_2(x))=(x-1)^2f_0(x).$$ Substituting $x=1$ gives $f_1(1)=f_2(1)$, \textit{i.e.} signs in Equations \ref{eq3} are the same for $i=1,2$. Consider derivatives at $x=1$: \[\begin{aligned}0&=\frac{d(x^N(f_1(x)-f_2(x)))}{dx}\\&=N(f_1(1)-f_2(1))+\frac{df_{1}}{dx}(1)-\frac{df_{2}}{dx}(1)\\&=\pm(\frac{d\Delta_1(x)}{dx}(1)-\frac{d\Delta_2(x)}{dx}(1)+\gamma_1\Delta_1(1)-\gamma_2\Delta_2(1))\\&=\gamma_1-\gamma_2.\end{aligned}\]where the last equation is from $\Delta_i(t)=\Delta_i(t^{-1})$ and $\Delta_i(1)=1$. Thus, we have $A(K_1,\mathfrak{s})=A(K_2,\mathfrak{s}).$
\end{proof}
\begin{proof}[Proof of Theorem \ref{simple}]
It follows from the proof of Theorem \ref{constrained} with $\Delta_i(t)=1$.
\end{proof}
\section{Classification}\label{s6}

The main result in this section is the proof of the sufficient part of Theorem \ref{thm2}. The following lemma enables us to prove it by considering knot groups, \textit{i.e.} fundamental groups of knot complements.
\begin{lemma}[\cite{Waldhausen1968}]\label{hakenlemma}
Let $M_1$ and $M_2$ be Haken manifolds with torus boundaries. If there is an isomorphism $\psi:\pi_1(M_1)\rightarrow \pi_1(M_2)$ that induces an isomorphism $\psi|_{\pi_1(\partial M_1)}:\pi_1(\partial M_1)\to \pi_1(\partial M_2)$, then there exists a diffeomorphism $\psi_0:(M_1,\partial M_1)\rightarrow (M_2,\partial M_2)$ inducing $\psi$.

In addition, if $M_1$ and $M_2$ are two knot complements and $\psi$ sends the meridian of one knot to the meridian of the other knot, then two knots are equivalent.
\end{lemma}

A constrained knot is defined by a doubly-pointed Heegaard digram, from which it is easy to obtain a Heegaard diagram of the knot complement similar to the case in Figure \ref{f6}. The Heegaard diagram is related to the handlebody decomposition of the corresponding 3-manifold, and then also related to the cell complex of the corresponding 3-manifold. Thus, it is possible to obtain a presentation of the fundamental group from the Heegaard diagram. We show this presentation explicitly in the following.

Suppose $K=C(p,q,l,u,v)$ is a constrained knot with $u>2v\ge 0$. Suppose $(T^2,\al_1,\be_1,z,w)$ is the standard diagram of $K$. Let $\Sigma$ be the surface of genus two obtained by attaching a 1-handle at basepoints $z$ and $w$. Suppose $\al_2$ is the curve on $\Sigma$ that is a union of an arc connecting $z$ to $w$ in $T^2-\al_1$ and an arc on the attached handle; see Figure \ref{gr}. Suppose $\be=\be_1$. Then $(\Sigma,\{\alpha_1,\alpha_2\},\beta)$ is a Heegaard diagram  of $E(K)$.

Let the innermost rainbow $R_0$ around $w$ be oriented from the right boundary point $x_r$ to the left boundary point $x_l$. This induces an orientation of $\beta$. Let $\alpha_1$ and $\alpha_2$ be oriented from the left vertical edge to the right vertical edge in the new diagram $C$ of the constrained knot.

Suppose $s$ and $t$ correspond to cores of $\alpha_1$-handle and $\alpha_2$-handle, respectively. In the above orientation, we can obtain a presentation $\pi_1(E(K))\cong \langle s,t|\omega=1\rangle$, where the word $\omega$ is given in the following way:
\begin{enumerate}[(1)]
\item starting at $x_l$ and traveling along $\beta$, suppose intersection points of $\beta\cap(\alpha_1\cup\alpha_2)$ are ordered as $x_1,x_2,\dots,x_m$;
\item if $x_i$ is an intersection point of $\alpha_1$ and $\beta$, it corresponds to a word $s^{\pm1}$, where the sign depends on the contribution of $x_i$ in the algebraic intersection number $\alpha_1\cap \beta$;
\item if $x_i$ is an intersection point of $\alpha_2$ and $\beta$, it corresponds to a word $t^{\pm1}$, where the sign depends on the contribution of $x_i$ in the algebraic intersection number $\alpha_2\cap \beta$;
\item the word $\omega$ is obtained from $x_1x_2\cdots x_m$ by replacing $x_i$ by corresponding words in $\{s,s^{-1},t,t^{-1}\}$.
\end{enumerate}
The word $\omega(p,q,l,u,v)=\omega(C(p,q,l,u,v))$ in the above setting is called the \textbf{standard relation} of a constrained knot $C(p,q,l,u,v)$. We begin by understanding the standard relation of a 2-bridge knot. For fixed integers $(u,v)$, let $\epsilon_i=(-1)^{\lfloor iv/u\rfloor}$.

\begin{lemma}
For $C(1,0,1,u,v)\cong \mathfrak{b}(u,v)$, the standard relation $\omega$ is $s^{\epsilon_1}t^{\epsilon_2}s^{\epsilon_3}\cdots s^{\epsilon_{2u-1}}t^{\epsilon_{2u}}$.
\end{lemma}
\begin{proof}
This is from the relation between the Schubert normal form and the Heegaard diagram of the 2-bridge knot. Note that the formula of the Alexander polynomial in Proposition \ref{alexander} follows from this presentation and Fox calculus \cite[Chapter II]{Turaev2002}.
\end{proof}

For fixed integers $(p,q,l)$ with $q\in[1,p-1],l\in[1,p]$ and $\operatorname{gcd}(p,q)=1$, suppose the integer $k\in(0,p]$ satisfies $k-1\equiv  (l-1)q\pmod p$ and the integer $q_i\in[0,p)$ satisfies $q_i\equiv iq\pmod p$. Define $$\theta_i=\theta_i(p,q,l)=\begin{cases}1 &q_i\in[0,k),\\0 &q_i\in [k,p), \end{cases}$$ \[s_*(p,q,l)=st^{\theta_l}st^{\theta_{l+1}}s\cdots st^{\theta_{p-1}}s \aand t_*(p,q,l)= t^{\theta_0}st^{\theta_1}s\cdots st^{\theta_{l-1}}.\]

In particular, we have $\theta_0=1$ and $\theta_{l-1}=1$. Note that the integer $q$ in $s_*(p,q,l)$ or $t_*(p,q,l)$ does not correspond the parameter $q$ in $C(p,q,l,u,v)$. Indeed, the constrained knot $C(p,q,l,u,v)$ corresponds to $s_*(p,q^\p,l)$ and $t_*(p,q^\p,l)$, where $qq^\p\equiv 1\pmod p$. We can see this fact from the following proposition.

\begin{proposition}\label{standardrelation}
For $K=C(p,q,l,u,v)$, suppose that the integer $q^\prime\in[0,p)$ satisfies $qq^\prime\equiv 1 \pmod p$. Suppose $s_*=s_*(p,q^\prime,l)$ and  $t_*=t_*(p,q^\prime,l)$. Define$$t_\#^{\epsilon_i}=\begin{cases}t^{\epsilon_i} & \epsilon_{i-1}=-\epsilon_{i+1},\\t_*^{\epsilon_i} & \epsilon_{i-1}=\epsilon_{i+1}. \end{cases}$$Then the standard relation of $K$ is  $\omega(p,q,l,u,v)=s^{\epsilon_1}_*t^{\epsilon_2}_{\#}s^{\epsilon_3}_*\cdots s^{\epsilon_{2u-1}}_*t_\#^{\epsilon_{2u}}$.
\end{proposition}
\begin{proof}
    The standard diagram of $C(p,q,l,u,v)$ generalizes the standard diagram of $C(1,0,1,u,v)$. Then $\omega(p,q,l,u,v)$ can be obtained from $\omega(1,0,1,u,v)$ by replacing $s$ and $t$ by some words. We figure out the replacement as follows.

    Suppose that the integer $k\in(0,p]$ satisfies $(k-1)q\equiv l-1 \pmod p$, which coincides with the definition of $k$ for $(p,q^\p,l)$ before this proposition. Note that we define $q_i^\p$ by $qq_i^\p\equiv i\pmod p$ since we consider $(p,q^\p,l)$ rather than $(p,q,l)$.

    Consider the new diagram $C$ of $K$ mentioned in Section \ref{s3}; see Figure \ref{l2}. There are regions $D_j$ for $j\in \mathbb{Z}/p\mathbb{Z}$, where the right edge of $D_j$ is glued to the left edge of $D_{j+q}$. Consider the part of $\al_2$ on $T^2$ that connects $z$ to $w$.  It goes across regions in the order $$D_1,D_{q+1},D_{2q+1},\dots, D_{l}.$$By definition of $k$, there are $k$ regions in the above sequence. By definition of $q_i^\p$, for any region $D_j$ in the above sequence, it lies at the $(q_{j-1}^\p+1)$-th position, so $q_j^\p<k$. For example, $q_0^\p=0$ implies that $D_1$ lies at the first position and $q_{l-1}^\p=k-1$ implies that $D_l$ lies at the $k$-th position. Then $\theta_j=1$ if and only if $\al_2\cap D_{j+1}$ is nonempty.

    Then the word $s_*(p,q^\p,l)$ corresponds to intersection points of $\be\cap(\al_1\cup \al_2)$ on an arc component of $\be\cap (\bigcup_{j=l+1}^{p-1} D_j)$. The word $t_*(p,q^\p,l)$ corresponds to intersection points of $\be\cap(\al_1\cup \al_2)$ on an arc component of $\be\cap (\bigcup_{j=1}^{l} D_j)$ that is also a subarc of a stripe.

    Thus, we can replace $s$ by $s_*=s_*(p,q^\p,l)$. When $\epsilon_{i-1}=-\epsilon_{i+1}$, the corresponding intersection point related to $t^{\epsilon_i}$ is on the rainbow, so we just replace $t^{\epsilon_i}$ by $t^{\epsilon_i}$ itself. When $\epsilon_{i-1}=\epsilon_{i+1}$, the corresponding intersection point related to $t^{\epsilon_i}$  is on the stripe, so we replace $t^{\epsilon_i}$ by $t_*^{\epsilon_i}=t_*^{\epsilon_i}(p,q^\p,l)$. This is how $t_\#^{\epsilon_i}$ is defined.

\end{proof}
Suppose $K_1=C(p,q,l,u,v)$ and $K_2=C(p,q^\p,l,u,v)$, where $qq^\p\equiv 1\pmod p$ and $l\in\{2,p\}$. Proposition \ref{standardrelation} provides presentations of $\pi_1(E(K_1))$ and $\pi_1(E(K_2))$. We will construct an explicit isomorphism $\pi_1(E(K_1))\cong \pi_1(E(K_2))$ based on the standard relations. First of all, let us introduce some notations.

Given words $w_1,w_2$ made by $s$ and $t$, let $h_{w_1,w_2}=h(w_1,w_2)$ be a map on words such that for any word $\omega$ made by $s$ and $t$, the word $h_{w_1,w_2}(\omega)$ is obtained from $\omega$ by replacing $s$ and $t$ by $w_1$ and $w_2$, respectively. For any integer $n$, define maps $$f_1^n=h(s,s^nt),f_2^n=h(t^ns,t),g_1^n=h(s,ts^n),g_2^n=h(s^nt,t)$$and $$h_0=h(t,s),h_1=h(t,s^{-1}),h_2=h_1\circ h_1=h(s^{-1},t^{-1}).$$The map $f_1^n$ induces an isomorphism $\langle s,t|\omega\rangle\cong \langle s,t|f_1^n(\omega)\rangle$ by mapping $t$ to $s^{n}t$ and $s$ to $s$, which is still denoted by $f_1^n$. The similar argument applies to $f_2^n$. For $m$ odd, let $f^n_m=f^n_1$. For $m$ even, let $f^n_m=f_2^n$. Given integers $p,q>0$, suppose $$\frac{q}{p}=[a_0;a_1,a_2,\dots,a_m]=a_0+\frac{1}{a_1+\frac{1}{a_2+\frac{1}{a_3+\dots}}}$$ is the unique continued fraction of $q/p$ with $a_i>0$ and $a_m>1$. Define $$f^{q/p}=f_m^{-a_m+1}\circ f^{-a_{m-1}}_{m-1}\circ \cdots \circ f_1^{-a_1} \circ f_0^{-a_0}\aand F^{q/p}=f_1^1\circ f_2^{-1}\circ f^{q/p}.$$ The maps $g_m^n,g^{q/p},G^{q/p}$ are defined similarly based on $g_1^n$ and $g_2^n$.
\begin{figure}[ht]
\centering 
\includegraphics[width=0.7\textwidth]{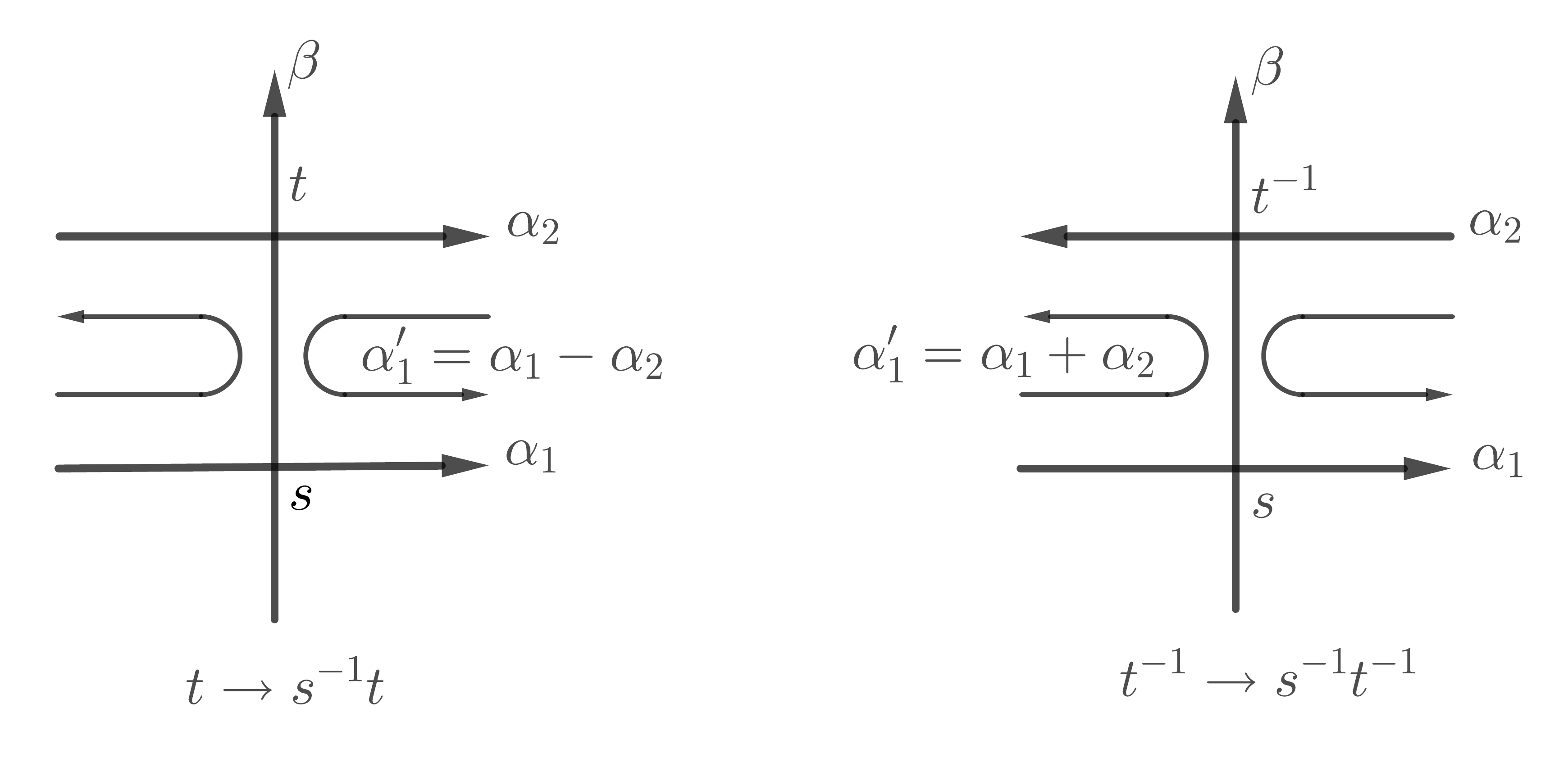} 
\caption{Examples of handle slides.}
\label{sli}
\end{figure}
\begin{remark}\label{remsli}
The isomorphisms $f^n_m$ and $g^n_m$ can be achieved by handle slides of $\al$ curves in the Heegaard diagram of the knot complement. Indeed, if there are two consecutive intersection points $x_i$ and $x_{i+1}$ in the definition of the standard relation that correspond to $s$ and $t$, respectively, then the arc of $\beta$ between $x_i$ and $x_{i+1}$ can be used for the handle slide. If $\alpha_1$ is slided over $\alpha_2$, then the relation $\omega$ becomes $f^{-1}_1(\omega)$. If $\alpha_2$ is slided over $\alpha_1$, then the relation $\omega$ becomes $g^{-1}_2(\omega)$. Moreover, when \[(x_i,x_{i+1})\to(s,t),(s,t^{-1}),(s^{-1},t),(s^{-1},t^{-1}),(t,s),(t,s^{-1}),(t^{-1},s),(t^{-1},s^{-1}),\]where $\to$ implies the replacement considered in the definition of the standard relation, then the corresponding maps are \[(f_1^{-1},g_2^{-1}),(g_1^{1},g_2^1),(f^1_1,f_2^1),(g^{-1}_1,f_2^{-1}),(f_2^{-1},g_1^{-1}),(g_2^{1},g_1^1),(f^1_2,f_1^1),(g^{-1}_2,f_1^{-1}),\]respectively. Two examples are shown in Figure \ref{sli}.
\end{remark}
The proof of the following lemma follows directly from definitions of maps.
\begin{lemma}
There are relations between maps:
\begin{enumerate}[(i)]
    \item $h_0\circ h_0=h_2\circ h_2=\operatorname{id}$;
    \item $f^n_1\circ h_1=h_1\circ f^{-n}_2,f^n_2\circ h_1=h_1\circ g^{-n}_1$;
    \item $g^n_1\circ h_1=h_1\circ g^{-n}_2,g^n_2\circ h_1=h_1\circ f^{-n}_1$.
\end{enumerate}
\end{lemma}
In the following lemmas, integers $p,q,q^\p$ satisfy $$p>0,q,q^\p\in[1,p-1], \operatorname{gcd}(p,q)=1\aand qq^\prime\equiv 1 \pmod p.$$
\begin{lemma}\label{lmm}
The following equations hold:\begin{equation}\label{sta1} f^{q/p}(s_*(p,q,2)ts)=ts\aand f^{q/p}(s_*(p,q,2)st)=st,\end{equation}\begin{equation}\label{sta11} g^{q/p}(tss_*(p,q,2))=ts\aand g^{q/p}(sts_*(p,q,2))=st.\end{equation}
\end{lemma}
\begin{proof}
When $l=2$, by definition $s_*(p,q,2)=st^{\theta_2}st^{\theta_{3}}s\cdots st^{\theta_{p-1}}s$, where $\theta_i=\theta_i(p,q,2)$. Suppose that the integer $k$ satisfies $k-1\equiv (l-1)q\pmod p$. We know that $k=q+1$. Suppose $$\frac{q}{p}=[0;a_1,a_2,\dots,a_m]$$with $a_i>0$ and $a_m>1$. We prove Equations \ref{sta1} by induction on $m$.

If $m=1$, then $q=1$ and $p=a_1$. Thus $s_*(p,q,2)=s^{a_1-1}$ and $ f^{q/p}=f^{-(a_1-1)}_1$ by definition. It can be checked directly that Equations \ref{sta1} hold.

Suppose Equations \ref{sta1} hold for $m=m_0-1$. Consider integers $q_i$ satisfying $q_i\equiv iq\pmod p$. Since $\operatorname{gcd}(p,q)=1$, if $q_i\equiv iq\equiv q\pmod p$, then $i=1$. So $q_i\neq q$ for $i\in [2,p-1]$. Since $k=q+1$, the condition $q_i\in[0,k)$ is the same as $q_i\in[0,q)$ for $i\in [2,p-1]$. Thus $\theta_i(p,q,2)=1$ if and only if \[\lfloor \frac{iq}{p}\rfloor-\lfloor \frac{(i-1)q}{p}\rfloor=1.\]In other words, we have $$\theta_i(p,q,2)=\lfloor \frac{iq}{p}\rfloor-\lfloor \frac{(i-1)q}{p}\rfloor \text{ for }i\in [2,p-1].$$If $\theta_i(p,q,2)=1$, there is some integer $j\in[1,q-1]$ so that $$i=\lfloor \frac{jp}{q}\rfloor+1 =ja_1+\lfloor \frac{jr}{q}\rfloor+1,$$where $$\frac{r}{q}=[0;a_2,a_3,\dots,a_{m_0}].$$Let $j_1=j,j_2=j-1$ for $j\in [2,q-1]$. Then we have $$(j_1a_1+\lfloor \frac{j_1r}{q}\rfloor+1)-(j_2a_1+\lfloor \frac{j_2r}{q}\rfloor+1)=a_1+\lfloor \frac{j_1r}{q}\rfloor-\lfloor \frac{j_2r}{q}\rfloor=a_1+\theta_j(q,r,2).$$Thus
\[\begin{aligned}
    s_*(p,q,2)ts=&s^{a_1}ts^{a_1}s^{\theta_2(q,r,2)}ts^{a_1}s^{\theta_3(q,r,2)}t\cdots s^{a_1}s^{\theta_{q-2}(q,r,2)}ts^{a_1}s^{\theta_{q-1}(q,r,2)}ts^{a_1}ts\\
    =&(s^{a_1}t)s^{\theta_2(q,r,2)}(s^{a_1}t)s^{\theta_3(q,r,2)}(s^{a_1}t)\cdots s^{\theta_{q-2}(q,r,2)}(s^{a_1}t)s^{\theta_{q-1}(q,r,2)}(s^{a_1}t)(s^{a_1}t)s\\
    =&h_{s^{a_1}t,s}(s_*(q,r,2)st)= f^{a_1}_1\circ h_0(s_*(q,r,2)st),
\end{aligned}\]
where the second equation follows from the fact that $\theta_i(p,q,2)=0$ if $i<a_1$. Similarly, 
\[\begin{aligned}
    s_*(p,q,2)st=&s^{a_1}ts^{a_1}s^{\theta_2(q,r,2)}ts^{a_1}s^{\theta_3(q,r,2)}t\cdots s^{a_1}s^{\theta_{q-2}(q,r,2)}ts^{a_1}s^{\theta_{q-1}(q,r,2)}ts^{a_1}st\\=&(s^{a_1}t)s^{\theta_2(q,r,2)}(s^{a_1}t)s^{\theta_3(q,r,2)}(s^{a_1}t)\cdots s^{\theta_{q-2}(q,r,2)}(s^{a_1}t)s^{\theta_{q-1}(q,r,2)}(s^{a_1}t)s(s^{a_1}t)\\=&h_{s^{a_1}t,s}(s_*(q,r,2)ts)= f^{a_1}_1\circ h_0(s_*(q,r,2)ts).
\end{aligned}\]
By the inductive assumption, we have $$f^{r/q}((s_*(p,q,2)ts)=ts\aand f^{r/q}((s_*(p,q,2)st)=st.$$

Since $f^{q/p}=h_0\circ f^{r/q}\circ h_0\circ f_1^{-a_1}$ and $h_0\circ h_0=\operatorname{id}$, we have \[ f^{q/p}((s_*(p,q,2)ts)= h_0\circ f^{r/q}(s_*(q,r,2)st)=h_0(st)=ts,\]
\[ f^{q/p}((s_*(p,q,2)st)= h_0\circ f^{r/q}(s_*(q,r,2)ts)=h_0(ts)=st.\]By a similar method, it can be proven that \[tss_*(p,q,2)=g^{a_1}_1\circ h_0(sts_*(q,r,2)),sts_*(p,q,2)=g^{a_1}_1\circ h_0(tss_*(q,r,2)).\]Then by induction, Equations \ref{sta11} hold.
\end{proof}
\begin{lemma}\label{lmmm}
The following equations hold:
 \[F^{q/p}(t)=f_1^{1}\circ f^{-1}_2\circ f^{q/p}(t)=h_0(s_*(p,q^\prime,2)ts),\]
  \[G^{q/p}(t)=g_1^{1}\circ g^{-1}_2\circ g^{q/p}(t)=h_0(sts_*(p,q^\prime,2)).\]
\end{lemma}
\begin{proof}
The proofs of two equations are similar. We only show the proof of the first equation. By the proof of Lemma \ref{lmm}, we know$$\theta_i(p,q,2)=\lfloor \frac{iq}{p}\rfloor-\lfloor \frac{(i-1)q}{p}\rfloor \text{ for }i\in[2,p-1] .$$Thus $\theta_i(p,q,2)=0$ if and only if $$\lfloor \frac{i(q-p)}{p}\rfloor-\lfloor \frac{(i-1)(q-p)}{p}\rfloor=\lfloor \frac{iq}{p}\rfloor-\lfloor \frac{(i-1)q}{p}\rfloor-1=-1.$$This is equivalent to $$\lfloor \frac{i(p-q)}{p}\rfloor-\lfloor\frac{(i-1)(p-q)}{p}\rfloor=1,$$\textit{i.e.} $\theta_i(p,p-q,2)=1$. Then \[\begin{aligned}f^{-1}_1\circ h_0(s_*(p,q,2)ts)=&s^{-1}ts^{-\theta_2(p,p-q,2)}t\cdots ts^{-\theta_{p-1}(p,p-q,2)}tt\\=s^{-1}h_1(s_*(p,p-q)st)s=&t^{-1}h_1(sts_*(p,p-q))t.
\end{aligned}\]Suppose $q/p=[0;a_1,a_2,\dots,a_m]$ with $a_i>0$ and $a_m>1$. We have   \[f^{-1}_2\circ f^{q/p}=\begin{cases}
f_2^{-1}\circ f_1^{-a_m+1}\circ f_2^{-a_{m-1}}\circ \cdots\circ f_2^{-a_2}\circ f_1^{-a_1}&m\text{ odd,} \\ f_2^{-a_m}\circ f_1^{-a_{m-1}}\circ\cdots \circ f_2^{-a_2}\circ f_1^{-a_1}&m\text{ even.}
\end{cases}\]
By the extended Euclidean algorithm, \[\frac{p-q^\prime}{p}=\begin{cases}
[0;1,a_m-1,a_{m-1},\dots,a_2,a_1]&m\text{ odd,} \\ [0;a_m,a_{m-1},\dots,a_2,a_1]&m\text{ even.}
\end{cases}\]
It can be proven by induction on $n$ that for $b/a=[0;b_{1},b_{2},\dots,b_{2n-1},b_{2n}]$, \begin{equation}\label{sta2}
    f^{-b_{1}}_2\circ f^{-b_{2}}_1\circ \cdots\circ f^{-b_{2n-1}}_2\circ f^{-b_{2n}}_1(t)=h_1(ts_*(a,b)s).
\end{equation} Indeed, if $n=1$, then $f^{-b_2}_2\circ f^{-b_1}_1(t)=(t^{-b_2}s)^{-b_1}t=(s^{-1}t^{b_2})^{b_1}t$. Equation \ref{sta2} is clear.

Suppose Equation \ref{sta2} holds for $n=n_0-1$. Let $$b^\prime/a^\prime=[0;b_{2},\dots,b_{2n_0-1},b_{2n_0}]\aand b^{\prime\prime}/a^{\prime\prime}=[0;b_{3},\dots,b_{2n_0-1},b_{2n_0}].$$By the proof of Lemma \ref{lmm},
\[\begin{aligned}tf^{b_1}_1(s_*(a^{\prime\prime},b^{\prime\prime},2)st)t^{-1}=&th_0(s_*(a^\prime,b^\prime,2)ts)t^{-1}\\=s^{-1}h_0(tss_*(a^\prime,b^\prime,2))s=&s^{-1}g^{b_1}_1(sts_*(a^{\prime\prime},b^{\prime\prime},2))s.
\end{aligned}\]
Thus \[\begin{aligned}f^{-b_{1}}_2\circ f_1^{-b_{2}}\circ h_1(ts_*(a,b,2)stt^{-1})&
=f^{-b_{1}}_2\circ  h_1\circ f_2^{b_{2}}(tf_2^{-b_{2}}\circ f_{1}^{-b_{1}}(s_*(a^{\prime\prime},b^{\prime\prime},2)st)t^{-1})\\
=f^{-b_{1}}_2\circ  h_1(t f_{1}^{-b_{1}}(s_*(a^{\prime\prime},b^{\prime\prime},2)st)t^{-1})&
=h_1\circ g^{b_{1}}_1\circ  (s^{-1}g^{b_1}_1(sts_*(a^{\prime\prime},b^{\prime\prime},2))s)\\
=h_1(s^{-1}(sts_*(a^{\prime\prime},b^{\prime\prime},2))s)&
=h_1(ts_*(a^{\prime\prime},b^{\prime\prime},2)s).\end{aligned}\]
\end{proof}
\begin{remark}\label{remsli1}
By Remark \ref{remsli}, the map $f^{q/p}$ can be regarded as a sequence of handle slides. Consider the matrix of algebraic intersection points$$\begin{bmatrix}
[\alpha_1]\cdot p[a]&[\alpha_2]\cdot p[a]\\ [\alpha_1]\cdot [m]& [\alpha_2]\cdot [m] \end{bmatrix},$$where $a$ and $m$ are curves in Figure \ref{gr}. The maps $f^n_1$ and $f^n_2$ induce column transformations of this matrix, which are still denoted by $f^n_1$ and $f^n_2$. Then 
\[f^{q/p}(\begin{bmatrix}
p&q\\0&1\end{bmatrix})=\begin{bmatrix}
1&1\\q^\prime-p&q^\prime\end{bmatrix}\aand F^{q/p}(\begin{bmatrix}
p&q\\0&1\end{bmatrix})=\begin{bmatrix}
1&0\\q^\prime&p\end{bmatrix}.\]
Indeed, the definitions of $f^{q/p}$ and $F^{q/p}$ come from the extend Euclidean algorithm for calculating $\gcd(p,q)$ (\textit{c.f.} the proof of Lemma \ref{lmmm}).
\end{remark}
\begin{proposition}\label{pro0}
Up to circular permutation, \[h_0\circ F^{q/p}(\omega(p,q^\prime,2,u,v))=\begin{cases}h_2(\omega(p,q,2,u,v))& v \text{ odd,}\\\omega(p,q,2,u,v)&v \text{ even.}\end{cases}\]
\end{proposition}
\begin{proof}
Suppose $a=s_*(p,q,2)$ and $b=s_*(p,q^\prime,2)$. Then $$t_*=t_*(p,q,2)=t_*(p,q^\p,2)=tst,\text{ and}$$ \[\omega(p,q^\prime,2,u,v)=a^{\epsilon_1}t^{\epsilon_2}_\#a^{\epsilon_3}\cdots t^{\epsilon_{2u}}_\#,\omega(p,q,2,u,v)=b^{\epsilon_1}t^{\epsilon_2}_\#b^{\epsilon_3}\cdots t^{\epsilon_{2u}}_\#.\]
The word $a^{\epsilon_{i-1}}t^{\epsilon_i}_\#a^{\epsilon_{i+1}}$ is one of the following cases:
\begin{enumerate}[(i)]
    \item $atsta=(ats)ta$ \text{ and } $a^{-1}(tst)^{-1}a^{-1}=a^{-1}t^{-1}(ats)^{-1}$;
    \item $ata^{-1}=(ats)t(ast)^{-1}$\text{ and } $at^{-1}a^{-1}=(ast)t^{-1}(ats)^{-1}$;
    \item $a^{-1}ta$\text{ and } $a^{-1}t^{-1}a$.
\end{enumerate}
Thus $\omega(p,q^\prime,2,u,v)=a^{\epsilon_1}_\#t^{\epsilon_2}a^{\epsilon_3}_\#\cdots t^{\epsilon_{2u}}$, where \[a^{\epsilon_i}_\#=\begin{cases}(ats)^{\epsilon_i} & \epsilon_{i}=\epsilon_{\epsilon_{i}+i},\\(ast)^{\epsilon_i}& \epsilon_{i}=-\epsilon_{\epsilon_{i}+i}.\end{cases}\]By Lemma \ref{lmm} and Lemma \ref{lmmm}, \[F^{q/p}(ats)=s=h_0(t),F^{q/p}(ast)=t^{-1}st=h_0(s^{-1}ts),F^{q/p}(t)=h_0(bts).\]Thus $h_0\circ F^{q/p}(\omega(p,q^\prime,2,u,v))=c^{\epsilon_1}_\#(bts)^{\epsilon_2}c^{\epsilon_3}_\#\cdots (bts)^{\epsilon_{2u}}$, where
\[c^{\epsilon_i}_\#=\begin{cases}t^{\epsilon_i} & \epsilon_{i}=\epsilon_{\epsilon_{i}+i},\\(s^{-1}ts)^{\epsilon_i}& \epsilon_{i}=-\epsilon_{\epsilon_{i}+i}.\end{cases}\]The word $(bts)^{\epsilon_{i-1}}c_{\#}^{\epsilon_i} (bts)^{\epsilon_{i+1}}$ is one of the following cases:
\begin{enumerate}[(i)]
    \item $(bts)t(bts)=b(tst)bts$ \text{ and }$(bts)^{-1}(t)^{-1}(bts)^{-1}=(bts)^{-1}(tst)^{-1}b^{-1}$;
    \item $(bts)(s^{-1}ts)(bts)^{-1}=btb^{-1}$\text{ and }$(bts)(s^{-1}ts)^{-1}(bts)^{-1}=bt^{-1}b^{-1}$;
    \item $(bts)^{-1}t(bts)$\text{ and }$(bts)^{-1}t^{-1}(bts)$.
\end{enumerate}
Thus $$h_0\circ F^{q/p}(\omega(p,q^\prime,2,u,v))=t_\#^{\epsilon_1}b^{\epsilon_2}t^{\epsilon_3}_\#\cdots b^{\epsilon_{2u}}=b_\#^{\epsilon_{u+1}}t^{\epsilon_{u+2}}b^{\epsilon_{u+3}}_\#\cdots b^{\epsilon_{3u}},$$where the last equation holds up to circular permutation. The proposition follows from the fact that $\epsilon_{u+i}=(-1)^v \epsilon_{i}$.
\end{proof}
\begin{proposition}\label{pro1}
Up to circular permutation, \[h_0\circ G^{(p-q)/p}(\omega(p,q^\prime,p,u,v))=\begin{cases}h_2(\omega(p,q,p,u,v))& v \text{ odd,}\\\omega(p,q,p,u,v)&v \text{ even.}\end{cases}\]
\end{proposition}
\begin{proof}
The essential idea of the proof is the same as that of Proposition \ref{pro0}. Now $$s_*(p,q,p)=s\aand t_*(p,q,p)=ts_*(p,p-q,2)t.$$Suppose $a=s_*(p,p-q,2)$ and $b=s_*(p,p-q^\prime,2)$. By analyzing cases of $s^{\epsilon_{i-1}}t_\#^{\epsilon_i}s^{\epsilon_{i}}$, there is similar result $\omega(p,q^\prime,p,u,v)=a^{\epsilon_1}_\#t^{\epsilon_2}a^{\epsilon_3}_\#\cdots t^{\epsilon_{2u}}$, where \[a^{\epsilon_i}_\#=\begin{cases}(sta)^{\epsilon_i} & \epsilon_{i}=\epsilon_{\epsilon_{i}+i},\\(sat)^{\epsilon_i}& \epsilon_{i}=-\epsilon_{\epsilon_{i}+i}.\end{cases}\]Note that $\epsilon_i\in\{\pm 1\}$ and $\epsilon_{\epsilon_i+i}=\epsilon_{\pm+i}$ in the definition of $a_\#^{\epsilon_i}$. By Lemma \ref{lmm} and Lemma \ref{lmmm}, we have $G^{(p-q)/p}(t)=stb$. Thus $$h_0\circ G^{(p-q)/p}(\omega(p,q^\prime,p,u,v))=c^{\epsilon_1}_\#(stb)^{\epsilon_2}c^{\epsilon_3}_\#\cdots (stb)^{\epsilon_{2u}},$$where \[c^{\epsilon_i}_\#=\begin{cases}t^{\epsilon_i} & \epsilon_{i}=\epsilon_{\epsilon_{i}+i},\\((stb)^{-1}(sts^{-1})(stb))^{\epsilon_i}& \epsilon_{i}=-\epsilon_{\epsilon_{i}+i}.\end{cases}\]By analyzing cases of $(stb)^{\epsilon_{i-1}}c_\#^{\epsilon_i}(stb)^{\epsilon_{i-1}}(stb)^{\epsilon_{i+1}}$, the similar equation hold$$h_0\circ G^{(p-q)/p}(\omega(p,q^\prime,p,u,v))=t_\#^{\epsilon_1}b^{\epsilon_2}t^{\epsilon_3}_\#\cdots b^{\epsilon_{2u}}$$Then this proposition follows from the similar argument as in Proposition \ref{pro0}.
\end{proof}

\begin{proof}[Proof of the sufficient part of Theorem \ref{thm2}]
For $i=1,2$, suppose $K_i=C(p_i,q_i,l_i,u_i,v_i)$, $M_i=E(K_i)$ and suppose $(\mu_i,\lambda_i)$ is the regular basis of $\partial M_i$. Suppose $$(p_1,u_1,v_1)=(p_2,u_2,v_2)=(p,u,v), q_1q_2\equiv 1\pmod p \aand l_1=l_2\in\{2,p\}.$$

By knot Floer homology, constraind knots $K_i$ are not unknots in lens spaces. By Proposition \ref{haken}, we know that $M_i$ are Haken manifolds.

Let $q^\p=q_1$ and $q=q_2$ in Proposition \ref{pro0} and Proposition \ref{pro1}. Let $\psi$ be the map from $\pi_1(M_1)$ to $\pi_1(M_2)$ induced by$$\begin{cases}
 h_0\circ F^{q/p} & l_1=l_2=2,\\
 h_0\circ G^{(p-q)/p} & l_1=l_2=p.
\end{cases}$$
By Proposition \ref{pro0} and Proposition \ref{pro1}, the map $\psi$ is an isomorphism. The meridians $\mu_i$ and longitudes $\lambda_i$ can be isotoped to lie on Heegaard diagrams of $M_i$ so that $\mu_1=m,\mu_2=pa$, where $a$ and $m$ are curves in Figure \ref{gr}. Moreover, suppose that meridians and longitudes are disjoint from $\beta_1$. By Remark \ref{remsli} and Remark \ref{remsli1}, the map $\psi$ can be achieved by handle slides of $\alpha$ curves. After handle slides, the meridian and the longitude are still disjoint from $\beta_1$, which implies $\psi$ induces an isomorphism $\psi|_{\pi_1(\partial M_1)}:\pi_1(\partial M_1)\to \pi_1(\partial M_2)$.

Moreover, for the case $l_1=l_2=2$, note that $t$ corresponds to $\mu_1\cap(\al_1\cup \al_2)$ and $s_*(p,q^\p,2)ts$ corresponds to $\mu_2\cap(\al_1\cup\al_2)$ in the presentations of the fundamental groups. By Lemma \ref{lmmm}, $$\psi|_{\pi_1(\partial M_1)}(\mu_1)=\psi(t)=s_*(p,q^\p,2)ts=\mu_2.$$Thus, by Lemma \ref{hakenlemma}, we know $K_1$ is equivalent to $K_2$.

For the case $l_1=l_2=p$, based on Lemma \ref{lmmm}, the proof is similar.
\end{proof}
\section{Magic links}\label{s7}

A constrained knot is defined by a doubly-pointed Heegaard diagram $(T^2,\al_1,\be_1,z,w)$, where $\be_1$ looks similar to the $\be$ curve in the diagram of a 2-bridge knot (\textit{c.f.} Proposition \ref{constrained2-bridge} and Lemma \ref{lam1}). In this section, we provide Dehn surgery descriptions for some families of constrained knots, which is inspired by the relation between constrained knots and 2-bridge knots. The main objects in this section are magic links.
\begin{definition}
Suppose integers $u,v$ satisfying $ 0\le v<u,\gcd(u,v)=1$ and $u$ odd. Especially $(u,v)=(1,0)$ is allowed. A \textbf{magic link} $\mathfrak{L}(u,v)=K_0\cup K_1\cup K_2$ is a 3-component link linked as shown in Figure \ref{magic}, where $K_0$ is the 2-bridge knot $\mathfrak{b}(u,v)$ in the standard presentation, $K_1$ and $K_2$ are unknots. For $-u<v<0$, let $\mathfrak{L}(u,v)$ be the mirror link of $\mathfrak{L}(u,-v)$. Let $\mathfrak{L}(1,1)$ be the mirror link of $\mathfrak{L}(1,0)$.
\end{definition}
\begin{remark}
The name of magic links is from the fact that the link complement $S^3-\mathfrak{L}(3,1)$ is diffeomorphic to the magic manifold studied in \cite{Martelli2006}.
\end{remark}
For $i=1,2$, suppose integers $p_i,q_i$ satisfying $p_i>0$ and $\operatorname{gcd}(p_i,q_i)=1$. Let $M(u,v,p_1/q_1,p_2/q_2)$ and $K_0(u,v,p_1/q_1,p_2/q_2)$ denote the manifold and the resulting knot $K_0^\prime$ obtained by $p_i/q_i$ Dehn surgery on $K_i$.
\begin{proposition}\label{mirr3}
The manifold $M(u,v,p_1/q_1,p_2/q_2)$ is diffeomorphic to $M(u,v,p_2/q_2,p_1/q_1)$. Moreover, the knots $K_0^\prime$ in these manifolds are equivalent.
\end{proposition}
\begin{proof}
The components $K_1$ and $K_2$ in the magic link switch the positions under the rotation around a vertical line, while $K_0$ remains unchanged.
\end{proof}
\begin{remark}
Manifolds $M(u,v,p_1/q_1,p_2/q_2),M(u,v,p_2/q_2,p_1/q_1)$ will not be distinguished in the rest of the paper. So are the corresponding knots $K_0$ these manifolds.
\end{remark}
\begin{proposition}\label{mirr}
For integers $u,v$ satifying $0<v<u,\gcd(u,v)=1$ and $u$ odd, the link $\mathfrak{L}(u,u-v)$ is the mirror link of $\mathfrak{L}(u,v)$. Thus $\mathfrak{L}(u,u-v)\cong \mathfrak{L}(u,-v)$ and $K_0(u,v,p_1/q_1,p_2/q_2)$ is the mirror image of $K_0(u,u-v,p_1/(-q_1),p_2/(-q_2))$.
\end{proposition}
\begin{proof}
Suppose $\mathfrak{b}(u,v)$ is in the standard presentation for $$\frac{v}{u}=[0;a_1,a_2,\dots,a_m].$$Since $(u-v)/u=1-v/u$, by adding one positive half-twist on the two left strands, the standard presentation for $[0;-a_1,-a_2,\dots,-a_m]$ becomes a standard presentation of $\mathfrak{b}(u,u-v)$. After isotoping the link outside twists related to $a_i$, the link $\mathfrak{L}(u,u-v)$ becomes the mirror link of $\mathfrak{L}(u,v)$.
\end{proof}
\begin{lemma}\label{dia}
The diagrams $(\Sigma_2,\alpha^*,\beta_1)$ in Figure \ref{magic2} are Heegaard diagrams of $E(\mathfrak{L}(3,1))$. For $i=1,2$, the meridian $m_i$ and the longitude $l_i$ of $K_i$ can be isotoped to lie on $\Sigma_2$ as in the diagrams. For general integers $u,v$ satisfying $0<v<u$, $\operatorname{gcd}(u,v)=1$ and $u,v$ odd, the similar assertion holds when $\beta_1$ is replaced by $\beta$ in the doubly-pointed Heegaard diagram of $\mathfrak{b}(u,v)$.
\end{lemma}

\begin{figure}[ht]
\centering 
\includegraphics[width=0.8\textwidth]{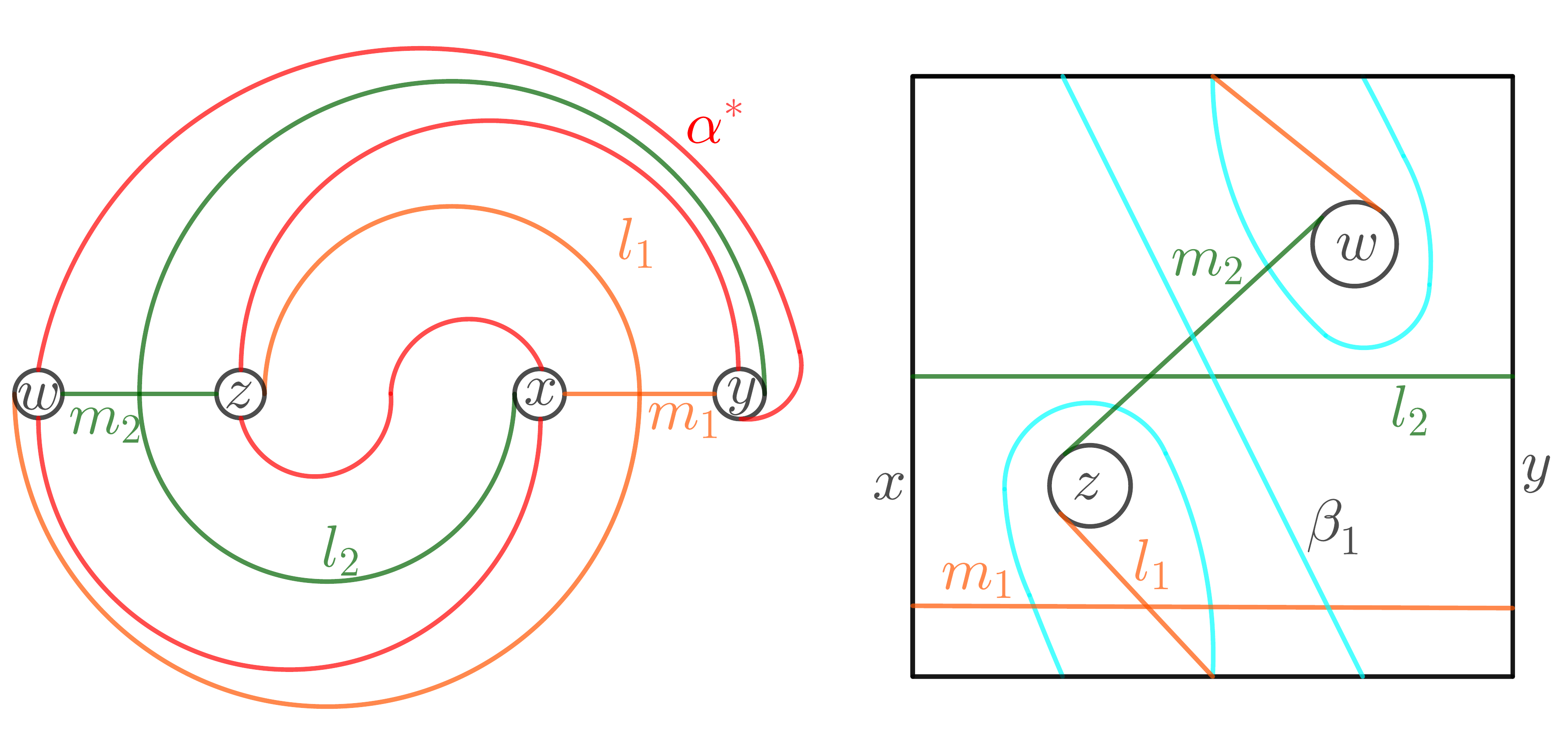} 
\caption{Heegaard diagrams of $E(\mathfrak{L}(3,1))$, where $\beta_1$ is omitted in the left figure and $\alpha^*$ is omitted in the right figure.\label{magic2}}

\end{figure}
\begin{figure}[ht]
\centering 
\includegraphics[width=0.7\textwidth]{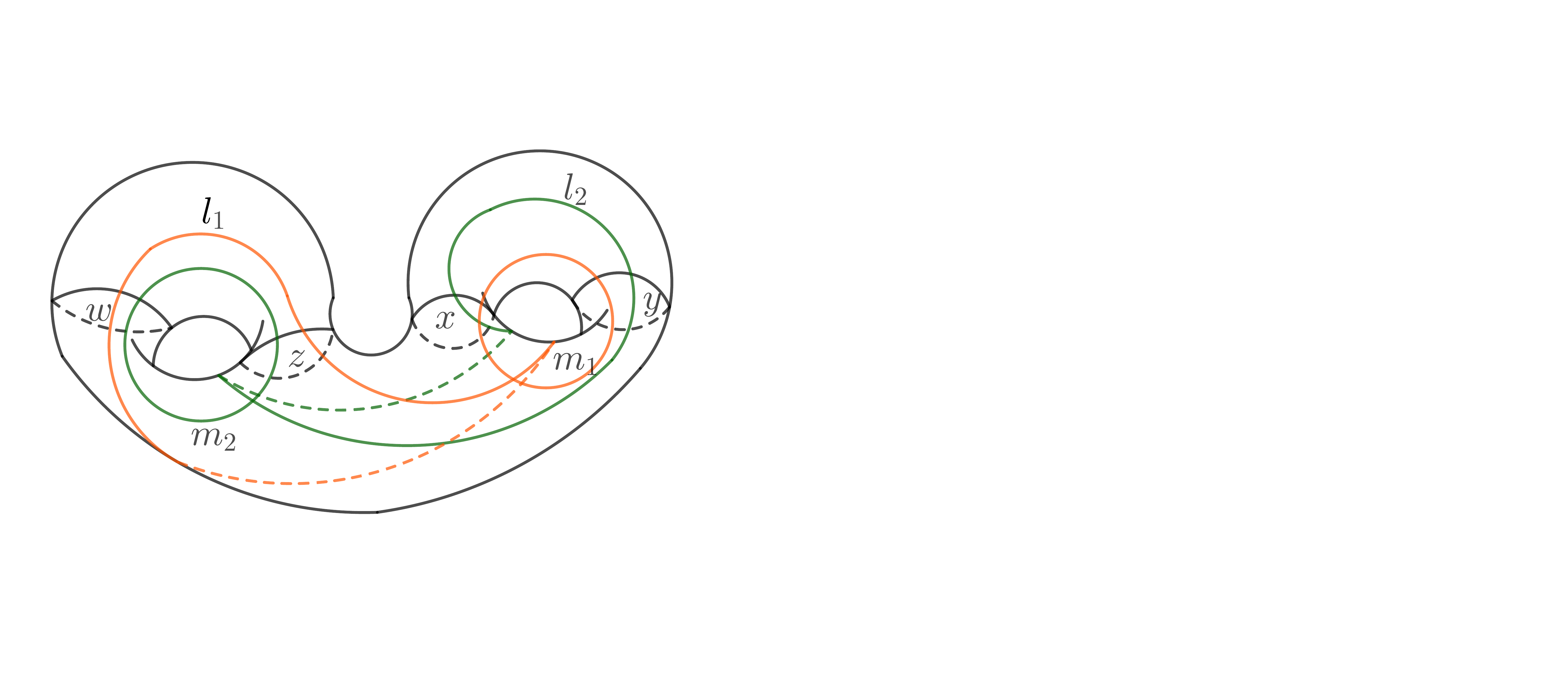} 
\caption{Meridians and longitudes on the Heegaard surface.\label{magic3}}

\end{figure}
\begin{proof}
Consider $(u,v)=(3,1)$. The curve $\alpha^*$ is separating and $\beta_1$ is non-separating. Therefore, the manifold obtained from $\Sigma_2\times[-1,1]$ by attaching 2-handles along $\alpha^*\times \{-1\}$ and $\beta_1\times \{1\}$ has 3 boundary components, each of which is a torus. Moreover, if two more 2-handles are attached along $m_1\times\{-1\}$ and $m_2\times \{-1\}$, the resulting manifold is $E(\mathfrak{b}(3,1))$. The longitude $l_0$ of $\mathfrak{b}(3,1)$ can be isotoped to lie on $\Sigma_2$ as shown in the Schubert normal form (\textit{c.f.} Figure \ref{f5}). Note that the geometric intersection number of $m_i$ and $l_i$ is one.

Conversely, components of the link corresponding to Heegaard diagrams in Figure \ref{magic2} can be obtained by pushing $l_i$ slightly into the handlebody corresponding to $\alpha=\{\alpha^*,m_1,m_2\}$ and pushing $l_0$ slightly into the handlebody corresponding to $\beta=\{\beta_1,m_0\}$, where $m_0$ is the meridian of $\mathfrak{b}(3,1)$ on $\Sigma_2$. This can be seen explicitly if we redraw the Heegaard surface as in Figure \ref{magic3}. After isotoping unknot components, it is easy to see the link from these diagrams is equivalent to $\mathfrak{L}(3,1)$. For general $(u,v)$, the proof applies without change.
\end{proof}
For integers $u,v$ satisfying $-u<v<0$ and $u,v$ odd, the corresponding diagram is obtained by reflecting the diagram of $\mathfrak{L}(u,-v)$ along a vertical line. Since $\mathfrak{L}(u,u-v)\cong  \mathfrak{L}(u,-v)$, Heegaard diagrams for all $v\in(-u,u)$ with $\operatorname{gcd}(u,v)=1$ and $(u,v)=(1,0),(1,1)$ are obtained from this approach. Such diagram is called a \textbf{standard diagram} of $E(\mathfrak{L}(u,v))$.

A resolution of an intersection point of a meridian and a longitude on the Heegaard surface is called a \textbf{positive resolution} or a \textbf{negative resolution} when the meridian turns left or right, respectively, to the longitude in any direction; see Figure \ref{resolve}.
\begin{figure}[ht]
\centering 
\includegraphics[width=0.6\textwidth]{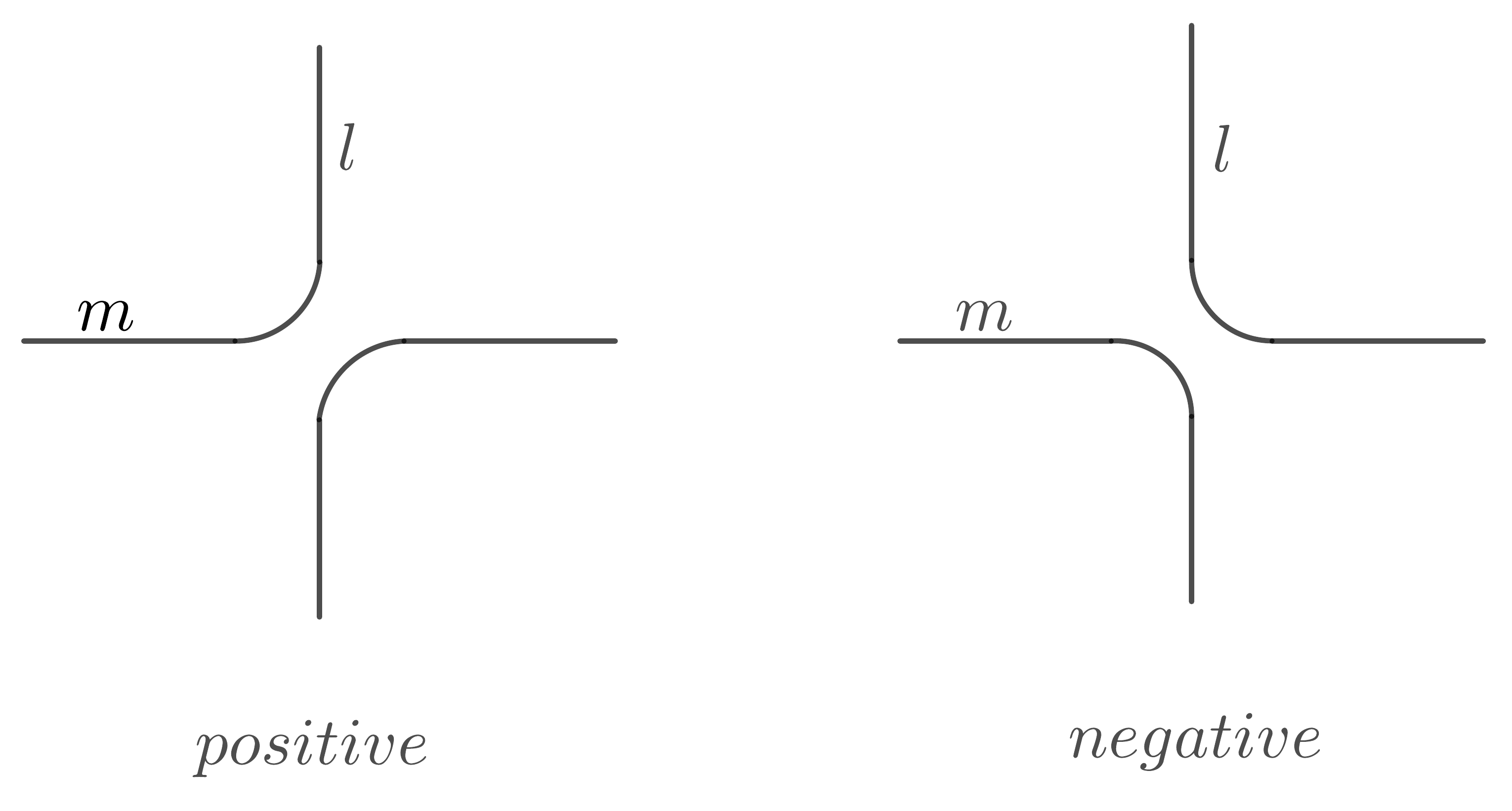} 
\caption{Positive resolution and negative resolution.}
\label{resolve}
\end{figure}
\begin{corollary}\label{coro2}
For $i=1,2$, suppose integers $p_i,q_i$ satisfying $p_i>0$ and $\operatorname{gcd}(p_i,q_i)=1$. The Heegaard diagram $(\Sigma_2,\{\alpha_1,\alpha_2\},\beta_1)$ of $E(K_0(u,v,p_1/q_1,p_2/q_2))$ is obtained by the following way: $\alpha_i$ is obtained by resolving intersection points of $|p_i|$ copies of $m_i$ and $|q_i|$ copies of $l_i$ positively or negatively if $q_i$ is positive or negative, respectively. Especially when $(p_i,q_i)=(1,0)$, the corresponding $\alpha_i$ is $m_i$.
\end{corollary}
\begin{proof}
    This follows from the definition of the Dehn surgery. Note that $\al_i$ is the meridian of the filling solid torus for $i=1,2$.
\end{proof}
\begin{figure}[ht]
\centering 
\includegraphics[width=0.9\textwidth]{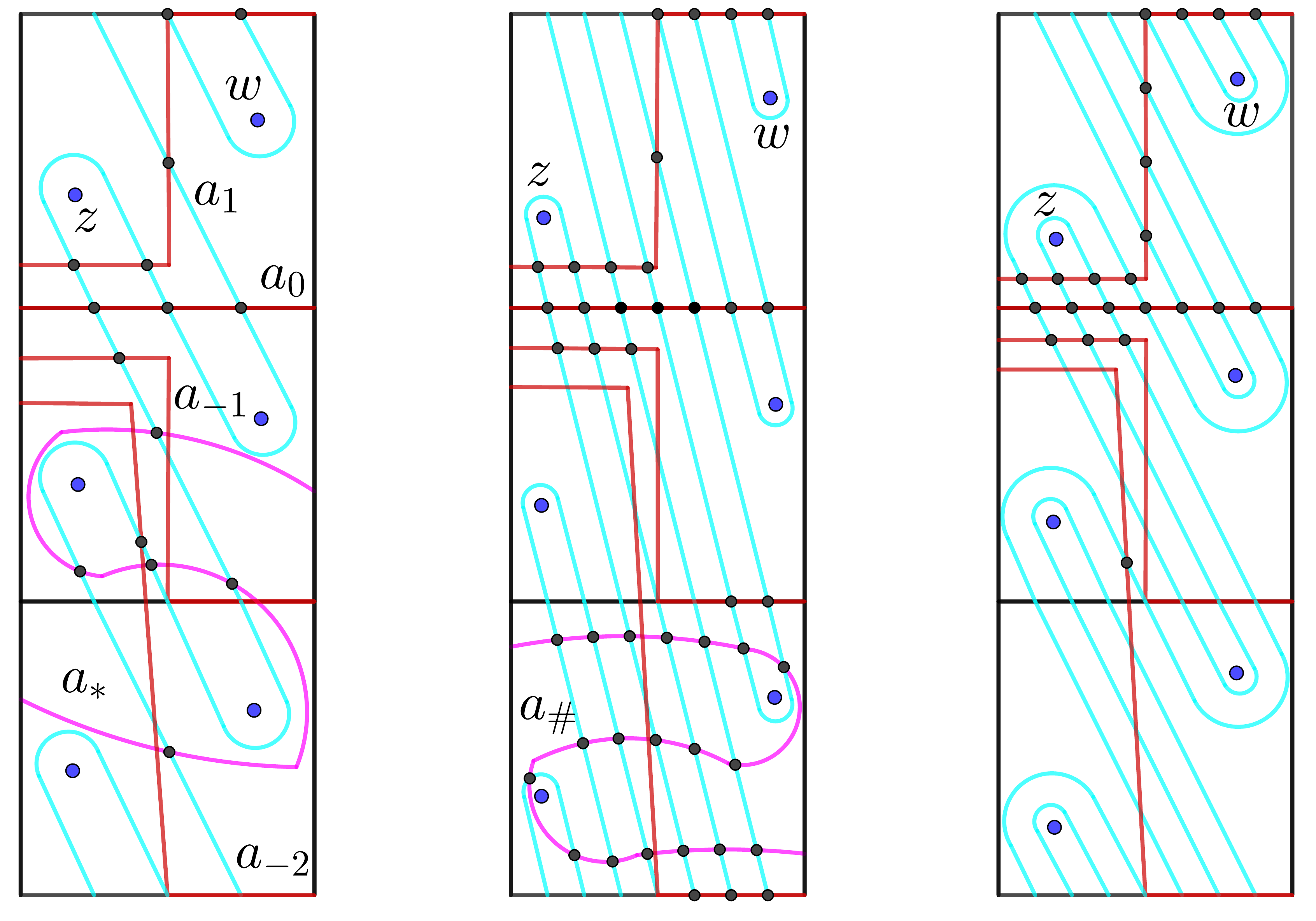} 
\caption{Cyclic covers of Heegaard diagrams corresponding to $(u,v)=(3,1),(7,1),(7,2)$.}
\label{f14}
\end{figure}
Consider cyclic covers of the diagram of a 2-bridge knot $\mathfrak{b}(u,v)$ as shown in Figure \ref{f14}. For $i\in\mathbb{Z}$, let $a_i=a_i(u,v)$ be a red strand connecting the left edge to the right edge and passing through $|i|$ copies of the fundamental domains, where the sign of $i$ determines the direction of the strand; see Figure \ref{f14} for examples of strands. Let $A_i=A_i(u,v)$ be the set consisting of strands that can be isotoped into the neighborhood of $a_i(u,v)$ in the complement of basepoints. Some intersection points of $a_i(u,v)$ and $\beta_1$ can be removed by isotopy. Intersection points that cannot be removed are shown in Figure \ref{f14}. Identifying endpoints of $a_i$, a diagram of a 2-bridge knot $\mathfrak{b}(U(u,v,i),V(u,v,i))$ can be obtained for some integers $ U(u,v,i),V(u,v,i)$.

Let $a_*=a_*(u,v)$ and $a_\#=a_\#(u,v)$ be the strands in Figure \ref{f14}. For $i=*,\#$, the set $A_i(u,v)$, the functions $U(u,v,i),V(u,v,i)$ are defined similarly. For $i\in\mathbb{Z}$ or $i=*,\#$, consider $V(u,v,i)\in \mathbb{Z}/U\mathbb{Z}-\{0\}$ for $U=U(u,v,i)>1$. When $U(u,v,i)=1$, consider $V(u,v,i)\in\{0,1\}$. In the latter case, we use the following conventions:
 $$n\equiv \begin{cases}1 & n  \text{ odd}\\ 0 & n \text{ even}\end{cases}\pmod {1}\aand \pm n\equiv \mp m\pmod {1} \text{ for }n \text{ odd and }m\text{ even.}$$
\begin{lemma}\label{lamuv}
Suppose integers $u,v$ satisfying $(u,v)=(1,0)$ or $0<2v<u,\operatorname{gcd}(u,v)=1$ and $u$ odd. For $i\in\{1,0,-1,-2,*,\#\}$, the functions $U(u,v,i)$ and $V(u,v,i)$ can be expressed explicitly as follows:
\begin{enumerate}[(i)]
\item $U(u,v,1)=u+2v,V(u,v,1)=v$;
\item $U(u,v,0)=u,V(u,v,0)=v$;
\item $U(u,v,-1)=u-2v,V(u,v,-1)\equiv v \pmod {u-2v}$;
\item $U(u,v,-2)=|u-4v|,V(u,v,-2)\equiv v \operatorname{sign}(u-4v) \pmod {|u-4v|}$ for $u>3$; $U(3,1,-2)=1,V(3,1,-2)=1$.
\item $U(u,v,*)=3u-4v,V(u,v,*)=u-v$;
\item $U(u,v,\#)=3u-2v,V(u,v,\#)=2u-v$.
\end{enumerate}
\end{lemma}
\begin{proof}
For fixed $(u,v)$, let $R_i$ and $S_i$ be numbers of rainbows and stripes in the diagram of $\mathfrak{b}(U(u,v,i),V(u,v,i))$. Case $(ii)$ is trivial, where $R_0=v$ and $S_0=u-2v$. Suppose $V^\prime$ satisfies $$0<V^\prime <U(u,v,i)\aand V^\prime\equiv V(u,v,i)\pmod {U(u,v,i)}.$$Define$$\epsilon_i =\begin{cases}-1 & 2V^\prime<U(u,v,i),\\1 & 2V^\prime>U(u,v,i).\end{cases}$$Then $(U(u,v,i),V(u,v,i))$ can be recovered by $(R_i,S_i,\epsilon_i)$ by \begin{equation}\label{equv}U(u,v,i)=2R_i+S_i,V(u,v,i)=\epsilon_i R_i.\end{equation}
Suppose that all isotopies on the surface move basepoints in the following discussion.

For Cases (i),(vi), let $x_1$ be the center of the fundamental domain and let $D_1=N(x_1)$ be the neighborhood containing two basepoints $z$ and $w$. Straightening strands isotopes the diagram by rotating $D_1$ clockwise and counterclockwise by $\pi$ for Cases (i) and (vi), respectively. Equivalently, the new $\beta$ is obtained by pushing rainbows on the top edge to the bottom right and bottom left, respectively. Rainbows and stripes satisfy the following equations and we obtain the results by forumlae in (\ref{equv}). \[R_1=R_0,S_1=2R_0+S_0,\epsilon_1=1 \aand R_\#=R_0+S_0,S_\#=2R_0+S_0,\epsilon_\#=-1.\]

For Case (v), let $x_2$ be the middle intersection point on the top edge and let $D_2=N(x_2)$ be the neighborhood containing all rainbows. Straightening the strand isotopes the diagram by rotating $D_2$ clockwise by $\pi$. Then we have $$R_*=R_0+S_0,S_*=S_0,\epsilon_*=1.$$

For Case (iii), the number $U(u,v,-1)$ is the same as $S_0$. Straightening the strand isotopes the diagram by rotating $D_2$ counterclockwise, which induces the formula of $V(u,v,-1)$. This isotopy can also be regarded as pulling back rainbows once.

For Case (iv), if $(u,v)=(3,1)$, then the formula is directly from Figure \ref{f14}. If $u>3$, then there are three subcases where $S_0>2R_0,2R_0>S_0>R_0$ and $R_0>S_0$, equivalently $u>4v,4v>u>3v$ and $3v>u>2v$, respectively. Note that $u$ is odd, so $u\neq 4v$.

Suppose $S_0>2R_0$ (\textit{e.g.} $(u,v)=(7,1),(13,3)$). In this subcase $V(u,v,-1)=v$. Straightening the strand isotopes the diagram by pulling back rainbows twice. Then $(U(u,v,-2),V(u,v,-2))$ is obtained by applying Case (iii) twice, \textit{i.e.} $$U(u,v,-2)=u-4v,V(u,v,-2)\equiv v \pmod {u-4v}.$$

Suppose $2R_0>S_0>R_0$ (\textit{e.g.} $(u,v)=(7,2),(15,4)$). Straightening the strand isotopes the diagram by rotating $D_2$ counterclockwise by $\pi$. After isotopy, the number of intersection points of $a_{-2}$ and $\beta$ is $U(u,v,-2)=2R_0-S_0=4v-u$. The number of rainbows is $R_{-2}=S_0-R_0$ and $\epsilon_{-2}=-1$. Hence $V(u,v,-2)=U(u,v,-2)-(S_0-R_0)=7v-2u.$

Suppose $R_0>S_0$ (\textit{e.g.} $(u,v)=(7,3)$). Straightening the strand isotopes the diagram by rotating $D_2$ counterclockwise by $\pi$. In this subcase, this isotopy is obtained by reversing the isotopy in Case (v). Then $$R_{-2}=R_0-S_0,S_{-2}=S_0,\epsilon_{-2}=1,U(u,v,-2)=4v-u,V(u,v,-2)=3v-u.$$The formula for Case (iv) then follows from summarizing the above subcases.
\end{proof}
\begin{remark}
Indeed, for any $i\in\mathbb{Z}$, functions $U(u,v,i)$ and $V(u,v,i)$ might be expressed explicitly. For example, we have $U(u,v,i)=u+2iv,V(u,v,i)=v$ for $i>0$. However, for $i<0$, functions are more complicated so we omit the discussion.
\end{remark}
The following lemma is a basic result from the Dehn surgery on the Hopf link.
\begin{lemma}\label{lem: cal pq}The manifold $M(u,v,p_1/q_1,p_2/q_2)$ is diffeomorphic to the lens space $$L(p_1p_2-q_1q_2,p_1p_2^\prime-q_1q_2^\prime),\text{ where }p_2q_2^\prime-q_2p_2^\prime=-1.$$
\end{lemma}
\begin{theorem}\label{thm4}
Suppose integers $u_0,v_0$ satisfying $(u_0,v_0)=(1,0)$ or $0<2v_0<u_0,\operatorname{gcd}(u_0,v_0)=1$ and $u_0$ odd. Suppose $U_i=U(u_0,v_0,i)$ and $V_i=V(u_0,v_0,i)$. The knot $K_0=K_0(u_0,v_0,p_1/q_1,p_2/q_2)$ is equivalent to $C(p,q,l,u,v)$ for $(l,u,v)$ in Table \ref{table3} and some $(p,q)$. In Cases (i)-(iv), $(p,q)=(p_1p_2-q_1q_2,q_1)$. In Cases (v)-(viii), $(p,q)=(p_1p_2-q_1q_2,q_1p_2)$. In Cases (ix)-(x), the parameter $p=p_1p_2-q_1q_2$ and $q\in\{\pm q_0^{\pm 1}\}$, where $q_0=p_1p_2^\p-q_1q_2^\p$ is calculated in Lemma \ref{lem: cal pq}.
\end{theorem}

\begin{table}[t]
\caption{Cases where Dehn surgeries on magic links induce constrained knots. \label{table3}}
\begin{tabular}[!htbp]{|p{0.7cm}|p{6.5cm}|p{5cm}|}
\hline  
Case&Conditions&$(l-1,u,v)$\\
\hline  
(i)&$p_2=1,q_1q_2< 0$&$(-q_1q_2,U_0,V_0)$\\
\hline
(ii)&$p_2=1,q_2>1,q_1>p_1>0,U_{-1}\ge U_{-2}$&$(p_1,U_{-1},V_{-1})$\\
\hline
(ii$^\prime$)&$p_2=1,q_2>1,q_1>p_1>0,U_{-1}<U_{-2}$&$(q_1q_2-2p_1,U_{-2},V_{-2})$\\
\hline
(iii)&$p_2=1,q_2<-1,-q_1>p_1>0$&$(q_1q_2-2p_1,U_*,V_*)$\\
\hline
(iv)&$(p_2,q_2)=(1,0)$&$(0,U_0,V_0)$\\
\hline
(v)&$p_1> 1,|q_1|=1,q_1q_2<0$&$(-q_1q_2,U_0,V_0)$\\
\hline
(vi)&$p_1> 1,q_1=1,p_2>q_2>0$&$(p_1p_2-2q_2,U_1,V_1)$\\
\hline
(vii)&$p_1> 1,q_1=-1,p_2>-q_2>0$&$(-p_2,p_1p_2+2q_2,U_\#,V_\#)$\\
\hline
(viii)&$(p_1,q_1)=(0,1)$&$(0,U_{-1},V_{-1})$\\
\hline
(ix)&$(p_2,q_2)=$(1,1)$,q_1>0,(p_1,q_1)\neq (1,1)$&$(\pm q_1,U_{n},V_{n})$ for some $n\in\mathbb{Z}$\\
\hline
(x)&$(p_2,q_2)=(1,-1),q_1<0,(p_1,q_1)\neq (1,-1)$&$l-1=\pm q_1$\\
\hline 
\end{tabular}
\end{table}


\begin{proof}
At first, we make some comments on the parameters $(p,q)$. Lemma \ref{lem: cal pq} provides a way to specify the ambient lens space of $K_0$. Explicitly, the lens space is $L(p,q)$, where $p=p_1p_2-q_1q_2$ and $q\in\{\pm q_0^{\pm 1}\}$.

In Cases (i)-(iv), $p_2=1$. Hence we can choose $q_2^\p=-1,p_2^\p=0$ in Lemma \ref{lem: cal pq}. Then we can set $q_0=p_1p_2^\prime-q_1q_2^\prime=q_1$. In Cases (v)-(viii), $|q_1|=1$. By Proposition \ref{mirr3}, we can switch the roles of $(p_1,q_1)$ and $(p_2,q_2)$ in Lemma \ref{lem: cal pq}. Hence we can pick $p_1^\p=q_1,q_1^\p=0$ so that $p_1q_1^\p-q_1p_1^\p=-1$. Then we can set $q_0=p_2p_1^\p-q_2q_1^\p=q_1p_2$.

From Remark \ref{rem: diff}, we know that $C(p,q,l,u,v)$ may be different from $C(p,q^{-1},l,u,v)$. Hence to define a constrained knot, we need to fix the choice of $q$ in the set $\{\pm q_0^{\pm 1}\}$. For Cases (i)-(viii), the later proof shows $q=q_0$. However, for Cases (ix)-(x), it is hard to provide a general formula for the choice of $q$ since the proof is not constructive. 

Then we prove the theorem case by case.

For Case (i), we consider two subcases: (ia) $p_2=1,q_2>0,q_1< 0$; (ib) $p_2=1,q_2<0,q_1>0$. The proofs of these two subcases are similar so we only prove Case (ia). 

In Case (ia), $|q_2|=q_2,|q_1|=-q_1$. Consider curves $m_1,l_1,m_2,l_2$ in Figure \ref{magic3} and the Heegaard diagram $(\Sigma_2,\{\alpha_1,\alpha_2\},\beta_1)$ of $E(K_0)$ in Corollary \ref{coro2}. For example, if $q_2=3$, then $\alpha_2$ is obtained by resolving intersection points of $m_2$ and $3$ copies of $l_2$ positively. Let $l^\prime_1$ be the curve obtained by sliding $l_1$ over $\alpha_2$ along an arc $a$ around $z$; see the top left subfigure of Figure \ref{f12}. Let $\alpha_1^\prime$ be obtained by taking $|p_1|$ copies of $m_1$ and $|q_1|$ copies of $l^\prime_1$ and resolving negatively. Then $(\Sigma_2,\{\alpha_1^\prime,\alpha_2\},\beta_1)$ is also a Heegaard diagram of $E(K_0)$ since $l^\prime_1$ is isotopic to $l_1$ in the link complement. Consider the genus 1 surface $\Sigma_1$ from $\Sigma_2$ by removing the 1-handle attaching to $z$ and $w$. Then $(\Sigma_1,\alpha^\p_1,\beta_1,z,w)$ is a doubly-pointed Heegaard diagram of $K_0$.  We can compare this diagram with the standard diagram of a constrained knot.
\begin{figure}[ht]
\centering 
\includegraphics[width=0.998\textwidth]{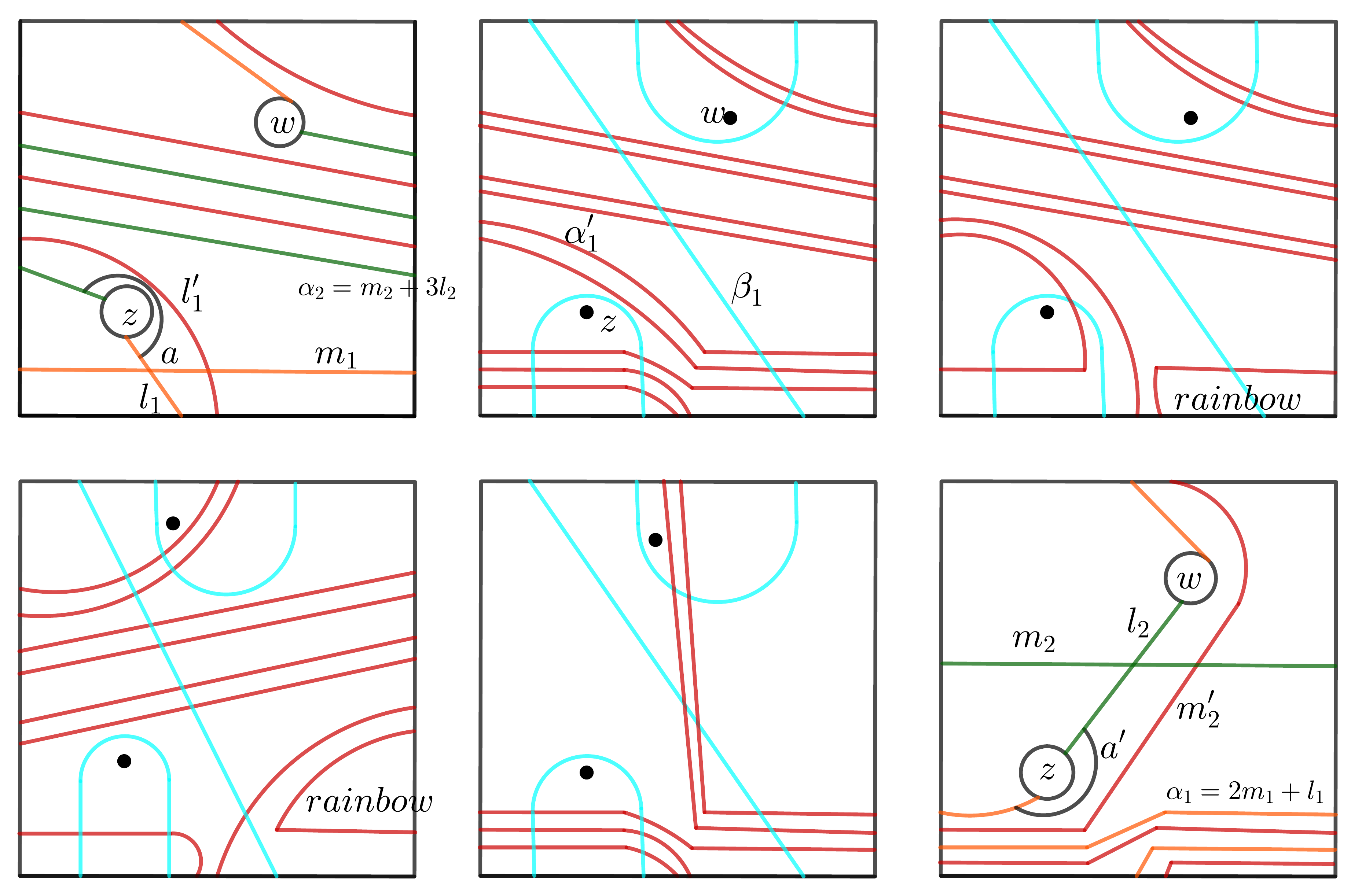} 
\caption{Examples of $K_0$.}
\label{f12}
\end{figure}

By construction, there are $q_2$ strands in $l_1^\p$ connecting the left edge to the right edge, where $(q_2-1)$ strands do not intersect the top edge and $1$ strand intersects the top edge. Since $m_1$ is a strand connecting the left edge to the right edge, there are $(p_1-q_1 q_2)$ strands in $\alpha_1^\p$ connecting the left edge to the right edge. These strands can be divided into two parts, in which the strands are isotopic to the strands $a_0$ and $a_{-1}$ defined before Lemma \ref{lamuv}, respectively. By counting the number of strands, we have$$|A_0(u_0,v_0)|=p_1\aand |A_{-1}(u_0,v_0)|=-q_1q_2.$$Hence $(\Sigma_1,\alpha^\p_1,\beta_1,z,w)$ is the same as the standard diagram (\textit{c.f.} Figure \ref{gr}) of $$C(p_1-q_1q_2,q_1,-q_1q_2+1,u_0,v_0).$$Thus, two knots are equivalent. For example, the top middle subfigure of Figure \ref{f12} corresponds to $C(9,-2,7,3,1)=C(9,7,7,3,1)$, where $$(p_1,q_1,p_2,q_2,u,v)=(3,-2,1,3,3,1).$$

Cases (ii)-(iv) are proven by the similar strategy. Indeed, we can compare $\alpha_1^\prime$ in the doubly-pointed diagram of $K_0$ with the standard diagram of a constrained knot to obtain the parameters. In particular, the type and the number of strands in $\alpha_1^\prime$ are important. So we only state the main difference about the curve $\alpha_1^\prime$.

For Case (ii),(ii$^\p$), let $\alpha_1^\prime$ be the curve as defined in Case (i). It is the union of strands with endpoints on the left edge and the right edge. By the assumption $q_1>p_1$, we may have rainbows in $\alpha_1^\prime$, \textit{i.e.}, strands whose endpoints are on the same edge. Since the rainbows on the right edge do not bound a basepoint, we can isotopy $\alpha_1^\p$ to remove them. After removing $p_1$ rainbows on the right edge, there are $(q_1q_2-2p_1)$ strands and $p_1$ strands isotopic to $a_{-1}$ and $a_{-2}$, respectively, \textit{i.e.}, $$|A_{-1}|=q_1q_2-2p_1\aand |A_{-2}|=p_1.$$The choice of Case (ii) and Case (ii$^\prime$) depends on if $U_{-1}\ge U_{-2}$ or $U_{-1}< U_{-2}$, respectively. This is because the parameter $u$ is the greater number in $\{U_{-1},U_{-2}\}$.

For Case (iii), the pair of sets $(A_{-1},A_{-2})$ in the above proof is replaced by $(A_{-1},A_*)$. Counting the number of strands, we have $$|A_{-1}|=q_1q_2-2p_1\aand  |A_*|=p_i.$$By Lemma \ref{lamuv}, the number $U_*$ is always greater than $U_{-1}$. 

For Case (iv), all strands are isotopic to $a_0$.

Examples can be found in Figure \ref{f12}. In all examples, $(u,v)=(3,1)$. In the top right subfigure, the diagram of $K_0$ is in Case (ii) with $(p_1,q_1,p_2,q_2)=(1,2,1,3)$, which corresponds to $C(-5,2,2,1,0)=C(5,3,2,1,0)$. In the bottom left subfigure, the diagram of $K_0$ is in Case (iii) with $(p_1,q_1,p_2,q_2)=(1,-2,1,-3)$, which corresponds to $C(-5,-2,5,5,2)=C(5,2,5,5,2)$. In the bottom middle subfigure, the diagram of $K_0$ is in Case (iii) with $(p_1,q_1,p_2,q_2)=(3,-2,1,0)$, which corresponds to $C(3,-2,1,3,1)=C(3,1,1,3,1)$.

For proofs of Cases (v)-(viii), we consider the curve $m^\prime_2$ obtained by sliding $m_2$ over $\alpha_1$ along an arc $a^\p$ around $z$; see the bottom right subfigure of Figure \ref{f12} for $p_1=2$. Now the resulting diagram of $E(K_0)$ is $(\Sigma_2,\{\alpha_1,\alpha_2^\prime\},\beta_1)$, where $\alpha^\prime$ is obtained from $m_2^\p$ and $l_2$ by resolution. The proofs are similar to Cases (i)-(iv).

For Cases (ix),(x), diagrams are more complicated. By Proposition \ref{char}, we can check by the distribution of the spin$^c$ structures of intersection points to obtain that the knot $K_0$ is a constrained knot. The parameter $l$ can be obtained by counting the number of strands.
\end{proof}
The following corollary is obtained by changing parameters in Table \ref{table3}.
\begin{corollary}\label{corol}
Suppose integers $p,q$ satisfying $p>q>0$ and $\operatorname{gcd}(p,q)=1$. The choices of $l$ from Theorem \ref{thm4} are in Table \ref{table4}. Note that Theorem \ref{magicpara} follows from the first two rows in Table \ref{table4}.
\end{corollary}
\begin{table}[ht]
\caption{Choices of the parameter $l$.\label{table4}}
\begin{tabular}[!htbp]{|p{6cm}|p{6cm}|}
\hline  
Case&$l-1$\\
\hline  
(i),(v) with $q_2>0$ and (iv),(viii),(ix)&$\pm nq$, where $nq\in[0,p)$\\
\hline
(i),(v) with $q_2<0$ and (iv),(viii),(x)&$\pm n(p-q)$, where $n(p-q)\in[0,p)$\\
\hline
(ii)&$\lceil p/q\rceil q-p$\\
\hline
(ii$^\prime$),(vi)&$2p-\lceil p/q\rceil q$\\
\hline
(iii),(vii) &$2p-\lceil p/(p-q)\rceil (p-q)$\\
\hline 
\end{tabular}
\end{table}

\begin{remark}For integers $u_0,v_0$ satisfying $(u_0,v_0)=(1,1)$ or $0<-2v_0<u_0,\gcd(u_0,v_0)=1$ and $u_0$ odd, the surgery description can be induced similarly to Table \ref{table3}. We omit the explicit description.
\end{remark}
We describe some special examples of Table \ref{table3} as follows.

Consider integers $u_0,v_0$ satisfying $(u_0,v_0)=(1,0)$ in Theorem \ref{thm4}. We know that the manifold $E(\mathfrak{L}(1,0))$ is diffeomorphic to $S^1\times F$, where $F$ is a disk with two holes. For integers $p_1,p_2,q_1,q_2$ satisfying $p_1p_2\neq q_1q_2$, the knot $K_0(1,0,p_1/q_1,p_2/q_2)$ is a torus knot in a lens space.

Cases (iii), (vii) in Table \ref{table3} cover the cases $(u,v)=(3,\pm 1)$. By Corollary \ref{corol}, for $p,q\in \mathbb{Z}$ with $p>q>0$, the knot $C(p,\pm q,2p-\lceil p/q\rceil q+1,3,\pm 1)$ is a torus knot.
\begin{theorem}\label{compo}
The knot $C(p,q,1,u,v)$ is the connected sum of the 2-bridge knot $\mathfrak{b}(u,v)$ and the core knot $C(p,q,1,1,0)$ of $L(p,q^\prime)$, where $qq^\prime\equiv 1 \pmod p$.
\end{theorem}
\begin{proof}
By Case (iv) in Theorem \ref{thm4}, the knot $C(p,q,1,u,v)$ is identified with $K_0(u,v,p/q,1/0)$, which is obtained by the $p/q$ surgery on the meridian of $\mathfrak{b}(u,v)$. By Corollary \ref{core}, the knot $C(p,q,1,1,0)$ is the core knot, which is obtained by the $p/q$ surgery on one component of the Hopf link.
\end{proof}
\section{1-bridge braid knots}\label{s8}
In this section, we describe another approach to construct constrained knots by Dehn surgeries. Many results are based on \cite[Section 3]{Greene2018}. The main objects in this section are 1-bridge braids defined below.
\begin{definition}
A knot in the solid torus $S^1 \times D^2$ is called a \textbf{1-bridge braid} if it is isotopic to a
union of two arcs $\gamma\cup \delta$ so that
$\gamma\subset \partial(S^1 \times D^2)$ is braided, \textit{i.e.} transverse to each meridian $\{\operatorname{pt}\}\times  \partial D^2$, and
$\delta$ is a bridge, \textit{i.e.} properly embedded in some meridional disk $\{\operatorname{pt}\}\times  D^2$.
\end{definition}
1-bridge braids are denoted by $B(w,b,t)$ \cite{Gabai1990}, where $w>0$ is the \textbf{winding number}, $b\in[0,w-2]$ is the \textbf{bridge width}, and $t\in[1,w-1]$ is the \textbf{twist number}. When $b=0$, the 1-bridge braid can be isotoped to lie on $\partial (S^1\times D^2)$. Let $B(w,w-1,t)$ denote $B(w,0,t+1)$.

As mentioned in \cite[Section 3]{Greene2018}, after isotopy, the arc $\gamma$ can be lifted to a straight line (a geodesic) in the universal cover $\mathbb{R}^2$ of $\partial (S^1 \times D^2)$, which is still denoted by $\gamma$. Suppose that $\gamma$ connects $(0,0)$ to $(t^\prime,w)$, where $t^\prime \in\mathbb{Q}\cap [t,t+1)$. Let $B(w,s(\gamma))$ denote this 1-bridge braid, where $s(\gamma)=t^\prime/w$ is called the \textbf{inverse slope} of $\gamma$. Suppose $s=n/d$ with $\operatorname{gcd}(n,d)=1$. Suppose that the integer $n_i\in[0,d)$ satisfies $n_i\equiv ni\pmod d$. Then $b$ is from the formula$$b=\#\{i\in [1,w-1]|n_i<n_w\}.$$
\begin{definition}\label{deff}
Suppose integers $p,q$ satisfying $0<q<p$ and $\operatorname{gcd}(p,q)=1$. Let $B(w,s(\gamma),p,q)$ denote the knot in $L(p,q)$ obtained by Dehn filling $(S^1\times D^2,B(w,s(\gamma)))$ along the curve on $\partial (S^1\times D^2)$ with slope $p/q$, which is called a \textbf{1-bridge braid knots}.
\end{definition}
\begin{proposition}
    For a 1-bridge braid $B(w,s(\gamma))$, Suppose $s$ represents the core of the solid torus, and suppose $t$ represents the meridian of the braid. For $j\in [1,w-1]$, define $$\theta_j=\begin{cases}1 & n_j<n_w,\\ 0 & n_j>n_w.\end{cases}$$Then the 2-variable Alexander polynomial of $B(w,s(\gamma))$ is \[\Delta(s,t)=\sum_{i=0}^{w-1}s^i t^{\sum_{j=1}^i \theta_j}.\]
\end{proposition}
\begin{proof}
Suppose $H_2=S^1\times D^2-N(\delta)$, which is diffeomorphic to a genus two handlebody. Let $D$ be the cancelling disk of $\delta$. There are two meridian disks $\{{\rm pt}\}\times \partial D^2$ and $D$ of $H_2$. Suppose their boundaries are $\alpha_1$ and $\alpha_2$, respectively. Suppose $\beta=\partial N(\gamma)$ and $\Sigma=\partial H_2$. Then $(\Sigma,\{\alpha_1,\alpha_2\},\beta)$ is a Heegaard diagram of $S^1\times D^2-N(B(w,s(\gamma)))$. It induces a presentation of the fundamental group by the method in Section \ref{s6}:\[\pi_1(S^1\times D^2-N(B(w,s(\gamma))))=\langle s,t|\omega t\omega^{-1}t^{-1}=1\rangle,\]where $\omega=st^{\theta_1}st^{\theta_2}s \cdots st^{\theta_{w-1}}s$. Then we can calculate the Alexander polynomial by Fox calculus \cite[Chapter II]{Turaev2002}.
\end{proof}
Let $F_n$ be the $n$-th Farey sequence, \textit{i.e.} the sequence containing all rational numbers $x/y$ with $0\le x\le y\le n$ and $\operatorname{gcd}(x,y)=1$, listed in the increasing order. For example, \begin{equation}\label{farey}
    F_1=(\frac{0}{1},\frac{1}{1}),F_2=(\frac{0}{1},\frac{1}{2},\frac{1}{1}),F_3=(\frac{0}{1},\frac{1}{3},\frac{1}{2},\frac{2}{3},\frac{1}{1}),F_4=(\frac{0}{1},\frac{1}{4},\frac{1}{3},\frac{1}{2},\frac{2}{3},\frac{3}{4},\frac{1}{1}).
\end{equation}

For a fixed integer $w$, suppose $f_-,f_+$ are successive terms in $F_{w-1}$. For any two 1-bridge braids with inverse slopes $s_2,s_2\in(f_-,f_+)$, there is an isotopy between them \cite[Section 3]{Greene2018}. If $s(\gamma)\in (f_-,f_+)$, the interval $\mathbb{S}(\gamma)=[f_-,f_+]$ is called the \textbf{simple interval} of $\gamma$. Two examples are shown in Figure \ref{ff}. 

For integers $w,t$ satisfying $\operatorname{gcd}(w,t)=1$, the 1-bridge braid knot $B(w,t/w,p,q)$ is the $(w,t)$ torus knot in $L(p,q)$ defined in Section \ref{s2}. Suppose $f_\pm=n_\pm/d_\pm$, where $n_\pm,d_\pm$ are integers satisfying $\operatorname{gcd}(n_\pm,d_\pm)=1$. If $d_\pm|w$, then the 1-bridge braid knot $B(w,s(\gamma),p,q)$ with $s(\gamma)\in(f_-,f_+)$ is the $(1,\mp w/d_\pm) $ cable knot of the $(d_\pm,nw/d_\pm)$ torus knot in $L(p,q)$, respectively (\textit{c.f.} \cite[Section 3.1]{Greene2018}). The braids $B(\omega,s(\ga))$ in the above two cases are calld \textbf{torus braids} and \textbf{cable braids}, respectively. In other cases, the braid $B(w,s(\gamma))$ is called a \textbf{strict braid}.
\begin{theorem}\label{simple1bb}
The 1-bridge braid knot $B(w,s(\gamma),p,q)$ is a simple knot if and only if $q/p\in \mathbb{S}(\gamma)$. In this case, it is the simple knot $S(p,q,wq)$.
\end{theorem}
\begin{proof}
The sufficient part follows from the discussion before \cite[Theorem 3.2]{Greene2018}. Indeed, the arc $\gamma$ can be isotoped to have the inverse slope $q/p$ (if $q/p=f_\pm$, then let the slope of $\gamma$ be $f_\pm\mp \epsilon$ for small $\epsilon>0$). Then the knot is the union of two arcs of slopes $0$ and $q/p$, respectively. Then it is straightforward to check that theknot is a simple knot. Note that the knot is homologous to $wq$ of the core of the filling solid torus. Thus, the knot is $S(p,q,wq)$.

The necessary part for a strict braid is shown by \cite[Theorem 3.2]{Greene2018}. When $B(w,s(\gamma))$ is not strict, the proof of \cite[Theorem 3.2]{Greene2018}  still applies because $d_\pm<w$.
\end{proof}
\begin{figure}[htbp]
\centering
\begin{minipage}[t]{0.28\textwidth}
\centering

\includegraphics[width=3.5cm]{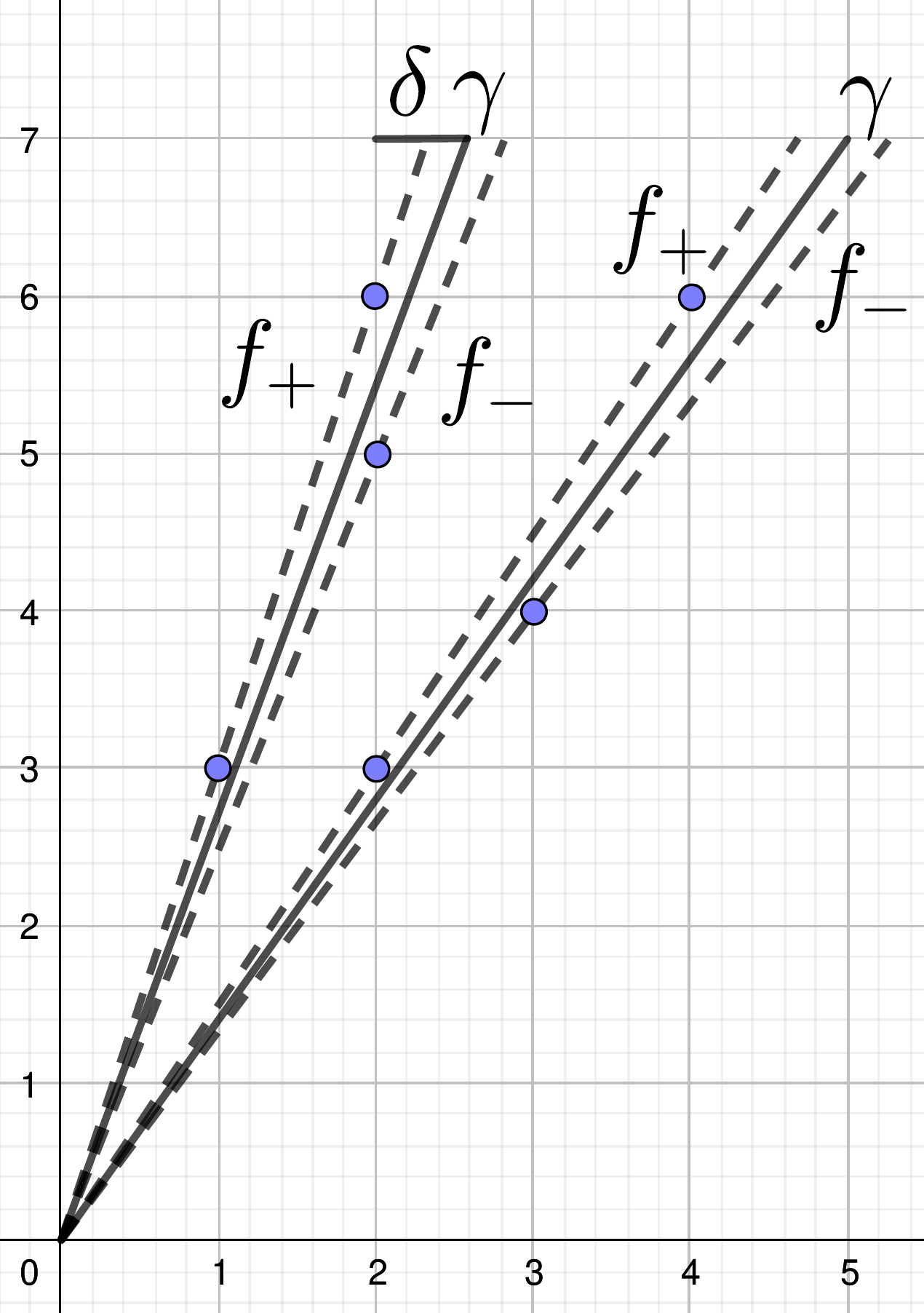}
\caption{1-bridge braid in $\mathbb{R}^2$.\label{ff}}
\end{minipage}
\begin{minipage}[t]{0.7\textwidth}
\centering

\includegraphics[width=9cm]{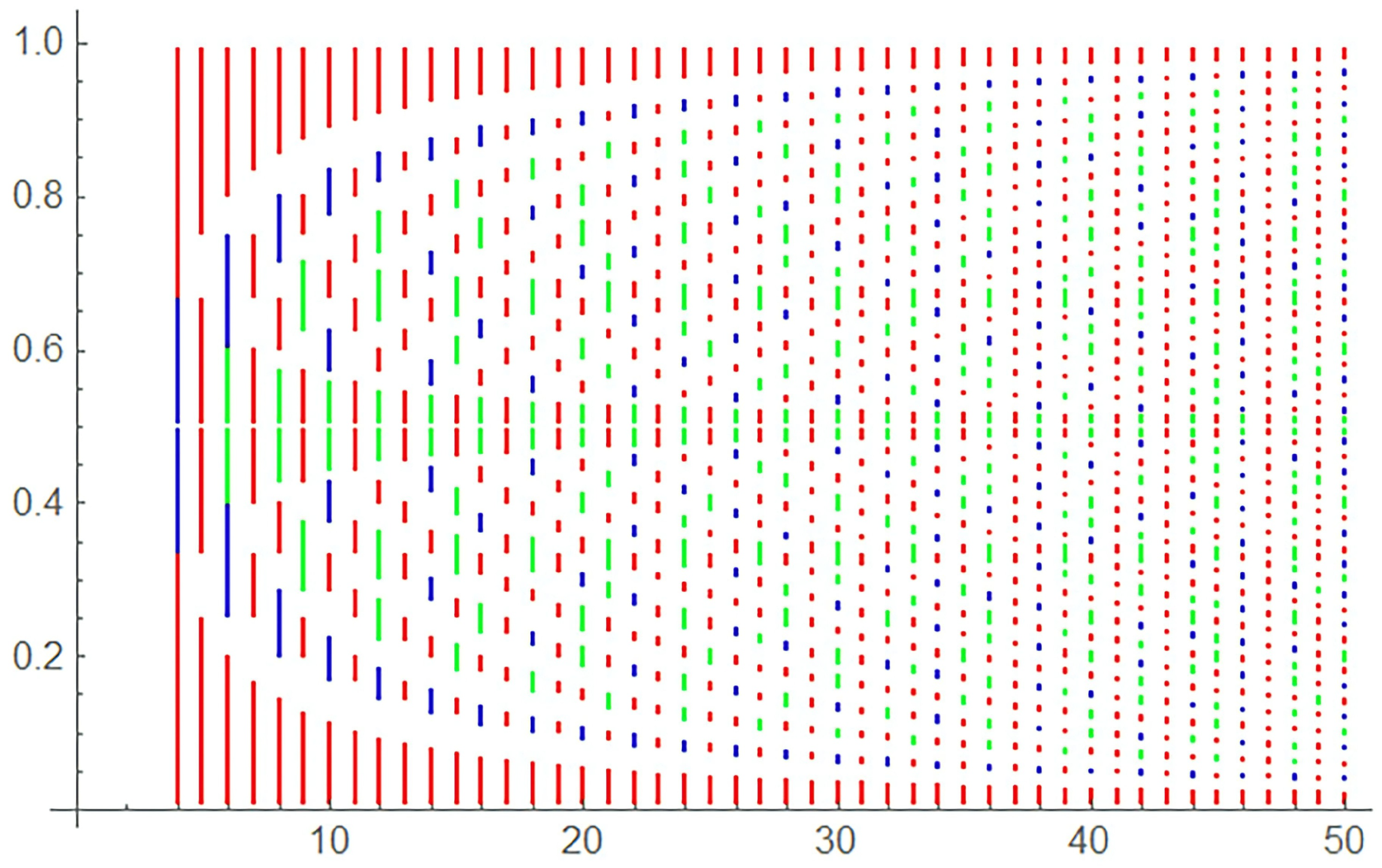}
\caption{Simple intervals.\label{si}}
\end{minipage}
\end{figure}
Let us consider special cases of simple knots obtained from Theorem \ref{simple1bb}. Consider examples of Farey sequences in (\ref{farey}). For $w\le 3$, all simple knots are from torus braids. For $w\le 4$, all simple knots are from either torus braids or cable braids. For $w\ge 4$, the union of the simple intervals for torus braids and cable braids are shown in Figure \ref{si}, where red arcs represent torus braids (they are Berge-Gabai knots of Type I, \textit{c.f.} \cite{Gabai1989,Berge1991b,Berge2018b}), blue arcs represent $(1,\pm 2)$ cable braids (they are Berge-Gabai knots of Type II), and green arcs represent other cable braids.
\begin{proof}[Proof of Theorem \ref{cla4}]
By Theorem \ref{simple1bb}, simple knots are 1-bridge braid knots. For constrained knots that are not simple knots, we show $C(p,q,l,u,1)$ are equivalent to 1-bridge braid knots. The case $C(p,q,l,u,-1)=C(p,q,l,u,u-1)$ is the mirror image of $C(p,-q,l,u,1)$ by Proposition \ref{mirr2} so is also equivalent to 1-bridge braid knots.

The proof is inspired by Figure \ref{f17n1}. Suppose $(T^2,\al_1,\be_1,z,w)$ is the standard diagram of $C(p,q,l,u,1)$. By definition, the constrained knot is the union of two arcs $a$ and $b$ connecting $z$ to $w$ in $T^2-\alpha_1$ and $T^2-\beta_1$, pushed slightly into the $\alpha_1$-handlebody and the $\beta_1$-handlebody, respectively. The arc $a$ can be chosen as a horizontal one, and there are infinitely many choices of isotopy classes of $b$ on $T^2$. Let $\gamma_i$ denote different choices of $b$ for $i\in\mathbb{Z}$. All choices induce equivalent knots because they are isotopic in the $\beta_1$-handlebody.

Since there is only one rainbow for $\be_1$, the arc $\ga_i$ does not have any rainbows. For a large integer $i$, the arc $\gamma_i$ can be isotoped to a straight line. Then $\gamma_i$ is transverse to each meridian disk of $\alpha_1$-handlebody and the union of $a$ and $\gamma_i$ is a 1-bridge braid in the $\alpha_1$-handlebody. Hence $C(p,q,l,u,1)$ is equivalent to a 1-bridge braid knot.
\end{proof}
\begin{figure}[htbp]
\centering
\begin{minipage}[t]{0.48\textwidth}
\centering

\includegraphics[width=5cm]{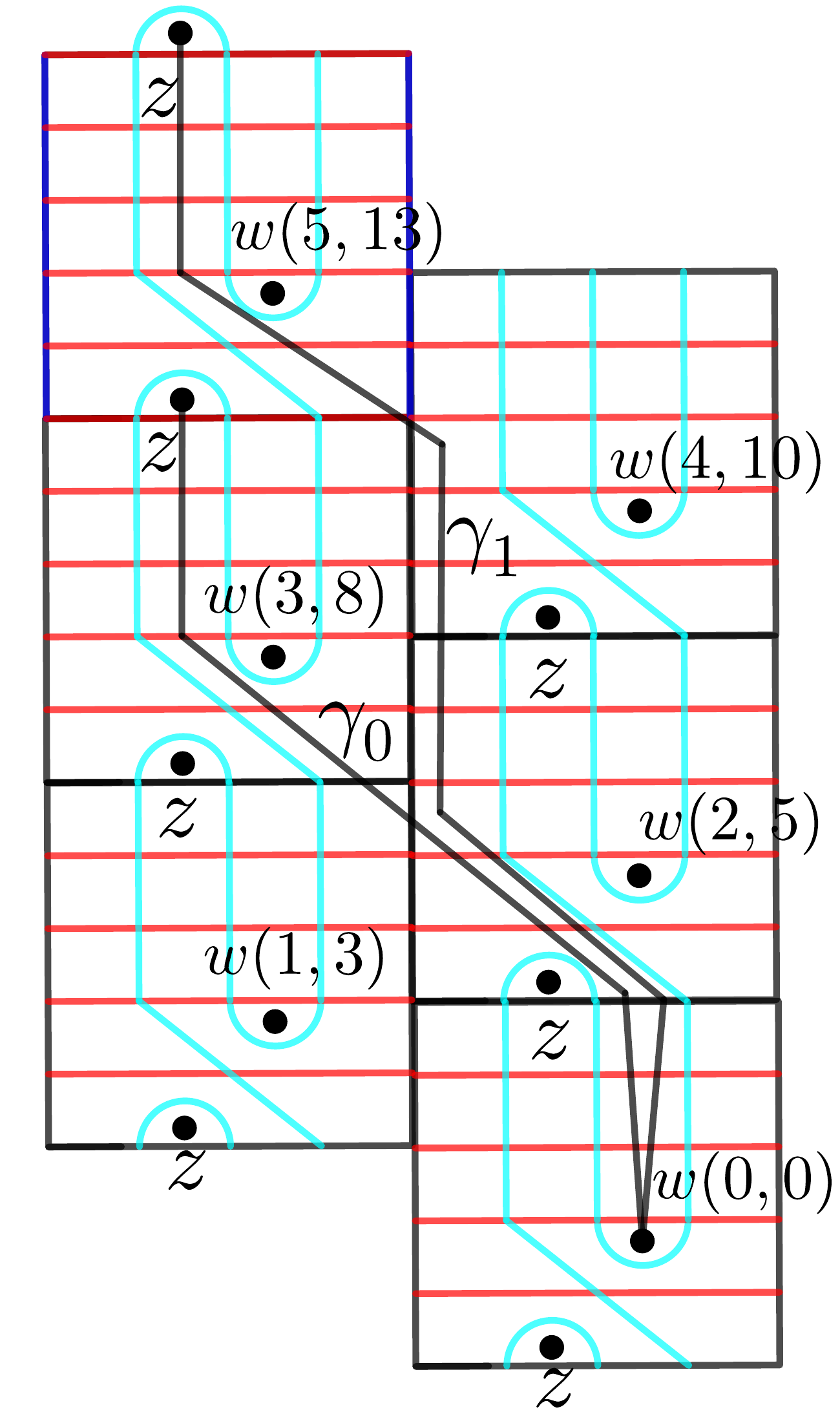}
\caption{Arcs $\gamma_i$ for $C(5,3,2,3,1)$.\label{f17n1}}
\end{minipage}
\begin{minipage}[t]{0.48\textwidth}
\centering

\includegraphics[width=6cm]{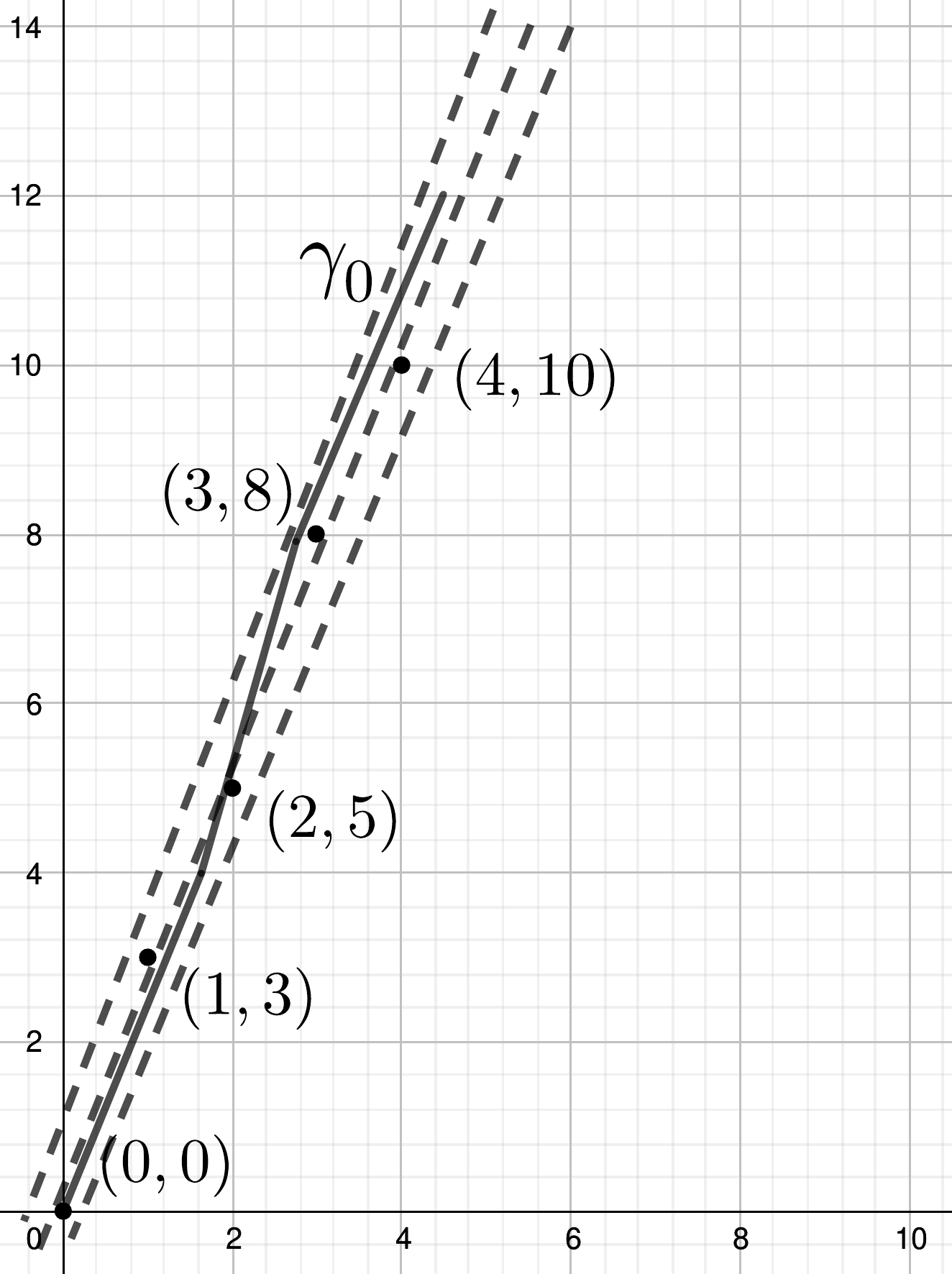}
\caption{Arc $\gamma_0$ for $C(5,3,2,3,1)$ in $\mathbb{R}^2$.\label{f17n2}}
\end{minipage}
\end{figure}
It is possible to find the explicit formula of $B(w(\gamma_i),s)$ in Theorem \ref{cla4} as follows. 

Suppose lifts of $w$ in the universal cover $\mathbb{R}^2$ of $T^2$ are lattice points $\mathbb{Z}^2$ as in Figure \ref{f17n2}. Then domains in Figure \ref{f17n1} lie in the narrow bands with dotted boundaries in Figure \ref{f17n2}. From the parameterization of the constrained knot, we know $C(p,q,l,u,1)$ is in $L(p,q^\p)$, where $qq^\p\equiv 1\pmod p$. Then the slope of the dotted boundaries is $p/q^\prime$. Indeed, these boundaries are $\be_0$ in the standard diagram $(T^2,\al_0,\be_0)$ of $L(p,q^\p)$.

Suppose $$\lambda=(qq^\prime-1)/p\aand  \epsilon=\begin{cases}0 & l+q\le p,\\ 1 &l+q> p.\end{cases}$$Suppose $\gamma_0$ is the first arc that can be straightened in the lift of $T^2-\beta_1$. Suppose $D_j$ for $j\in\mathbb{Z}/p\mathbb{Z}$ are regions in the new diagram $C$ mentioned in Section \ref{s3}. The part of $\gamma_i$ that lies in $\bigcup_{i=l+1}^p D_i$ and the disk bounded by the unique rainbow of $\be_1$ around a basepoint is called the part \textbf{in the generalized rainbow}. There are two parts of $\ga_i$ in generalized rainbows related to $z$ and $w$.

The parameter $w(\gamma_i)$ is the same as  $|\gamma_i\cap\alpha_1|$. Thus \[w(\gamma_i)=p(u-3)+2(p-l+1)+(q+l-1-p\epsilon)+pi=p(u-1-\epsilon+i)+q-l+1,\]where $p(u-3)+2(p-l+1)$ is from parts of $\ga_i$ in two generalized rainbows and $(q+l-1-p\epsilon)+pi$ is from the remain part. Any lift of $w$ in the left band in Figure \ref{f17n2} has the coordinate $$(\lambda+nq^\prime,q+np)\text{ for some }n\in\mathbb{Z}.$$The closest lift of $w$ near $\gamma_i$ other than $(0,0)$ has the coordinate $$(\lambda+n_0q^\prime,q+n_0p)\text{, where }n_0=(u-1)/2-\epsilon+i.$$It lies at a lift of the region $D_l$ that intersects the endpoint of the part of $\ga_i$ in the generalized rainbow related to $z$. Thus, the inverse slope of $\gamma_i$ is $$\frac{\lambda+n_0q^\prime}{q+n_0p}-r$$for a small rational number $r$.

In practice, for given parameters $(p,q,l,u,1)$, it is possible to determine if a constrained knot $C(p,q,l,u,1)$ is from a torus braid or a cable braid. For example, consider $(l,u,v)=(2p-\lceil p/q\rceil q+1,3,1)$ and $i=0$, then $$\epsilon=1,n_0=0\aand \omega=(1+\lceil p/q\rceil )q-p.$$The inverse slope is $\lambda /q-r$. Suppose $x=(1+\lceil p/q\rceil )\lambda-q^\prime$. Since $$\frac{\lambda}{q}=\frac{x+q^\prime}{w+p},$$the rational number $x/w$ is in the simple interval $\mathbb{S}(\gamma_0)$, \textit{i.e.} $\gamma_0$ is isotopic to the arc with inverse slope $x/w$. Thus $C(p,q,2p-\lceil p/q\rceil q+1,3,1)$ is a torus knot. This is
consistent with the example from the magic link $\mathfrak{L}(1,0)$ mentioned in Section \ref{s7}.
\section{SnapPy manifolds}\label{s9}

A compact orientable manifold $M$ with torus boundary is called a \textbf{(hyperbolic) 1-cusped manifold} if the interior of $M$ admits a hyperbolic metric of finite volume. All 1-cusped manifolds that have ideal triangulations with at most 9 ideal tetrahedra are  included in \textit{SnapPy} \cite{snappy}. They are called \textbf{SnapPy manifolds}. In this section, we explain the strategy to study the relation between 1-cusped manifolds and constrained knots by computer program. Codes and results can be found in \cite{Ye}.

Suppose $M$ is a 1-cusped manifold and suppose $\gamma$ is a simple closed curve on $\partial M$. The pair $(M,\gamma)$ is called an \textbf{exceptional filling} if Dehn filling along $\gamma$ gives a nonhyperbolic manifold $M(\gamma)$. For such $(M,\gamma)$, the core of the filling solid torus induces a knot in $M(\gamma)$. The induced knot $K(M,\gamma)$ is called a \textbf{SnapPy knot} if $M$ is a SnapPy manifold. Dunfield provided a census of exceptional fillings for SnapPy manifolds \cite{Dunfield2018}. In this census, there are 44487 exceptional fillings $(M,\gamma)$, covering 38056 different SnapPy manifolds, for which $M(\gamma)$ is a lens space.  

Suppose $M(\gamma)=L(p,q)$ and $m$ is the meridian of $K=K(M,\gamma)$. If $H_1(M;\mathbb{Z})\cong \mathbb{Z}$ and it is generated by $t$, then $\tau(K)=\Delta_{K}(t)/(1-t)$ \cite{Turaev2002}. The Alexander polynomial only depends on $M$ and can be found in \textit{SnapPy}. The Euler characteristic $\chi(\widehat{HFK}(M(\ga),K))$ can be calculated by Lemma \ref{lam9}. Suppose it is $\sum a_it^i$. Since $$H_1(M;\mathbb{Z})/([m])\cong H_1(M(\ga);\mathbb{Z})\cong \mathbb{Z}/p\mathbb{Z},$$we know $[m]=t^p$. Then $\chi(\widehat{HFK}(M(\ga),K))$ can be decomposed into $p$ polynomials $$\sum_{i\equiv i_0\pmod p}a_it^i \text{ for }i_0\in[0,p).$$Suppose $$F_{i_0}(t)=\sum_{i\equiv i_0\pmod p}a_it^{(i-i_0)/p}$$and let $f_i(t)$ be images of $F_i(t)$ in $\mathbb{Z}[t]/\pm(t)$. The exceptional filling $(M,\gamma)$ has \textbf{$n$ form(s)} if the set $\{f_i(t)|i\in[0,p)\}$ has $n$ elements.

If $F_{i}(t)$ is a monomial for any $i$, then $(M,\gamma)$ has 1 form. By Theorem \ref{simple}, the Euler characteristic must be the same as the simple knot in the same homology class. Such $(M,\gamma)$ is called a \textbf{simple filling}. It does not necessarily induce a simple knot since Conjecture \ref{conjb} has not been proven yet.

For $l=1$, the constrained knot $C(p,q,l,u,v)$ is not hyperbolic since it is satellite by Theorem \ref{compo}. If $F_{i}(t)$ is symmetric, coefficients of $F_{i}(t)$ are alternating for any $i$, and $(M,\gamma)$ has 2 forms, then $K$ might be a constrained knot $C(p,q^\prime,l,u,v)$, where $l>1,u>1$ and $q^\prime\equiv \pm q^{\pm1}\pmod p$. As in the proof of the necessary part of Theorem \ref{thm2}, a tuple of \textbf{virtual parameters} $(l,u,v)$ can be calculated by $F_i(-1)$. Conversely, given $(p,q^\prime,l,u,v)$, the characteristic of the corresponding constrained knot is given by Theorem \ref{thm1}. If $\chi(\widehat{HFK}(K))$ is equivalent to $\chi(\widehat{HFK}(C(p,q^\prime,l,u,v))$ as elements in $\mathbb{Z}[H_1(M;\mathbb{Z})]$ for virtual parameters $(l,u,v)$, then $(M,\gamma)$ is called a \textbf{constrained filling}. If symmetrized Alexander polynomials of $K$ and $C(p,q^\prime,l,u,v)$ are the same, then $(M,\gamma)$ is called a \textbf{general constrained filling}. If $H_1(M;\mathbb{Z})\cong \mathbb{Z}$, then $(M,\gamma)$ is a constrained filling if and only it is a general constrained filling.

If $\operatorname{Tors}H_1(M;\mathbb{Z})$ is nontrivial, then the Turaev torsion $\tau(M)$ can be calculated by a presentation of $\pi_1(M)$. \textit{SnapPy} provides a presentation of $\pi_1(M)$ and the related words of the preferred meridian and the preferred longtide (they are not necessarily the same as the meridian and the longitude mentioned in Section \ref{s2}). By the filling slope from Dunfield's census, the homology class $[m]\in H_1(M;\mathbb{Z})$ is obtained. The algorithm described above also works and definitions also apply to this case.

The codes in \cite{Ye} construct complements of constrained knots in \textit{SnapPy} by functions in the Twister package. Then the function \textit{M.identify()} in \textit{SnapPy} tells us if the manifold with a constrained filling is indeed the complement of a constrained knot. Mirror manifolds are not distinguished here.

In Dunfield's census, there are 16355 simple fillings and 8537 constrained fillings, covering 15262 and 8508 SnapPy manifolds, respectively. All 15262 and 8421 of 8508  SnapPy manifolds are complements of simple knots and constrained knots, respectively. There are 1838 manifolds that are both complements of simple knots and constrained knots with $u>1$. Thus, there are 21845 SnapPy manifolds that are complements of constrained knots in lens spaces. Other than these manifolds, there are 77 SnapPy manifolds that are complements of 2-bridge knots, which are special cases of constrained knots.

The choice of the slope in a constrained filling is subtle. For example, suppose $M=m003$ and $\gamma_1=(-1,1),\gamma_2=(0,1)$ in the basis from \textit{SnapPy}. Then both $M(\gamma_1)$ and $M(\gamma_2)$ are diffeomorphic to $L(5,4)$ and $M$ is the complement of $C(5,4,5,3,1)$. Indeed, there is an isometry of $M$ sending $\gamma_1$ to $\gamma_2$. Both $M(\gamma_1)$ and $M(\gamma_2)$ induce the same knot $C(5,4,5,3,1)$. All 9 pairs of slopes in Dunfield's census with this subtlety are from isometries, except the case $M=m172,\gamma_1=(0,1),\gamma_2=(1,1)$. Manifolds $M(\gamma_1)$ and $M(\gamma_2)$ are oppositely oriented copies of the same lens space. The first slope induces $S(49,18,7)$ and the second induces $S(49,18,21)$ (up to the mirror image), which are not equivalent. This example is interesting in the study of cosmetic surgery \cite{Bleiler1999}. In a word, the SnapPy knots induced by 15262+8421=23683 constrained fillings in the above discussion are all constrained knots.

There are 87 SnapPy manifolds with constrained fillings but not complements of constrained knots. For such a manifold, either the constrained knot with corresponding virtual parameters is not hyperbolic, or there is another SnapPy manifold which is the complement of the constrained knot with the same parameters. For example, the manifold $m390$ has a constrained filling $(1,0)$ with virtual parameters $(7,4,7,5,2)$, while $E(C(7,4,7,5,2))$ is diffeomorphic to $s090$.

If $\operatorname{Tors}H_1(M;\mathbb{Z})$ is nontrivial, then there are 54 general constrained fillings that are not constrained fillings. For example, manifolds $M_1=m400$ and $M_2=m141$ satisfy $$|\operatorname{Tors}H_1(M_i;\mathbb{Z})|=2\aand\Delta_{M_i}(t)=t^5 - t^4 + t^2 + t^{-2} - t^{-4} + t^{-5}\text{ for }i=1,2$$and $M_1$(1,1)$\cong M_2(-1,1)\cong L(18,13)$. Both manifolds have general constrained fillings and $M_2\cong E(C(18,3,18,3,1))$. Calculation shows $(M_1,(1,1))$ is not a constrained filling, \textit{i.e.} the Euler characteristic of the induced knot is different from that of $C(18,3,18,3,1)$.

For exceptional manifolds in Proposition \ref{mfds}, manifolds $m206$ and $m370$ have exceptional fillings with 2 forms and have virtual parameters $(l,u,v)=(5,5,2),(8,5,2)$, respectively. Unfortunately, both exceptional fillings are not even general constrained fillings. The manifold $m390$ is discussed above. For other 5 manifolds, there is no lens space filling (even $S^1\times S^2$ filling). It is harder to obtain information in Heegaard Floer theory.

In the rest of this section, we discuss the ways to obtain the genus and the fibreness of a knot. The genera and fibreness of Snappy knots can also be found in \cite{Ye}.
\begin{definition}\label{gedef}
Suppose $K$ is a knot in $Y=L(p,q)$ and suppose $H_1(E(K);\mathbb{Z})\cong \mathbb{Z}\oplus \mathbb{Z}/d\mathbb{Z}\cong \langle t,r\rangle (dr)$. By the excision theorem, Poinc\'{a}re duality and the universal coefficient theorem, we have\[H_2(Y,K;\mathbb{Z})\cong H_2(E(K),\partial E(K);\mathbb{Z})\cong H^1(E(K);\mathbb{Z})\cong \operatorname{Hom}(H_1(E(K);\mathbb{Z}),\mathbb{Z})\cong \mathbb{Z}.\]

Suppose $S$ is a connected, oriented and proper embedded surface representing the generator of $H_2(E(K),\partial E(K);\mathbb{Z})$. It is called a \textbf{Seifert surface} of $K$. Let the \textbf{genus} $g(K)$ and the \textbf{Thurston norm} $x([S])$ be the minimal values of $g(S)$ and $-\chi(S)$ among all Seifert surfaces, respectively.
\end{definition}
\begin{definition}
For a homogeneous element $x$ of $\widehat{HFK}(Y,K)$, suppose $\operatorname{gr}(x)=at+br\in H_1(E(K);\mathbb{Z})$. Let $\operatorname{gr}_0(x)$ be the number $a$. The width of $\widehat{HFK}(Y,K)$ is the maximal value of $|\operatorname{gr}_0(x)-\operatorname{gr}_0(y)|$ among all pairs of homogeneous elements $(x,y)$. Suppose homogeneous elements $x_0,y_0$ satisfy $$\operatorname{width}\widehat{HFK}(Y,K)=|\operatorname{gr}_0(x_0)-\operatorname{gr}_0(y_0)|.$$Suppose $H(x_0)$ is the subgroup of $\widehat{HFK}(Y,K)$ generated by homogeneous elements $x$ satisfying $\operatorname{gr}_0(x)=\operatorname{gr}_0(x_0)$. The \textbf{top rank} of $\widehat{HFK}(Y,K)$ is $\dim_\mathbb{Q}H(x_0)\otimes \mathbb{Q}$.
\end{definition}
\begin{theorem}[\cite{Ni2009,Juhasz2008}]
Consider $Y,K,S$ in Definition \ref{gedef} such that $E(K)$ is irreducible. Suppose $m$ is the meridian of $K$. Then the width of $\widehat{HFK}(Y,K)$ equals to $x([S])+|[m]\cdot [\partial S]|$, where $[m]\cdot [\partial S]$ is the algebraic intersection number on $\partial E(K)$.
\end{theorem}
\begin{proposition}
Consider $Y,K,S$ in Definition \ref{gedef}. Suppose $E(K)$ is irreducible. Suppose $(m,l)$ is the regular basis of $K$. Let $n$ be the minimal number of boundary components of a Seifert surface. Then $|[m]\cdot [\partial S]|=p/d$ and $n=\operatorname{gcd}(d,p/d)$. Thus, \[x([S])=\operatorname{width}(\widehat{HFK}(Y,K))-p/d\aand g(K)=1+\frac{x([S])-\operatorname{gcd}(d,p/d)}{2}.\]
\end{proposition}
\begin{proof}
Suppose $[K]=k[b]$, where $[b]$ is a generator of $H_1(Y;\mathbb{Z})$. Since $d=\operatorname{gcd}(p,k)$, the order of $[K]$ in $H_1(Y;\mathbb{Z})$ is $p/d$. By Poinc\'{a}re duality and the universal coefficient theorem, we have \[H_2(E(K);\mathbb{Z})\cong H^1(E(K),\partial E(K);\mathbb{Z})\cong \operatorname{Hom}(H_1(E(K),\partial E(K)),\mathbb{Z})=0.\]By the long exact sequence from $(E(K),\partial E(K))$, the boundary map $$\partial_*:H_2(E(K),\partial E(K);\mathbb{Z})\to H_1(\partial E(K);\mathbb{Z})$$is injective and the image of $\partial _*$ is the same as the kernel of the map $$i_*:H_1(\partial E(K);\mathbb{Z})\to H_1(E(K);\mathbb{Z}).$$Since $H_1(E(K);\mathbb{Z})/([m])\cong H_1(Y;\mathbb{Z})$, we have $[\partial S]=\pm (x[m]+p/d[l])$ for some $x\in \mathbb{Z}$. Then $|[m]\cdot [\partial S]|=p/d$ and $n=\operatorname{gcd}(x,p/d)$.

Let $[m],[l]$ also denote their images in $H_1(E(K);\mathbb{Z})$. By Lemma \ref{lam6}, we have $$[m]=\pm (p/d)t+ar\aand \operatorname{gcd}(p/d,d,a)=1.$$Suppose $[l]=yt+zr$ for some $y,z\in \mathbb{Z}$. Since $[\partial S]\in \operatorname{Ker}(i_*)$, the number $xa+(p/d)z$ is divisible by $d$. Suppose $n_0=\operatorname{gcd}(d,p/d)$. Then $\operatorname{gcd}(n_0,a)=1$ and $n_0|xa+(p/d)z$. Thus $n_0|x$ and $n_0|n$. Suppose $l^*$ is the homological longitude. Then $n[l^*]=[\partial S]$ and the image of $[l^*]$ in $H_1(E(K);\mathbb{Z})$ is $wr$ for some $w\in \mathbb{Z}$. Thus $n|d$ and $n|n_0$. This induces $n=n_0$.
\end{proof}
\begin{theorem}[\cite{Ni2007,Juhasz2008}]
Consider $Y,K,S$ in Definition \ref{gedef} such that $E(K)$ is irreducible. If the top rank of $\widehat{HFK}(Y,K)$ is 1, then $K$ is fibred with the fiber $S$.
\end{theorem}
\begin{proof}
Suppose $Y(S)$ is the balanced sutured manifold $(N,\nu)$, where $N=Y- \operatorname{Int}(S\times I)$ and $\nu=\partial S\times I$. Lemma 3.9 and the proof of \cite[Theorem 1.5]{Juhasz2008} imply that the rank of $SFH(Y(S))$ is the same as the top rank of $\widehat{HFK}(Y,K)$. Then $Y(S)$ is a product sutured manifold by \cite[Theorem 9.7]{Juhasz2008}, which implies $K$ is fibred with fiber $S$.
\end{proof}

%
%
%
\bibliographystyle{abbrv}



\end{document}